\documentclass[10pt,a4]{article}
\usepackage{latexsym}
\usepackage{amsmath}
\usepackage{amssymb}
\usepackage{amsthm}
\usepackage{amscd}
\usepackage{mathrsfs}
\usepackage[all]{xy}

\newtheorem{theorem}{Theorem}[section]
\newtheorem{lemma}[theorem]{Lemma}
\newtheorem{corollary}[theorem]{Corollary}
\newtheorem{proposition}[theorem]{Proposition}

\theoremstyle{definition}

\newtheorem{remark}[theorem]{Remark}
\newtheorem{definition}[theorem]{Definition}

\theoremstyle{remark}
\newtheorem{notation}{Notation}

\makeatletter

\@addtoreset{equation}{section}
\makeatother

\renewcommand{\eqref}[1]{(\ref{#1})}

\renewcommand{\bigskip}{\vspace{0.2cm}}

\begin{document}

\title{{\Large \textbf{On the 
cuspidal representations of 
${\rm GL}_2(F)$
of level $1$ or $1/2$ 
in the cohomology of 
 the Lubin-Tate 
 space $\mathcal{X}(\pi^2)$}}}

\maketitle

\begin{center}
\textbf{ Tetsushi Ito, 
Yoichi Mieda and Takahiro Tsushima}\\
\end{center}

\maketitle
\begin{abstract}
In \cite{T2}, we explicitly
compute the stable reduction of the Lubin-Tate space
$\mathcal{X}(\pi^2)$, in the 
equal characteristic case.
This paper is a continuation of \cite{T2}.
In this paper, we compute 
defining equations of irreducible 
components which 
appear in the Lubin-Tate space
$\mathcal{X}(\pi^2)$ in the mixed 
characteristic case.
We also determine the action of ${\rm GL}_2$, 
the action of the central 
division algebra of invariant $1/2$
 and the inertia action
 on the components in 
 the stable reduction of $\mathcal{X}(\pi^2)$.
As a result, in a sense, we observe
that the local Jacquet-Langlands correspondence 
and the local Langlands correspondence 
for cuspidal representations
of ${\rm GL}_2(F)$ of level $1$ or $1/2$
are realized in the cohomology of the Lubin-Tate space
$\mathcal{X}(\pi^2)$.
\end{abstract}

\section{Introduction}
Let $F$ be a non-archimedean local field
 with ring of integers $\mathcal{O}_F$, 
  uniformizer $\pi$
 and residue field $k$ 
 of characteristic $p>0.$
 Let $|k|=q.$ 
 We fix an algebraic 
 closure $F^{\rm ac}$ of $F$. 
 Its completion is denoted 
 by $\mathbf{C}.$
 We write $k^{\rm ac}$ 
 for the residue field of $\mathbf{C}.$
 Let $v(\cdot)$  denote the valuation of 
 $\mathbf{C}$ normalized such that $v(\pi)=1.$
 Let $F^{\rm nr}$ denote 
 the maximal unramified extension of $F$
 in $F^{\rm ac}$ and $F_0$ 
 denote its completion.
Let $\Sigma$ denote the unique, up to isomorphism, 
one-dimensional 
formal $\mathcal{O}_F$-module
 of height $h$ over
$k^{\rm ac}.$ 
Then, the {\it Lubin-Tate space} 
$\mathcal{X}(\pi^n)/F_0$ 
of level $n$ and height $h$
means a generic 
fiber of a formal scheme, 
which is a deformation 
space of $\Sigma$ equipped with
Drinfeld level $\pi^n$-structure. 
Then, the space 
$\mathcal{X}(\pi^n)$ 
is a rigid 
analytic variety 
of dimension $h-1$ 
over 
$F_0.$
In the remainder 
of this introduction, 
we assume $p \neq 2$ and $h=2.$ 
The existence of 
the stable model 
of a curve with genus greater than $1$
is guaranteed by the work 
of Deligne-Mumford in \cite{DM}.
In \cite{W3}, Jared Weinstein determines 
irreducible components
in the stable reduction of $\mathcal{X}(\pi^n)$
up to purely inseparable map, 
when $F=\mathbb{Q}_p$ and 
${\rm char}\ F=p>0$, by using the 
Deligne-Carayol theorem 
for ${\rm GL}_2$ 
and the Bushnell-Kutzko type theory.
See also a series of papers \cite{W}-\cite{W4}.

In this paper, we compute 
irreducible components 
in the stable reduction
of $\mathcal{X}(\pi^2)$ 
explicitly for any $F$, by using blow-up.
The stable reduction of $\mathcal{X}(\pi^2)$
has actions of 
$G_2^F:={\rm GL}_2(\mathcal{O}_F/\pi^2),$
the central division algebra 
$\mathcal{O}^{\times}_D$
of invariant $1/2$ 
and the inertia group $I_F.$
Then, we analyze \'{e}tale cohomologies
of the components in the stable reduction
 of $\mathcal{X}(\pi^2)$ as a 
 $G_2^F \times 
 \mathcal{O}^{\times}_D \times 
 I_F$-representation.
 As a result, in a sense, we observe
that the local Jacquet-Langlands correspondence 
and the local Langlands correspondence 
for cuspidal representations
of ${\rm GL}_2(F)$ of level $1$ or $1/2$
are realized in the cohomology of the Lubin-Tate space
$\mathcal{X}(\pi^2)$.
These are done without depending on
Deligne-Carayol's theorem for ${\rm GL}_2.$
Rather, to check Deligne-Carayol's theorem 
for ${\rm GL}_2$ purely locally 
in a lower level is our aim 
in this paper.
In \cite{T2}, in a case 
${\rm char}\ F=p>0,$ 
we completely 
compute the stable 
reduction of 
$\mathcal{X}(\pi^2)$
explicitly.
This paper is a sequence of loc.\ cit.
Similar computations of 
the stable models 
are found in 
\cite{CM}, \cite{CW}, 
\cite{E}, \cite{E2} and \cite{T}. 
In particular, in \cite{CM}, 
very significant 
and stimulus developments 
 are done
  to compute the stable model
of a curve.  
See also \cite{DR} and \cite{KM}.
We write $\overline{\mathcal{X}(\pi^2)}$
for the stable reduction of 
$\mathcal{X}(\pi^2).$

We define 
several subspaces $\mathbf{Y}_{3,1, \ast }, 
\mathbf{Y}_{2,2}$
and $\mathbf{Z}_{1,1, \ast }$ 
with $\ast \in 
\mathbb{P}^1(\mathbb{F}_q)$
 of $\mathcal{X}(\pi^2)$, 
 and compute their reduction
 in the mixed characteristic case, 
 in section \ref{com3}. 
We identify $\mathcal{X}(1)$ with 
an open unit ball $B(1) \ni u$ appropriately.
Then, under 
the natural 
projection 
$\mathcal{X}(\pi^2) 
\to \mathcal{X}(1) 
\simeq B(1),$ 
the subspaces 
$\bigcup_{\ast \in 
\mathbb{P}^1(\mathbb{F}_q)}\mathbf{Y}_{3,1, \ast }, 
\mathbf{Y}_{2,2}$
and $\bigcup_{\ast  \in \mathbb{P}^1(\mathbb{F}_q)}
\mathbf{Z}_{1,1 \ast }$ in $\mathcal{X}(\pi^2)$
are over the loci $v(u)=1/(q+1)$, $v(u) \geq q/(q+1)$
and $v(u)=1/2$ 
respectively.
 To compute their reduction, we choose 
 some simple model of the universal formal 
 $\mathcal{O}_F$-module over 
 ${\rm Spf}\ 
 \mathcal{O}_{F_0}[[u]],$ 
 which is 
  given in \cite{GH}.
 See subsection \ref{gh1} 
 for more details on this model.
It is well-known that the set 
$\pi_0(\mathcal{X}(\pi^n)_{\mathbf{C}})$
of geometrically connected components 
of $\mathcal{X}(\pi^n)$
is identified with $(\mathcal{O}_F/\pi^n)^{\times}$.
Furthermore, all geometrically connected components
are defined over the classical Lubin-Tate extension 
$F_n/F_0.$ 
See subsection 
\ref{geo} for 
the definition 
of $F_n.$
Now, we fix an identification
$\pi_0(\mathcal{X}(\pi^n)) \simeq 
(\mathcal{O}_F/\pi^n)^{\times}$ appropriately. See 
Theorem \ref{caq} for the identification.
We choose one connected component 
$\mathcal{X}^{i}(\pi^2) \subset \mathcal{X}(\pi^2) 
\times_{F_0}F_2$
with $i \in (\mathcal{O}_F/\pi^2)^{\times}.$
Let $\bar{i}$ be the image of $i$
by the canonical map
$(\mathcal{O}_F/\pi^2)^{\times} \to k^{\times}.$
For a subspace 
$W \subset \mathcal{X}(\pi^2)$, 
we write $W^{i}$ for the intersection
$(W\times_{F_0}F_2) 
\cap \mathcal{X}^{i}(\pi^2)$.
For an affine curve $X$ over $k^{\rm ac}$,
we write $X^c$ for the smooth compactification of 
the normalization of $X.$
The genus of $X$ means the genus of $X^c.$
Then, for
 each $\ast \in \mathbb{P}^1(\mathbb{F}_q),$
 the reduction $\overline{\mathbf{Y}}^{i}_{3,1,\ast }$
 is defined by $x^qy-xy^q=1$ with genus $q(q-1)/2$ and 
 $\overline{\mathbf{Y}}^{i,c}_{3,1,\ast}$ 
 appears in  
 $\overline{\mathcal{X}(\pi^2)}.$
 On the other hand, for
 each $\ast \in \mathbb{P}^1(\mathbb{F}_q),$
 the reduction $\overline{\mathbf{Z}}^{i}_{1,1,\ast }$
 is defined by $Z^q=X^{q^2-1}+X^{-(q^2-1)}$ with genus $0$.
 This curve has singularities at $X \in \mu_{2(q^2-1)}.$
 Then, by analyzing these singular points, 
 we find $2(q^2-1)$ components 
 $\{W^{i,c}_{\ast,\zeta}\}_{\zeta \in \mu_{2(q^2-1)}}$ 
 having an affine model $a^q-a=s^2$ 
 with genus $(q-1)/2.$
 Then, the components 
 $\{Z_{\ast}:=
 \overline{\mathbf{Z}}^{i,c}_{1,1,\ast }\}_{\ast \in 
 \mathbb{P}^1(\mathbb{F}_q)}$ appear
 in $\overline{\mathcal{X}(\pi^2)}.$
 In $\overline{\mathcal{X}(\pi^2)},$ the components
 $\{W^{i,c}_{\ast,\zeta}\}_{\zeta \in \mu_{2(q^2-1)}}$
 attach to $Z_{\ast }$ at distinct 
 $2(q^2-1)$ points of $Z_{\ast}$ over 
 the singular points of 
 $\overline{\mathbf{Z}}^{i}_{1,1,\ast }.$
Furthermore, 
the reduction $\overline{\mathbf{Y}}^{i}_{2,2}$
 is defined by the following equations
 \[
 x^qy-xy^q=\bar{i},\ Z^q=x^{q^3}y-xy^{q^3}.
 \]
This affine curve has singularities 
at 
\[
\mathcal{S}^{i}_{00}:=\{
(x_0,y_0,Z_0) \in \overline
{\mathbf{Y}}^i_{2,2}(k^{\rm ac})\ |\  
x_0^{q^2-1}=y_0^{q^2-1}=-1
\}.
\]
We have $|\mathcal{S}^{i}_{00}|=
|{\rm SL}_2(\mathbb{F}_q)|=q(q^2-1).$
By analyzing 
these singular points, 
we find $q(q^2-1)$
components $\{X^{i,c}_j\}_{j 
\in \mathcal{S}^{i}_{00}}$,  
each having an affine model 
$X^q+X=Y^{q+1}$ with genus $q(q-1)/2.$
Then, in 
$\overline{\mathcal{X}(\pi^2)}$, 
$\overline{\mathbf{Y}}^{i,c}_{2,2}$
appears and 
$\{X^{i,c}_j\}_{j 
\in \mathcal{S}^{i}_{00}}$ attach to 
$\overline{\mathbf{Y}}^{i,c}_{2,2}$ 
at distinct $q(q^2-1)$
points above the points 
$\mathcal{S}^i_{00}$.

In section \ref{acd1}-\ref{aci1}, 
by using the explicit computation 
of components in $\overline{\mathcal{X}(\pi^2)}$
explained above, 
 we also determine the (right)
  action of 
$G_2^F \times \mathcal{O}^{\times}_F \times I_F$
 on the components
  explicitly.
To determine the action of 
$\mathcal{O}^{\times}_D$, we use the description
of the action of $\mathcal{O}^{\times}_D$ on 
the Lubin-Tate space in \cite{GH}. 
See \ref{waq1} for more details.
Let $\varphi$ denote a prime element 
of $D$, and 
$U_D^n \subset \mathcal{O}^{\times}_D$
 a subgroup $1+(\varphi^n).$
Then, 
the group $\mathcal{O}^{\times}_D$
acts on the stable reduction 
$\overline{\mathcal{X}(\pi^2)}$
by factoring 
through the finite quotient
$\mathcal{O}^{\times}_3:
=
\mathcal{O}_D^{\times}/U_D^3.$
On the other hand, 
we 
compute the 
inertia
 action on 
 the basis of \cite{CM2}.

Let $l \neq p$ be a prime number.
In section \ref{acu1}, 
we analyze the following \'{e}tale cohomologies,
 by using 
the explicit description of the action of 
$\mathbf{G}:=G_2^F
 \times \mathcal{O}^{\times}_3 \times I_F$, 
\[
W:=\bigoplus_{i \in 
(\mathcal{O}_F/\pi^2)^{\times}} 
\bigoplus_{j \in \mathcal{S}^i_{00}}
H^1(X^{i,c}_j,\overline{\mathbb{Q}}_l),\ 
W':=\bigoplus_{(i,\ast ,y_0) \in 
(\mathcal{O}_F/\pi^2)^{\times} \times 
\mathbb{P}^1(\mathbb{F}_q) 
\times \mu_{2(q^2-1)}}
H^1(W^{i,c}_{\ast ,y_0},\overline{\mathbb{Q}}_l)
\]
%of the components $\{X^i_j\}_{(i,j) \in 
%(\mathcal{O}_F/\pi^2)^{\times} \times \mathcal{S}^i_{00}}$
as $\mathbf{G}$-representations.
Since the dual graph of the stable reduction 
$\overline{\mathcal{X}(\pi^n)}$ is a tree, 
the spaces $W$ and $W'$ are
 $\mathbf{G}$-subrepresentations 
 of 
the \'{e}tale cohomology 
$H^1(\mathcal{X}(\pi^2)_{\mathbf{C}},
\overline{\mathbb{Q}}_l).$
Note that the group $\mathbf{G}$
acts on $H^1(\mathcal{X}(\pi^2)_{\mathbf{C}},
\overline{\mathbb{Q}}_l)$ 
on the left.
%Then, the action of $\mathcal{O}^{\times}_D$ on 
%$W$ and $W'$ factors
%through $\mathcal{O}_D^{\times}/U_D^3$. 
%Let $E/F$
%be the quadratic unramified extension. 
%Then, we have $\Gamma:=(\mathcal{O}/\pi^2)^{\times} \subset 
%\mathcal{O}^{\times}_D/U_D^3.$
%The restriction $W|_{G_2^F \times\{1\} \times \{1\}}$
%is a direct sum of all cuspidal representations
%of $G_2^F$ with multiplicity $q$.
%See \ref{Go} for a definition of 
%cuspidal representation of $G_2^F$. 
%In subsection \ref{Lo}, we also write down the Lusztig surface for 
%${\rm GL}_2(\mathcal{O}_F/\pi^2)$ as in \cite{Lus}.
%The cohomology $H_c^2$ of this surface
%contains all ``cuspidal'' representations of 
%${\rm GL}_2(\mathcal{O}_F/\pi^2)$
%in a sense of \cite{Sta}. See subsection \ref{Go}
%for a precise definition of cuspidal representation of 
%${\rm GL}_2(\mathcal{O}_F/\pi^2)$. 
%These representations are parametrized by ``primitive''
%characters of $\Gamma$. 
%We compare the \'{e}tale cohomology group 
%$W$
%with \'{e}tale cohomology $H_c^2$ of the Lusztig surface 
%as a ${\rm GL}_2(\mathcal{O}_F/\pi^2) \times \Gamma 
%\times \{1\}$-representation
%in Proposition \ref{gey}.2.
%Then, we analyze the \'{e}tale cohomology $H^1$ of 
%the components $\{X^i_j\}_{(i,j) 
%\in (\mathcal{O}_F/\pi^2)^{\times} \times \mathcal{S}^i_{00}},$ 
%with each having an affine model $X^q+X=Y^{q+1},$
%as a $\mathbf{G}:={\rm GL}_2(\mathcal{O}_F/\pi^2) \times 
%\mathcal{O}^{\times}_D/U_D^3 \times I_F$-representation.
By investigating 
$W$ (resp.\ $W'$) 
as in Corollary \ref{lap} (resp.\ 
Corollary \ref{dek2}) as 
a $\mathbf{G}$-representation, 
 we show 
that the local Jacquet-Langlands 
correspondence 
and the $\ell$-adic local Langlands correspondence 
for unramified (resp.\ ramified) cuspidal representations
of ${\rm GL}_2(F)$ of level $1$ (resp.\ level $1/2$)
are realized in the cohomology group $W$. (resp.\ $W'.$)
To check these things, we heavily depend on 
 \cite{BH}.
See \ref{con} for more details.
This is a part of Carayol's program in \cite{Ca}.
 See also \cite{Bo} and \cite{HT}. 
For general $h$, T.\ Yoshida computes 
the stable model of $\mathcal{X}(\pi)$
 and studies 
its cohomology in \cite{Y}.

In a subsequent paper \cite{IMT}, 
we will investigate 
a relationship 
between $W$ and the cohomology
$H_c^2$ of some Lusztig surface 
for $G_2^F.$
%in \cite{Lus}
%or \cite{Lus2}.
The third author thanks Professor T.\ Saito 
for helpful comments.
He also thanks to N.\ Imai for very 
fruitful discussions around topics in this
paper.
\begin{notation}
We fix some notations in this paper.
Let $F$ be a non-archimedean local field
 with the ring of integers $\mathcal{O}_F$, 
 the uniformizer $\pi$
 and the residue field 
 $k=\mathbb{F}_q$ of characteristic $p>0.$
 We fix an algebraic closure 
 $F^{\rm ac}$ of $F$. 
 The completion of $F^{\rm ac}$ is 
 denoted by $\mathbf{C}.$
 We write 
 $\mathcal{O}_{\mathbf{C}}$
 for the ring of integers of $\mathbf{C}$.
 We write $k^{\rm ac}$ 
 for the residue field of $\mathbf{C}.$
 For an element $a \in 
\mathcal{O}_{\mathbf{C}}$,
we denote by $\bar{a}$ the image of $a$
by the reduction map 
$\mathcal{O}_{\mathbf{C}}
\to k^{\rm ac}.$
 Let $v(\cdot)$  denote 
 the normalized valuation of 
 $\mathbf{C}$ such that $v(\pi)=1$ 
 and $|\cdot|$
the absolute value 
given by $|x|=p^{-v(x)}$.
 Let $F^{\rm nr}$ denote 
 the maximal unramified extension of $F$
 in $F^{\rm ac}$ and $F_0$ 
 denote its completion.
 Let $E/F$ denote the quadratic unramified extension.
Let $\mathcal{R}=|\mathbf{C}^\ast|=p^{\mathbb{Q}}.$
We let $L$
be a complete subfield of $\mathbf{C}.$
For $r \in \mathcal{R},$
we let $B_L[r]$ and $B_L(r)$
denote the closed and open disks over $L$
of radius $r$
around $0,$ i.e.\
the rigid spaces over $L$
whose $\mathbf{C}$-valued points
are $\{x \in \mathbf{C}:\ |x| \leq r\}$ 
and $\{x \in \mathbf{C}:|x|<r\}$
respectively.
If $r, s \in \mathcal{R}$ and $r \leq s,$
let $A_L[r,s]$ and $A_L(r,s)$ be 
the rigid spaces over $L$
whose $\mathbf{C}$-valued points are 
$\{x \in \mathbf{C} :\ r \leq |x| \leq s\}$ and 
$\{x \in \mathbf{C}:\ r<|x|<s\},$
which we call a closed annulus and an open annulus respectively.
Furthermore, we write $C_L[r]$ for $A_L[r,r]$, which we call
a circle.
We simply write $G_n^F$ for ${\rm GL}_2(\mathcal{O}_F/\pi^n).$
Let $D$ be an $F$-algebra which is a division algebra, 
with center $F$ and dimension $4.$
Let $\mathcal{O}_D$ be the ring of integers
of $D$ and $\varphi $ a prime element of $D$ such
 that $\varphi^2=\pi.$ 
We set $U_D^n:=1+(\varphi^n) \subset D^{\times}.$
The group $U_D^n$ is a compact open normal
subgroup of $D^{\times}.$
We put $\mathcal{O}^{\times}_n:=\mathcal{O}^{\times}_D/U_D^n.$
Let ${\rm Nrd}_{D/F}:D^{\times} \to F^{\times}$
be the reduced norm map, and ${\rm Trd}_{D/F}:D \to F$
the reduced trace map.
For an element $x \in (\mathcal{O}_F/\pi^2)^{\times},$
we write $x=a_0+a_1\pi$ with 
$a_0 \in \mu_{q-1}(\mathcal{O}_F)$ and 
$a_1 \in \mu_{q-1}(\mathcal{O}_F) \cup \{0\}.$
Then, we identify $(\mathcal{O}_F/\pi^2)^{\times}$
with $\mathbb{F}^{\times}_q \times \mathbb{F}_q$
by the following map $x=a_0+a_1\pi \mapsto 
(\bar{a}_0,\frac{\bar{a}_1}{\bar{a}_0}).$
Let $W_F$ denote the Weil group of $F$ and 
$I_F \subset W_F$ denote the inertia subgroup. 
For a prime number $l \neq p$ and 
a finite abelian group $A$, we set
 $A^{\vee}:={\rm Hom}
 (A,\overline{\mathbb{Q}}_l^{\times}).$ 
If we have $v(f-g)>\alpha$ with $\alpha 
\in \mathbb{Q}_{ \geq 0}$, we write $f \equiv g\ 
({\rm mod}\ \alpha+).$
For an affine curve $X/k^{\rm ac},$
the genus of $X$ means the genus of 
the smooth compactification $X^c$ of 
the normalization of $X.$
\end{notation}

\section{Preliminaries from
 \cite{W2}, \cite{W3} and \cite{GH}} 
In this section, we collect several things and facts 
around 
the Lubin-Tate space 
$\mathcal{X}(\pi^n)$, for example, formal 
$\mathcal{O}_F$-module, good
model of the universal $\mathcal{O}_F$-module 
$\mathcal{F}^{\rm univ}$, 
a set of geometrically connected components of 
$\mathcal{X}(\pi^n)$, 
 the definition of the action of
 $\mathcal{O}^{\times}_D$ on $\mathcal{X}(\pi^n)$ 
 etc.\ from \cite{W2}, \cite{W3} and \cite{GH}. 
Furthermore, we define 
several subspaces in $\mathcal{X}(\pi^2),$
whose reduction will
 be computed in the proceeding section.
\subsection{definition of formal modules}
We begin with the 
definition of formal $\mathcal{O}_F$-modules.
\begin{definition}
Let $R$ be a commutative $\mathcal{O}_F$-algebra, 
with structure map $i:\mathcal{O}_F
\longrightarrow R.$
A formal one-dimensional $\mathcal{O}_F$-module $\mathcal{F}$
is a power series 
$\mathcal{F}(X,Y)=X+Y+\cdots \in R[[X,Y]]$
which is commutative, associative, 
admits $0$ as an identity,
together with a power series $[a]_{\mathcal{F}}(X) \in R[[X]]$
for each $a \in \mathcal{O}_F$ satisfying 
$[a]_{\mathcal{F}}(X) \equiv i(a) X\ {\rm mod}\ X^2$
and $\mathcal{F}([a]_{\mathcal{F}}(X),
[a]_{\mathcal{F}}(Y))=[a]_{\mathcal{F}}(\mathcal{F}(X,Y)).$
\end{definition}

The addition law on a formal 
$\mathcal{O}_F$-module $\mathcal{F}$
will usually be written 
$X+_{\mathcal{F}}Y.$
If $R$ is a $k$-algebra, 
we either have $[\pi]_{\mathcal{F}}(X)=0$
or else $[\pi]_{\mathcal{F}}(X)=f(X^{q^h})$
for some power series $f(X)$
with $f'(0) \neq 0.$
In the latter case, we say $\mathcal{F}$
has height $h$ over $R.$
Let $\Sigma$
be a one-dimensional 
formal $\mathcal{O}_F$-module
over $k^{\rm ac}$ of height $h.$
The formal $\mathcal{O}_F$-module 
$\Sigma$ is unique up to isomorphism.
Furthermore, a model for $\Sigma$
is given by the simple rules 
when ${\rm char}\ F>0$
$$X+_{\Sigma}Y=X+Y,\ 
[\zeta]_{\Sigma}(X)=\zeta X\ 
(\zeta \in k),\ 
[\pi]_{\Sigma}(X)=X^{q^h}.$$

The functor of deformations of $\Sigma$ 
to complete local Noetherian $\mathcal{O}_{F_0}$-algebra
is reresentable by a universal deformation
$\mathcal{F}^{\rm univ}$
over an algebra $\mathcal{A}(1)$
which is isomorphic to the power series ring 
$\mathcal{O}_{F_0}[[u_1,..,u_{h-1}]]$
in $(h-1)$ variables, cf \cite{Dr}.
That is, if $A$ is a complete local 
$\mathcal{O}_{F_0}$-algebra with maximal ideal $P,$
then, the isomorphism classes of deformations of $\Sigma$ to $A$
are given exactly by specializing each $u_i$
to an element of $P$ in $\mathcal{F}^{\rm univ}.$

\subsection{Moduli of deformations with level structure}
\label{sb1}
Let $A$ be a complete local $\mathcal{O}_F$-algebra
with maximal ideal $M,$ and let $\mathcal{F}$
be a one-dimensional $\mathcal{O}_F$-module 
over $A,$ and let $h \geq 1$
be the height of $\mathcal{F} \otimes (A/M).$
\begin{definition}
Let $n \geq 1.$
A Drinfeld level $\pi^n$-structure
on $\mathcal{F}$ is an 
$\mathcal{O}_F$-module homomorphism
$$\phi:(\pi^{-n}\mathcal{O}_F/\mathcal{O}_F)^h 
\longrightarrow M$$
for which
the relation
$$\prod_{x \in (\pi^{-1}\mathcal{O}_F/\mathcal{O}_F)^h}
(X-\phi(x))\ |\ [\pi]_{\mathcal{F}}(X)$$
holds in $A[[X]].$
If $\phi$ is a Drinfeld level 
$\pi^n$-structure, the image of $\phi$
of the standard basis
elements $(\pi^{-n},0,..,0),\dots,(0,0,..,\pi^{-n})$
of $(\pi^{-n}\mathcal{O}_F/\mathcal{O}_F)^h$
form a Drinfeld basis of $\mathcal{F}[\pi^n].$
\end{definition}

Fix a formal $\mathcal{O}_F$-module $\Sigma$
of height $h$ over $k^{\rm ac}.$
Let $A$ be a noetherian local 
$\mathcal{O}_{F_0}$-algebra
such that the structure morphism 
$\mathcal{O}_{F_0} \longrightarrow A$
induces an isomorphism between residue fields.
A deformation of $\Sigma$
with level $\pi^n$-structure 
over $A$ is a triple $(\mathcal{F},\eta,\phi)$
where $\mathcal{F}$ is a formal $\mathcal{O}_F$-module over $A$, 
$\eta:\Sigma \overset{\sim}{\to } 
\mathcal{F} \otimes k^{\rm ac}$
is an isomorphim of 
$\mathcal{O}_F$-modules 
over $k^{\rm ac}$
and $\phi$ is a Drinfeld 
level $\pi^n$-structure
on $\mathcal{F}.$
\begin{proposition}(\cite{Dr})
The functor which assigns to each $A$
as above the set of deformations of $\Sigma$
with Drinfeld level $\pi^n$-structure over $A$
is represented by a regular local ring $\mathcal{A}(\pi^n)$
of relative 
dimension $h-1$ over 
$\mathcal{O}_{F_0}.$
Let $X_1^{(n)},...,X_{h}^{(n)}$
be the corresponding
Drinfeld basis for $\mathcal{F}^{\rm univ}[\pi^n].$
Then, these elements form a set of 
regular parameters for $\mathcal{A}(\pi^n).$
\end{proposition}

There is a finite injection of 
$\mathcal{O}_{F_0}$-algebras $[\pi]_u
:\mathcal{A}(\pi^n)
\hookrightarrow \mathcal{A}(\pi^{n+1})$
corresponding to the obvious degeneration map of functors.
We therefore may consider $\mathcal{A}(\pi^{n})$
as a subalgebra of $\mathcal{A}(\pi^{n+1}),$
with the equation
$[\pi]_u(X_i^{(n)})=X_i^{(n+1)}$
holding in $\mathcal{A}(\pi^{n+1}).$
Let $X(\pi^n)={\rm Spf}\ 
\mathcal{A}(\pi^n),$
so that $X(\pi^n) 
\longrightarrow 
{\rm Spf}\ \mathcal{O}_{F_0}$
is formally smooth of relative dimension
$h-1.$ Let $\mathcal{X}(\pi^n)$
be the generic fiber of $X(\pi^n).$
Then, $\mathcal{X}(\pi^n)$
is a rigid analytic variety of dimension 
$h-1$ over 
$F_0$.
We call $\mathcal{X}(\pi^n)$
the {\it Lubin-Tate space of level $n.$}
The coordinates $X_i^{(n)}$
are then analytic functions on $\mathcal{X}(\pi^n)$
with values in the open unit disc.
We have that $\mathcal{X}(1)$
is the rigid analytic 
open unit polydisc of dimension $h-1.$
The group 
${\rm GL}_h(\mathcal{O}_F/\pi^n\mathcal{O}_F)$
acts on the right on $\mathcal{X}(\pi^n)$
and on the left on $A(\pi^n).$
The degeneration map 
$\mathcal{X}(\pi^n) 
\longrightarrow \mathcal{X}(1)$
is Galois with group 
${\rm GL}_h(\mathcal{O}_F/\pi^n\mathcal{O}_F).$
For an element $M \in {\rm GL}_h(\mathcal{O}_F/\pi^n\mathcal{O}_F)$
and an analytic function $f$ on $\mathcal{X}(\pi^n),$
we write $M(f)$
for the translated function $z \mapsto f(zM).$
When $f$ happens to be one of the parameters $X_i^{(n)},$
there is a natural definition of $M(X_i^{(n)})$
when $M \in M_h(\mathcal{O}_F/\pi^n\mathcal{O}_F)$
is an arbitrary matrix: if $M=(a_{ij}),$ then 
$$M(X_i^{(n)})=[a_{j1}]_{\mathcal{F}^{\rm univ}}(X_1^{(n)})+
_{\mathcal{F}^{\rm univ}} \dots 
+_{\mathcal{F}^{\rm univ}}
[a_{jh}]_{\mathcal{F}^{\rm univ}}(X_h^{(n)}).$$

\subsection{The universal deformation in the equal characteristic case}
We briefly recall a simple model of 
$\mathcal{F}^{\rm univ}$ given 
in \cite[2.2]{W2}, 
\cite[3.8]{W4} 
and in \cite[Proposition 5.1.1]{St2}.
Assume ${\rm char}\ F=p>0,$ 
so that $F=k((\pi))$
is the field of Laurent series over $k$
in one variable,
with the ring of integers $\mathcal{O}_F=k[[\pi]].$

The universal deformation $\mathcal{F}^{\rm univ}$ of $\Sigma$
 has a simple model over $\mathcal{A}(1) \simeq 
\mathcal{O}_{F_0}[[u_1,..,u_{h-1}]]:$
$$X+_{\mathcal{F}^{\rm univ}}Y=X+Y$$
$$[\zeta]_{\mathcal{F}^{\rm univ}}(X)=\zeta X, \zeta \in k$$
\begin{equation}\label{rp0}
[\pi]_{\mathcal{F}^{\rm univ}}(X)
=\pi X+u_1 X^q+\cdots+u_{h-1}X^{q^{h-1}}+X^{q^h}.
\end{equation}
In particular, 
we have the following, in the case $h=2,$
\begin{equation}\label{rp1}
[\pi]_{\mathcal{F}^{\rm univ}}(X)=X^{q^2}+uX^q+\pi X
\end{equation}
over ${X}(1) \simeq {\rm Spf}\ 
\mathcal{O}_{F_0}[[u]].$

\subsection{The universal deformation in the 
mixed characteristic case}\label{gh1}
From now on, we assume $h=2.$
In this subsection, we fix an identification
${X}(1) \simeq {\rm Spf}\ 
\mathcal{O}_{F_0}[[u]]
$ over which
the universal formal $\mathcal{O}_F$-module
$\mathcal{F}^{\rm univ}$
has a simple form 
(``$\mathcal{O}_F$-typical'') 
in the mixed characteristic case.

Let $\mathcal{F}$ be a formal 
$\mathcal{O}_F$-module of 
dimension $1$ over an 
$\mathcal{O}_F$-algebra $R$.
Let $w$ be an invariant differential on 
$\mathcal{F}$. 
See \cite{GH}[section 3] for more details.
We assume that $R$ is flat over 
$\mathcal{O}_F.$
Then, a {\it logarithm} 
of $\mathcal{F}$
means a unique isomorphism
$\mathcal{F} \overset{\sim}{\to}
\mathbb{G}_a$ over $R \otimes F$
with $df=w.$
In the following, 
we consider a logarithm $f$
which has the following form
\[
f(X)=X+\sum_{k \geq 1}b_kX^{q^k}
\]
with $b_k \in R \otimes F.$
By \cite[(5.3) and 
(12.3)]{GH}, 
the universal logarithm $f(X)$ 
over $\mathcal{O}_F[[u]]$ 
satisfies the following 
Hazewinkel's 
``functional equation'' 
in \cite[21.5]{Haz}
\begin{equation}\label{so}
f(X)=X+\frac{f(uX^q)}{\pi}
+\frac{f(X^{q^2})}{\pi}.
\end{equation}
If we write
$f(X)=\sum_{k \geq 0}b_kX^{q^k} 
\in F[[u,X]],$
the coefficients 
$\{b_k\}_{k \geq 0}$ 
satisfy the following
\[b_0=1,\ \pi 
b_k=\sum_{0 \leq j \leq k-1}b_jv_{k-j}^{q^{j}}\]
where $v_1=u, v_2=1, v_k=0\ (k \geq 3).$
For example, we have the followings
\[b_0=1,\ 
b_1=\frac{u}{\pi},\ 
 b_2=\frac{1}{\pi}\biggl(1+\frac{u^{q+1}}{\pi}\biggr),\ 
 b_3=\frac{1}{\pi^2}\biggl(u+u^{q^2}+\frac{u^{q^2+q+1}}{\pi}
 \biggr),\]
 \begin{equation}\label{for}
 b_4=\frac{1}{\pi^2}\biggl(
 1+\frac{u^{q+1}+u^{q^3+1}+u^{q^2(q+1)}}{\pi}
 +\frac{u^{(q+1)(q^2+1)}}{\pi^2}\biggr), \cdots.
 \end{equation}
By \cite[Proposition 5.7]{GH}
 or \cite[21.5]{Haz}, if we set
as follows
\begin{equation}\label{sos}
\mathcal{F}^{\rm univ}(X,Y):=f^{-1}(f(X)+f(Y))
\end{equation}
\begin{equation}\label{soso}
[a]_{\mathcal{F}^{\rm univ}}(X):=f^{-1}(af(X))
\end{equation}
for $a \in \mathcal{O}_F,$
it is known that
these power series 
have coefficients in $\mathcal{O}_F[[u]]$
and define the universal formal 
$\mathcal{O}_F$-module $\mathcal{F}^{\rm univ}$
 over $\mathcal{O}_F[[u]]$ with logarithm $f(X).$
In the following, we simply write
$[a]_u$ for $[a]_{\mathcal{F}^{\rm univ}}$ and 
$X+_uY$ for $\mathcal{F}^{\rm univ}(X,Y)$.
Then, we have $[\zeta]_u(X)=\zeta X$
for $\zeta \in \mu_{q-1}(\mathcal{O}_F).$
Note that the above model 
$\mathcal{F}^{\rm univ}$ exists 
for general $h$, 
which is given in \cite{GH}.

We have the following approximation formula
for $[\pi]_u$, which plays a key role when 
we compute irreducible components 
in the stable reduction of $\mathcal{X}(\pi^2).$
\begin{lemma}\label{all}
We assume that ${\rm char}\ F=0.$
Let $e_{F/\mathbb{Q}_p}=v(p)$ be 
the absolute ramification index.
\\1.\ Then, we have the
 following congruence 
\[
[\pi]_{u}(X) \equiv 
\pi X+uX^q
+X^{q^2}
-\biggl(\frac{q}{\pi}\biggr)
u^2X^{q(q^2-q+1)}\] 
\[({\rm mod}\ (\pi^3,\pi^2X^{q^4-q^2+1},\pi u
X^{q^3-q^2+1}, \pi u^2 X,  
\pi^2uX^q,u^3X^{q^2+1},X^{q^5},u^5,u^4X^{q^2})),\] 
\\2.\ If $e_{F/\mathbb{Q}_p} \geq 2$, we have 
the following 
congruence 
\[
[\pi]_{u}(X) \equiv 
\pi X+uX^q(1-\pi^{q-1})
+X^{q^2}\] 
\[({\rm mod}\ (X^{q^5},u^{2q+1},
\pi^3,\pi^2uX^{q+1},\pi
 u^2 X^{q(q^2-q+1)})).
\]
\\3.\ If $e_{F/\mathbb{Q}_p} \geq 3$, we have
 the following congruence 
 \[
 [\pi]_{u}(X) \equiv 
\pi X+uX^q(1-\pi^{q-1})
+X^{q^2}\] 
\[
({\rm mod}\ (u^{q^2+q+1},X^{q^6},\pi 
u^{2q+1},
\pi^4,\pi^3uX^{q+1},\pi^2
 u^2 X^{q(q^2-q+1)})).
 \]
\end{lemma}
\begin{proof}
These assertions follow
from direct computations by using 
the relationship $f([\pi]_{u}(X))=\pi f(X)$ 
and (\ref{for}).
\end{proof}

%By the functional equation (\ref{so}),
%the following equality
%holds
%\begin{equation}\label{so2}
%f(X)=X+\frac{uX^q}{\pi}+\frac{X^{q^2}}{\pi}
%+\frac{f(u^{q+1}X^{q^2})+f(u^{q^2}X^{q^3})+f(uX^{q^3})+f(X^{q^4})}{\pi^2}.
%\end{equation}

Throughout the paper, 
$\mathcal{F}^{\rm univ}/
{X}(1)$
means the universal
 formal $\mathcal{O}_F$-module 
 given
in (\ref{rp1}) 
in the equal
 characteristic case,  
  and  (\ref{sos}) 
  in the mixed 
  characteristic case
with the identification 
$\mathcal{X}(1) 
\simeq {\rm Spf}\ 
\mathcal{O}_{F_0}[[u]]$ 
respectively. By $X(1) \simeq 
{\rm Spf}\ 
\mathcal{O}
_{F_0}[[u]],$ we also identify
$\mathcal{X}(1) \simeq B(1) \ni u.$

\subsection{Geometrically connected 
components of the Lubin-Tate space 
$\mathcal{X}(\pi^n)$ in 
\cite{W3}}\label{geo}
We collect some facts 
on the geometrically 
connected components
of $\mathcal{X}(\pi^n)$ 
from \cite[subsection 3.6]{W3}.

Let ${\rm LT}$ be a 
one-dimensional 
formal $\mathcal{O}_F$-module
over ${\mathcal{O}}_{F_0}$
for which ${\rm LT} 
\otimes k^{\rm ac}$ has height one;
this is unique up to isomorphism.
In the following, 
we choose a model of
 ${\rm LT}$ such that 
\[
[\pi]_{\rm LT}(X)=\pi X-X^{q}.
\]
See \cite[Section 3]{W2}.
For $n \geq 1,$ let 
$F_n:=F_0({\rm LT}[\pi^n]
(\mathbf{C})).$
Then, by the 
classical Lubin-Tate theory, 
$F_n/F_0$
is an abelian extension
 with Galois group
$(\mathcal{O}_F/\pi^n
\mathcal{O}_F)^{\times},$
and the union 
$\bigcup_{n \geq 1}F_n$
is the completion of 
the maximal 
abelian extension of $F.$ 

Let $\Sigma$ be a one-dimensional 
formal $\mathcal{O}_F$-module 
over $k^{\rm ac}$
of height $2$ as 
in subsection \ref{sb1}.
Then, we have $\mathcal{O}_D
={\rm End}_{k^{\rm ac}}(\Sigma).$
We have the following right action of 
$\mathcal{O}^{\times}_D$ on 
$\mathcal{X}(\pi^n).$
Let $b \in \mathcal{O}^{\times}_D.$
Then, $b$ acts on $\mathcal{X}(\pi^n)$
by $(\mathcal{F},\eta,\phi) 
\mapsto (\mathcal{F},\eta \circ b,\phi).$
We will introduce another description
of the above $\mathcal{O}^{\times}_D$-action 
in the proceeding subsection. 
%and $D:=\mathcal{O}_D 
%\otimes_{\mathcal{O}_F} F.$
%Then, $D$ is the division algebra over $F$ of invariant $1/2$
%with the ring of integers $\mathcal{O}_D.$
%Let ${\rm Nrd}_{D/F}:D^{\times} \longrightarrow F^{\times}$ be the reduced norm.

All geometrically connected components 
of $\mathcal{X}(\pi^n)$
are defined over $F_n,$ 
and they are in bijection 
with primitive elements 
in the free of rank
$1$ $(\mathcal{O}_F/\pi^n\mathcal{O}_F)$-module
 ${\rm LT}[\pi^n](\mathbf{C}).$
 Now, we identify 
 ${\rm Gal}(\mathbf{C}/F_0) \simeq 
 I_F={\rm Gal}(\mathbf{C}/F^{\rm ur})$
 by the canonical restriction.
For the following theorem, we refer to \cite{St} 
and \cite[Theorem 3.2]{W3}:

\begin{theorem}\label{caq}(\cite{St})
For each $n \geq 1,$ there exists a locally constant 
rigid analytic morphism
$\Delta^{(n)}:\mathcal{X}(\pi^n) 
\longrightarrow {\rm LT}[\pi^n]$
which surjects 
onto the subset of ${\rm LT}[\pi^n]$ 
which are primitive. 
For a triple $(g,b,\tau) 
\in {\rm GL}_2(\mathcal{O}_F)
 \times \mathcal{O}^{\times}_D 
\times I_F,$ we have
$$\Delta^{(n)} \circ (g,b,\tau)
=[\delta(g,b,\tau)]_{\rm LT}\bigl(\Delta^{(n)}\bigr),$$
where $\delta$ is the homomorphism
\begin{equation}\label{cnn}
\delta:{\rm GL}_2(F) \times D^{\times} \times W_F 
\to F^{\times};
(g,b,w) \mapsto {\rm det} (g) \times 
{\rm Nrd}_{D/F}(b)^{-1} \times 
\mathbf{a}_F(w)^{-1}
 \end{equation}
and 
$\mathbf{a}_F:W_F^{\rm ab} \overset{\sim}{\to} F^{\times} $ 
is the reciprocity map from
local class field theory, normalized so that
$\mathbf{a}_F$ sends a geometric Frobenius element to 
 a prime element.
The geometric fiber of $\Delta^{(n)}$ are connected.
\end{theorem}
\begin{remark}
In Theorem \ref{caq}, we 
consider the {\it right}
action of ${\rm GL}_2(\mathcal{O}_F)
 \times \mathcal{O}^{\times}_D 
\times I_F$
on the Lubin-Tate space 
$\mathcal{X}(\pi^n)\times_{F_0} 
\mathbf{C}.$ See \cite[1.3]{Ca} for more 
detail on group action on the Lubin-Tate tower.
\end{remark}
%\begin{remark}
%In this paper, we consider the {\it left} action of $W_F$.
%In \cite[Theorem 3.2]{W3} and \cite{St}, they 
%consider the right action of $W_F.$
%\end{remark}
We write $\pi_0(\mathcal{X}(\pi^n))$ for 
the set of geometrically connected components of 
$\mathcal{X}(\pi^n).$
Then, by Theorem \ref{caq}, we can identify 
$\pi_0(\mathcal{X}(\pi^n)) \simeq 
(\mathcal{O}_F/\pi^n)^{\times}$, and 
know the action of 
${\rm GL}_2(\mathcal{O}_F)
 \times \mathcal{O}^{\times}_D 
\times I_F$ on it.
For $i \in (\mathcal{O}_F/\pi^n)^{\times}$,
let $\mathcal{X}^i(\pi^n) \subset 
\mathcal{X}(\pi^n) \times_{F_0} F_n$
 denote the connected 
component 
corresponding to $i$ under the above identification
$\pi_0(\mathcal{X}(\pi^n)) \simeq 
(\mathcal{O}_F/\pi^n)^{\times}$.
Then, for a subspace 
$W \subset \mathcal{X}(\pi^n),$
we write $W^i$ for the intersection
$(W \times_{F_0} F_n)
 \cap \mathcal{X}^i(\pi^n)$.

\subsection{Action of the division algebra 
$\mathcal{O}^{\times}_D$ 
on the Lubin-Tate space 
$\mathcal{X}(\pi^n)\ (n \geq 0)$}\label{waq1}
%Let $D/F$ denote 
%the division algebra of invariant $1/2.$
%We write $\mathcal{O}_D$ for the ring of integers of $D.$
In this subsection, we 
recall a description of 
the right
 action of $\mathcal{O}_D^{\times}$ 
on the Lubin-Tate 
space $\mathcal{X}(\pi^n)$ 
from \cite{GH}.
Recall that we set 
$X(\pi^n):={\rm Spf}\ 
\mathcal{A}(\pi^n)$ and 
$\mathcal{X}(\pi^n)$
is the generic 
fiber of $X(\pi^n).$
For a formal scheme $X$ 
over 
${\rm Spf}\ \mathcal{O}_{F_0}$
and a finite extension $L/F,$
we simply 
write $X\times_{F}L$ for 
the base change 
$X \times_{{\rm Spf}\ 
\mathcal{O}_{F_0}} {\rm Spf}\ 
\mathcal{O}_{\hat{L}^{\rm ur}}.$
%In this subsection, we compute the action of
%the division algebra ${\rm Aut}(\Sigma) \simeq \mathcal{O}^{\times}_D$
%with invariant $1/2$ on the reduction 
%$\overline{\mathbf{Y}}^{1}_{1,1}.$
Let $E/F$ be the unramified 
quadratic extension.
Let $\sigma \neq 1 \in {\rm Gal}(E/F)$.
The ring of integers 
$\mathcal{O}_D$ 
has the following 
description
$\mathcal{O}_D 
\simeq \mathcal{O}_{E} \oplus 
\varphi \mathcal{O}_{E}$
with $a \varphi=\varphi a^{\sigma}$ for $a 
\in \mathcal{O}_{E}.$
Let $b=\alpha+\varphi 
\beta \in \mathcal{O}^{\times}_D$
with $\alpha \in
 \mathcal{O}^{\times}_{E}$ and $\beta
 \in \mathcal{O}_E.$
%Let $(\mathcal{F},\iota) \in \mathcal{X}(1)$ and $F(X,Y)$ the formal 
%group law of $\mathcal{F}.$
%Then, we consider $b$ as an element $b(X)$ of $k[[X]],$
%because of
%$\mathcal{O}^{\times}_D \simeq {\rm Aut}(\Sigma).$ 
%Take a lifting $\tilde{b} \in \mathcal{O}_{K_2}[[X]]$ of $b(X).$
%Then, the formal group law of $b(\mathcal{F},\iota)$
%is described as follows
%$\tilde{b}^{-1}(F(\tilde{b}(X),\tilde{b}(Y))).$
%Then, this action does not depend on the choice of 
%the lifting $\tilde{b}$ of $b.$ 
%See \cite[pp.49 in Section 14]{GH} for more details.
By \cite[Proposition 14.7]{GH}, 
we have the following map
\begin{equation}\label{caa0}
b:{X}(1)\times_F E 
\to {X}(1)\times_F E.
\end{equation}
In the following, 
we recall the 
definition of the map
(\ref{caa0}).

For an element $z \in \mathcal{O}_{E}$,
we denote by $\bar{z}$ for the image of 
$z$ by the canonical map $\mathcal{O}_{E}
\to \mathbb{F}_{q^2}$.
By \cite[Section 13]{GH}, 
$b=\alpha+\varphi \beta
\in \mathcal{O}_D^{\times}
 \simeq 
{\rm Aut}(\Sigma)$ is written as follows
as an element of $\mathbb{F}_{q^2}[[X]],$
\begin{equation}\label{cop0}
b(X) \equiv \bar{a}_0 X+\bar{a}_1X^{q^2}+(\bar{\beta} X)^q\ 
({\rm mod}\ (X^{q^3}))
\end{equation}
where we write $\alpha=\sum_{i \geq 0} a_i \pi^i$
with $a_i \in \mu_{q^2-1}(\mathcal{O}_E) \cup\{0\}.$
Hence, we acquire the following congruence 
\begin{equation}\label{cop}
b^{-1}(X) \equiv 
\frac{X-(\bar{\beta}/\bar{a}_0)^qX^q
+\bigl(-(\bar{a}_1/\bar{a}_0)
+(\bar{\beta}/\bar{a}_0)^{q+1}
\bigr)X^{q^2}}{\bar{a}_0}\ 
({\rm mod}\ (X^{q^3})).
\end{equation}
Let $\tilde{b} \in \mathcal{O}_{E}[[X]]$ 
denote a lifting of $b \in 
\mathbb{F}_{q^2}[[X]].$
Let $\mathcal{F}_{\tilde{b}}/\mathcal{X}(1) 
\times_F E$ 
denote the universal formal 
$\mathcal{O}_F$-module 
whose multiplications
are defined as follows
\[
\mathcal{F}_{\tilde{b}}(X,Y)={\tilde{b}}^{-1} 
\circ \mathcal{F}(\tilde{b}(X),\tilde{b}(Y)),\ 
[\pi]_{\mathcal{F}_{\tilde{b}}}(X)=\tilde{b}^{-1}
\circ [\pi]_u \circ \tilde{b}(X).
\]
We simply write $[\pi]_{\tilde{b}}$ for 
$[\pi]_{\mathcal{F}_{\tilde{b}}}$.
Then, clealy we have the
 following isomorphism as in \cite[(14.4)]{GH}
\[
\tilde{b}^{-1}:\mathcal{F}^{\rm univ} 
\overset{\sim}{\to}
\mathcal{F}_{\tilde{b}}:\ 
(X,u) \mapsto (\tilde{b}^{-1}(X),u)
\]
and the following equality
\begin{equation}\label{caa1}
[\pi]_b(\tilde{b}^{-1}(X))=
\tilde{b}^{-1} 
([\pi]_u(X)). 
\end{equation}
Since $\mathcal{F}_{\tilde{b}}$ is 
the universal 
formal 
$\mathcal{O}_F$-module
over ${X}(1) \times_F E$,
we have the following map
\begin{equation}\label{caa-1}
b:{X}(1) \times_F E
\to {X}(1)
\end{equation}
as in \cite[(14.5)]{GH}.
This map depends only on $b$.
The map (\ref{caa-1}) clearly extends to
the following map as in 
\cite[Proposition 14.7]{GH}
\begin{equation}\label{caa-2}
b:{X}(1) \times_F E \to 
{X}(1) \times_F E.
\end{equation}
This map is the map (\ref{caa0})
which we want. 
The map $b$ is  
an automorphism, which is
 proved in loc.\ cit.
Let $b^\ast \mathcal{F}^{\rm univ}/
{X}(1) \times_F E$ 
denote the pull-back of 
the universal formal 
$\mathcal{O}_F$-module 
 $\mathcal{F}^{\rm univ}/{X}(1)$
 by the map (\ref{caa-1}).
 Then, we have the following 
 unique isomorphism
 by \cite[(14.6)]{GH}
 \[
 j:b^\ast \mathcal{F}^{\rm univ} \overset{\sim}{\to} 
 \mathcal{F}_{\tilde{b}}:\ 
 (X,u) \mapsto (j(X),u)\]
such that 
$j(X) \equiv X\ ({\rm mod}\ (\pi,u)).$
Hence, we acquire the following equality
\begin{equation}\label{caa2}
[\pi]_{b^\ast \mathcal{F}^{\rm univ}}(j^{-1}(X))
=j^{-1}([\pi]_b(X)).
\end{equation}

On the other hand, we have the following isomorphism
\[
b^\ast \mathcal{F}^{\rm univ}
 \overset{\sim}{\to}
\mathcal{F}^{\rm univ}:\ 
(X',u) \mapsto (X',b(u)).\ 
\]
We consider 
the following embedding
\begin{equation}\label{ea}
{X}(\pi^n) \hookrightarrow
 \mathcal{F}^{\rm univ}[\pi^n] 
\times_{{X}(1)} \mathcal{F}^{\rm univ}[\pi^n] 
\ni (X_n,Y_n,u).
\end{equation}
The image of this embedding 
is determined by the Drinfeld condition 
for $(X_n,Y_n).$
In the remainder of this paper, 
we always consider the space 
${X}(\pi^n)$ as a subspace of 
the product $ \mathcal{F}^{\rm univ}[\pi^n] 
\times_{{X}(1)} \mathcal{F}^{\rm univ}[\pi^n].$
Then, 
we have the following
well-defined map
\begin{equation}\label{dac1}
b:{X}(\pi^n) \times_F E \to 
{X}(\pi^n) \times_F E:\ 
(X_n,Y_n,u) \mapsto 
(j^{-1} \circ \tilde{b}^{-1}(X_n),
j^{-1} \circ \tilde{b}^{-1}(Y_n),
b(u)).
\end{equation}
This is the action of 
$b \in \mathcal{O}^{\times}_D$ on 
${X}(\pi^n).$
This action also induces the action of 
$\mathcal{O}^{\times}_D$ on 
$\mathcal{X}(\pi^n)$.
We set
\begin{equation}\label{diq}
b^\ast (X):=j^{-1} \circ \tilde{b}^{-1}(X).
\end{equation}

\subsection{Several subspaces in 
$\mathcal{X}(\pi^n)$}\label{sub}
We will define several subspaces in the
 Lubin-Tate space $\mathcal{X}(\pi^n).$
We identify $\mathcal{X}(1)$ 
with an open unit ball
$B(1) \ni u$ as in subsection \ref{gh1}. 
Let $n \geq 1$ be a positive integer.
Let 
\[p_n:\mathcal{X}(\pi^n) 
\to \mathcal{X}(1)\ :\ (\mathcal{F}, \eta, \phi) 
\mapsto (\mathcal{F},\eta)
\]
be the canonical projection.
Let $\mathbf{TS}^0$ be the 
closed disc $B[p^{-\frac{q}{q+1}}] 
\subset \mathcal{X}(1).$ 
This is called the 
``too-supersingular locus.''
We define subspaces 
$\mathbf{Y}_{2n-m,m}\ (1 \leq m \leq n)$ 
as follows;
\[\mathbf{Y}_{n,n}:=p_n^{-1}(\mathbf{TS}^0) 
\subset \mathcal{X}(\pi^n),\]
\[\mathbf{Y}_{2n-m,m}:=
p_n^{-1}(C[p^{-\frac{1}
{q^{n-m-1}(q+1)}}]) 
\subset \mathcal{X}(\pi^n)\ 
(1 \leq m \leq n-1).\]

Let $n \geq 2$ be a positive integer.
For $1 \leq m \leq n-1,$ we define subspaces $\mathbf{Z}_{2(n-1)-m,m}$
as follows
\[\mathbf{Z}_{2(n-1)-m,m}:=p_n^{-1}(C[p^{-\frac{1}{2q^{(n-1)-m}}}]) \subset 
\mathcal{X}(\pi^n)\ (1 \leq m \leq n-1).\]

%%\subsection{Subspaces of the spaces $\mathbf{Y}_{3,1}$ and $\mathbf{Z}_{1,1}$}
We introduce several subspaces in 
the spaces $\mathbf{Y}_{3,1}$ and $\mathbf{Z}_{1,1},$
whose reduction appears in the stable reduction of the Lubin-Tate space
 $\mathcal{X}(\pi^2).$

As in (\ref{ea}), we consider
$X(\pi^n)$ as a formal subscheme of
 $\mathcal{F}^{\rm univ}[\pi^n] \times_{X(1)}
  \mathcal{F}^{\rm univ}[\pi^n] \ni (X_n,Y_n,u).$
\begin{definition}\label{fd}
1.\ Let $(X_2,Y_2,u) \in \mathbf{Y}_{3,1,0} \subset \mathbf{Y}_{3,1}$
be a subspace 
defined by the following conditions;
\[v(u)=\frac{1}{q+1},\ 
v(X_1)=\frac{q}{q^2-1},\ 
v(X_2)=\frac{1}{q(q^2-1)},\ v(Y_1)=\frac{1}{q(q^2-1)},v(Y_2)=\frac{1}{q^3(q^2-1)}.
\]
\\2. Let 
$(X_2,Y_2,u) \in \mathbf{Y}_{3,1,\infty} \subset \mathbf{Y}_{3,1}$
be a subspace defined by the following condition;
$(X_2,Y_2,u) \in \mathbf{Y}_{3,1,\infty}$
is equivalent to 
$(Y_2,X_2,u) \in \mathbf{Y}_{3,1,0}.$
\\3. Let $(X_2,Y_2,u) 
\in \mathbf{Y}_{3,1,c} 
\subset \mathbf{Y}_{3,1}$
be a subspace defined by the following conditions;
\[
v(u)=\frac{1}{q+1},\ v(X_1)=v(Y_1)=\frac{1}{q(q^2-1)},\ 
v(X_2)=v(Y_2)=\frac{1}{q^3(q^2-1)}.
\]
4.\ Let $(X_2,Y_2,u) \in 
\mathbf{Z}_{1,1,0} \subset \mathbf{Z}_{1,1}$
be a subspace 
defined by the following conditions;
\[
v(u)=\frac{1}{2},\ 
v(X_1)=\frac{1}{2(q-1)},\ 
v(X_2)=\frac{1}{2q^2(q-1)},\ 
v(Y_1)=\frac{1}{2q(q-1)},\ 
v(Y_2)=\frac{1}{2q^3(q-1)}.
\]
\\5. Let $(X_2,Y_2,u) \in 
\mathbf{Z}_{1,1,\infty} \subset \mathbf{Z}_{1,1}$
be a subspace defined by the following condition;
$(X_2,Y_2,u) \in \mathbf{Z}_{1,1,\infty}$
is equivalent to $(Y_2,X_2,u) \in \mathbf{Z}_{1,1,0}.$
\\6. Let $(X_2,Y_2,u) \in  \mathbf{Z}_{1,1,c} \subset \mathbf{Z}_{1,1}$
be a subspace defined by the following conditions;
\[
v(u)=\frac{1}{2},\ 
v(X_1)=v(Y_1)=\frac{1}{2q(q-1)},\ 
v(X_2)=v(Y_2)=\frac{1}{2q^3(q-1)}.
\]
\end{definition}
\begin{lemma}
Let the notation be as above.
\\1.\ The space $\mathbf{Y}_{3,1}$
has the following description
$$\mathbf{Y}_{3,1}=\mathbf{Y}_{3,1,0}
 \coprod \mathbf{Y}_{3,1,\infty} \coprod
\mathbf{Y}_{3,1,c}.$$
\\2.\ The space $\mathbf{Z}_{1,1}$
has the following description
$$\mathbf{Z}_{1,1}=\mathbf{Z}_{1,1,0}
 \coprod \mathbf{Z}_{1,1,\infty} \coprod
\mathbf{Z}_{1,1,c}.$$
\end{lemma}

\section{Computation of 
irreducible components 
in the stable reduction of $\mathcal{X}(\pi^2)$}
\label{com3}
From this section until the end of this paper, 
we assume $p \neq 2.$
In this section, we will compute 
the reduction of the spaces
$\mathbf{Y}_{2,2},\mathbf{Y}_{3,1}$ 
and $\mathbf{Z}_{1,1}$ by using the approximation formula
for $[\pi]_u$ in Lemma \ref{all}.
In the equal characteristic case, we have already computed 
the reduction of them in \cite{T2}, hence we assume
${\rm char}\ F=0$ in this section.
However, the computation in this section 
also holds in the equal characteristic case.
In the equal characteristic case, the computation 
in loc.\ cit.\ is easier than 
the one in this section,
because the rigid analytic morphism
$\Delta^{(n)}:\mathcal{X}(\pi^n) \longrightarrow {\rm LT}[\pi^n]$
in subsection 2.5 is given by the Moore determinant,
 which has a very simple 
form. See \cite{W2} for more details on the Moore determinant.  
We briefly give a summary of results in this section.
The reduction of the space $\mathbf{Y}_{3,1,0}$
has $q(q-1)$ connected components and 
each component is defined by 
$x^qy-xy^q=\zeta$ 
with some $\zeta \in \mathbb{F}^{\times}_q.$
Similarly, the reduction of the space $\mathbf{Y}_{3,1,c}$
has $q(q-1)^2$ connected components and each is defined by 
$x^qy-xy^q=\zeta$
with some $\zeta \in \mathbb{F}^{\times}_q.$
The reduction of the space $\mathbf{Z}_{1,1,0}$
has $q(q-1)$ connected components and each is defined by
\[Z^q=y^{q^2-1}+y^{-(q^2-1)}.\]
This affine curve with genus $0$ 
has singular points at $y \in \mu_{2(q^2-1)}.$
By blowing-up these points, we obtain 
$2(q^2-1)$ irreducible components with genus $(q-1)/2$,
 defined by $a^q-a=s^2.$
Similar things happen 
for the reduction of the space $\mathbf{Z}_{1,1,c}.$

The reduction of the space $\mathbf{Y}_{2,2}$
has $q(q-1)$ connected components and each component is defined by
\[Z^q=x^{q^3}y-xy^{q^3},\ x^qy-xy^q=\zeta\]
with some $\zeta \in \mathbb{F}^{\times}_q.$
This affine curve with genus $q(q-1)/2$
has $q(q^2-1)$ singular points at 
$(x,y)=(x_0,y_0) \in (k^{\rm ac})^2$ 
with $x_0^{q^2-1}=y_0^{q^2-1}=-1.$
Then, by blowing-up these points, we find $q(q^2-1)$ 
irreducible components
defined by $a^q-a=t^{q+1}.$

By \cite[Proposition 5.1]{W3}, it is 
guaranteed that the reduction of 
the spaces, explained above,
 really appears 
in the stable reduction
 $\overline{\mathcal{X}(\pi^2)}.$
 See also 
 \cite[Proposition 4.37]{Liu}
  for an analogous statement 
  in the language of schemes.

\subsection{Computation of 
the reduction of $\mathbf{Y}_{2,2}$}\label{yo1}
In this subsection, we 
compute the reduction 
of the space $\mathbf{Y}_{2,2}$
 by using the approximation 
 formula for $[\pi]_u$ 
 in Lemma \ref{all}.1.
 Let 
 $\pi_0(\mathbf{Y}_{2,2})$ denote 
 the set of geometrically 
 connected components of $\mathbf{Y}_{2,2}.$
First, we check that 
the reduction $\overline{\mathbf{Y}}_{2,2}$
has $q(q-1)$ connected components
 and fix an identification 
$\pi_0(\mathbf{Y}_{2,2}) \simeq 
(\mathcal{O}_F/\pi^2)^{\times}$.
Furthermore, we show that each component
is defined by 
the following equations
\[Z^q=x^{q^3}y-xy^{q^3},\ x^qy-xy^q=\zeta
\]
with some $\zeta \in \mathbb{F}^{\times}_q.$

Let $(X_2,Y_2)$ denote the Drinfeld $\pi^2$-basis
of $\mathcal{F}^{\rm univ}.$
We set $X_1:=[\pi]_{u}(X_2)$
 and $Y_1:=[\pi]_u(Y_2).$
On the space $(X_2,Y_2,u) 
\in \mathbf{Y}_{2,2},$
we have the followings as in
 subsection \ref{sub}
\[v(u) \geq \frac{q}{q+1},\ 
v(X_1)=v(Y_1)=\frac{1}{q^2-1},\ 
v(X_2)=v(Y_2)=\frac{1}{q^2(q^2-1)}.\]
We choose an element $\kappa_1$ such that
$\kappa_1^{q^3(q^2-1)}=\pi$ with 
$v(\kappa_1)=1/q^3(q^2-1).$
We set $\kappa:=\kappa_1^q$ and 
put $\gamma:=\kappa^{(q-1)(q^2-1)}$ with 
$v(\gamma)=(q-1)/q^2.$
We write
 $\gamma^{1/q(q^2-1)}$
for an 
element 
$\kappa_1^{q-1}$.
Then, we change variables as follows
$u=\kappa^{q^{3}(q-1)}u_0,\ 
X_1=\kappa^{q^2}x_1,\ 
Y_1=\kappa^{q^2}y_1,\ 
X_2=\kappa x$
and $Y_2=\kappa y.$
By $0=[\pi]_u(X_1)=[\pi]_u(Y_1)$ 
and Lemma \ref{all}.1,
we acquire the following
congruence
\begin{equation}\label{aw0}
u_0=-x_1^{q(q-1)}
-\frac{1}{x_1^{q-1}}
=-y_1^{q(q-1)}-\frac{1}{y_1^{q-1}}\ ({\rm mod}\ 1+).
\end{equation}
Let $f(X,Y)$ be 
a polynomial with coefficients in 
$\mathcal{O}_F$ such 
that $X^q-Y^q=(X-Y)^q+\pi 
f(X,Y).$
By the Drinfeld
 condition for $(X_1,Y_1)$,
we have $x_1 \neq y_1.$ 
Hence, 
we acquire 
\[(x_1^qy_1-x_1y_1^q)^{q-1}
=1-\pi 
\frac{f(x_1^qy_1,x_1y_1^q)}
{x_1^qy_1-x_1y_1^q}\ 
({\rm mod}\ 1+).
\]
This splits 
to $(q-1)$
 congruences 
 as follows
\begin{equation}\label{aw1}
x_1^qy_1-x_1y_1^q
=\zeta+
\pi f(x_1^qy_1,
x_1y_1^q)\ 
({\rm mod}\ 1+)
\end{equation}
with $\zeta 
\in \mu_{q-1}
(\mathcal{O}_F).$
%\begin{remark}
%We can also obtain the following equivalent congruence to (\ref{aw1}),
%under the variables $(X_1,Y_1),$
%\begin{equation}\label{f1}
%X_1^qY_1-X_1Y_1^q \equiv \pi_1\ ({\rm mod}\ (q/(q-1))+)
%\end{equation}
%where $\pi_1$ satisfies $\pi_1^{q-1}=\pi.$
%\end{remark}

By $X_1=[\pi]_u(X_2)$ 
and $Y_1=[\pi]_u(Y_2)$ and Lemma \ref{all}.1,
 we acquire the following 
congruences
\begin{equation}\label{aw2}
x_1 \equiv x^{q^2}
+\gamma^q u_0 x^q+\gamma^{q+1}x,\ 
y_1 \equiv y^{q^2}+\gamma^q u_0 
y^q+\gamma^{q+1}y\ ({\rm mod}\ 1+).
\end{equation}
We set
\begin{equation}\label{aw3'}
\mathcal{Z}:
=(x^qy-xy^q)^{q}
-\gamma (x^{q^3}y-xy^{q^3}).
\end{equation}
%\begin{equation}\label{aw3}
%(x^qy-xy^q)^{q^2}-\gamma^q(x^{q^3}y-xy^{q^3})^q
%-\gamma^q(x^qy-xy^q)^q+\gamma^{q+1}(x^{q^3}y-xy^{q^3})
%=\zeta-pf(x^{q^2}y^q,x^qy^{q^2})
%\end{equation}
%modulo $1+.$
Substituting (\ref{aw2})
to (\ref{aw1}), we acquire the 
following congruence by using (\ref{aw0})
\begin{equation}\label{aw4}
\mathcal{Z}^q-\gamma^q\mathcal{Z}
=\zeta+\pi
(f^q-f)\ 
({\rm mod}\ 1+)
\end{equation}
where we write 
$f$ for 
$f(x^{q^2}y^q,
x^qy^{q^2}).$
We choose an element $\gamma_0$
such that
$\gamma_0^q-\gamma^q\gamma_0=\zeta.$
We set as follows
\begin{equation}\label{aw5}
\mathcal{Z}=
(x^qy-xy^q)^{q}
-\gamma (x^{q^3}y-xy^{q^3})=
\gamma_0-\gamma^{\frac{q}{q-1}} c.
\end{equation}
We set $\mu:=c+f.$
Then, by substituting (\ref{aw5}) to 
(\ref{aw4}), and 
dividing it by $\pi$, 
 we acquire $\mu^q 
 \equiv 
 \mu\ ({\rm mod}\ 0+).$
%\begin{remark}
%If ${\rm char}\ F=0$ and 
%the ramification index 
%$e_{F/\mathbb{Q}_p} \geq 2$
%or ${\rm char}\ F=p>0,$ 
%we have $c \equiv \mu\ ({\rm mod}\ 0+)$, 
%because of $pf/\pi \equiv 0\ ({\rm mod}\ 0+).$
%\end{remark}
Therefore, the set 
$\pi_0(\mathbf{Y}_{2,2})$
 is identified with
$(\bar{\zeta},\bar{\mu}) 
\in \mathbb{F}^{\times}_q 
\times \mathbb{F}_q.$
Hence, we fix the 
following identification 
\begin{equation}\label{fc}
\pi_0(\mathbf{Y}_{2,2}) 
\simeq 
(\mathcal{O}_F/\pi^2)^{\times} 
\simeq 
\mathbb{F}_q^{\times} 
\times \mathbb{F}_q
 \ni (\bar{\zeta},
 \frac{\bar{\mu}}
 {\bar{\zeta}}).
\end{equation}

%\begin{remark}Let $\pi_2$ 
%be an element satisfying $\pi_2^q-\pi \pi_2=\pi_1.$
%We have $v(\pi_2)=1/q(q-1).$
%Let $\mu:=X_1Y_2^q-X_1^qY_2-X_2^qY_1+X_2Y_1^q.$
%This is called the Moore determinant in level $2$.
%The following congruence (\ref{f2}) means that
%the Moore determinant well approximates the $\pi^2$-torsion point
%of the universal formal $\mathcal{O}_F$-module of height $1,$
%also in the mixed characteristic case.
%The following congruence holds
%\begin{equation}\label{f2}
%X_1^qY_1-X_1Y_1^q \equiv \mu^q-\pi 
%\mu\ ({\rm mod}\ (q/(q-1))+).
%\end{equation}
%Hence, this induces the following congruence with the congruence (\ref{f1})
%$$\mu \equiv \pi_2\ ({\rm mod}\ (1/(q-1))+).$$ 
%This congruence also induces the congruence (\ref{aw5}).
%\end{remark}

We choose an element $\tilde{\gamma}_0$
such that
$\tilde{\gamma}_0^q=\gamma_0.$
We set as follows
\begin{equation}\label{aw6}
x^qy-xy^q=\tilde{\gamma}_0+\gamma^{1/q}Z_1.
\end{equation}
By substituting (\ref{aw6}) 
to (\ref{aw5}) and 
dividing it by $\gamma$, 
we obtain the following congruence 
\begin{equation}\label{aw7}
Z_1^q \equiv x^{q^3}y-xy^{q^3}
-\gamma^{1/(q-1)}c\ ({\rm mod}\ (1/q^2)+).
\end{equation}
Hence, we have proved 
the following proposition by (\ref{aw6}) and (\ref{aw7}).
\begin{proposition}
The reduction $\overline{\mathbf{Y}}_{2,2}$
 has $q(q-1)$ connected components, 
 and each component is defined by the following equations
 \[x^qy-xy^q=\zeta,\ Z^q=x^{q^3}y-xy^{q^3}\]
 with some $\zeta \in \mathbb{F}^{\times}_q.$
\end{proposition}

\subsection{Analysis of 
the singular residue classes in $\mathbf{Y}_{2,2}$}
\label{yo2}
In this subsection, 
we analyze the singular residue classes of $\mathbf{Y}_{2,2}.$
Let $\{\overline{\mathbf{Y}}^i_{2,2}\}_{i=(\zeta,\tilde{\mu}) \in 
\mathbb{F}^{\times}_q \times \mathbb{F}_q}$ denote the connected components
of the reduction $\overline{\mathbf{Y}}_{2,2}.$
Let $i=(\zeta,\tilde{\mu}) \in \mathbb{F}^{\times}_q \times \mathbb{F}_q.$
As in the previous subsection, the reduction
$\overline{\mathbf{Y}}^i_{2,2}$ is defined by the following equations
\[x^qy-xy^q=\zeta,\ Z^q=x^{q^3}y-xy^{q^3}\]
with $\zeta \in 
\mathbb{F}^{\times}_q.$
This affine curve has singular points at the following set
\[\mathcal{S}^i_{00}:=
\{(x_0,y_0,Z_0) \in \overline{\mathbf{Y}}^i_{2,2}(k^{\rm ac})\ |\ 
x_0^{q^2}y_0-x_0y_0^{q^2}=0\}.
\]
Note that we have $|\mathcal{S}^i_{00}|=q(q^2-1)=
|{\rm SL}_2(\mathbb{F}_q)|.$
This set $\mathcal{S}^{i}_{00}$
is identified with the following set
$\{(x_0,y_0) \in ({k^{\rm ac}}^{\times})^2\ 
|\ x_0^qy_0-x_0y_0^q=\zeta,\ 
x_0^{q^2-1}=y_0^{q^2-1}=-1\}.$
For $j:=(x_0,y_0) \in \mathcal{S}^i_{00},$
we denote by $\mathbf{X}^i_j$ the underlying affinoid 
of the singular residue class 
of $(x,y)=j$ in the space $\mathbf{Y}^i_{2,2}.$
Furthermore, we write $X^i_j$ 
for the reduction of the affinoid 
$\mathbf{X}^i_j.$
In this subsection, we compute the reduction $X^i_j.$
 
By using (\ref{aw6}), 
the congruence 
(\ref{aw7}) is rewritten as follows
\begin{equation}\label{aw8}
Z_1^q \equiv \tilde{\gamma}_0^q 
\frac{(\tilde{\gamma}_0^{q-1}+y^{q^2-1})^q}{y^{q^2-1}}
+y^{q(q^2-1)}(\tilde{\gamma}_0+\gamma^{1/q}Z_1)-\gamma^{1/(q-1)}c\ 
({\rm mod}\ (1/q^2)+).
\end{equation}
 We choose an element $\tilde{\gamma}_1$ such that
 $\tilde{\gamma}_1^q
 +\tilde{\gamma}_0
 +\gamma^{1/q}
 \tilde{\gamma}_1=0.$
 Furthermore, choose $y_0,x_0$ such that 
 $y_0^{q^2-1}+\tilde{\gamma}_0^{q-1}=0$
 and $x_0^qy_0-x_0y_0^q
 =\tilde{\gamma}_0+\gamma^{1/q}\tilde{\gamma}_1.$
We set $w:=\gamma^{1/q(q-1)}$ and 
$w_1:=y_0^q \gamma^{1/(q^2-1)}.$
 Then, we have $v(w)=1/q^3$ 
 and $v(w_1)=1/q^2(q+1).$
 
 We change variables as follows
 \begin{equation}\label{ss}
 Z_1=\tilde{\gamma}_1+w a,\  
 y=y_0+w_1z_1.
 \end{equation}
Substituting them to (\ref{aw8}), 
we acquire the following congruence
\begin{equation}\label{aw9}
w^q(a^q+a) 
\equiv -w_1^{q+1}\zeta \biggl(\frac{z_1}{y_0}\biggr)^{q+1}
-\gamma^{1/(q-1)}c_0\ ({\rm mod}\ (1/q^2)+)
\end{equation} 
where we set 
$c_0=\mu
-f(x_0^{q^2}y_0^q,
x_0^qy_0^{q^2})$ with some $\mu \in 
\mu_{q-1}(\mathcal{O}_F)\cup \{0\}.$
Hence, by dividing 
(\ref{aw9}) by $w^q,$ the following congruence
holds
\begin{equation}\label{aw10}
a^q+a=
\zeta z_1^{q+1}-c_0\ ({\rm mod}\ 0+).
\end{equation}
Hence, the reduction $X^i_j$ is defined by the 
following equation
$a^q+a=\zeta z_1^{q+1}
-c_0.$
%Finally, we choose elements $a_0$ and $\zeta_1$
%such that 
%$a_0^q-a_0=-\frac{c_0}{y_0^{q+1}}$ and
%$\zeta_1^{q+1}=\zeta.$
%Then, by setting $a:=a_0+b$ 
%and $s_1:=\zeta_1z_1/y_0^q,$ the reduction $X^i_j$ is defined by 
%$b^q-b=s_1^{q+1}.$
%Therefore, we acquire the following proposition.
Therefore, we have proved the
 following Proposition \ref{fou-1}.
We set
\[
\mathcal{S}:=\{((\zeta,\tilde{\mu}),(x_0,y_0)) \in  
(\mathbb{F}^{\times}_q \times \mathbb{F}_q) \times (k^{\rm ac})^2\ |\ 
x_0^qy_0-x_0y_0^q=\zeta,\ x_0^{q^2-1}=y_0^{q^2-1}=-1\}.
\]
For an element
 $((\zeta,\tilde{\mu}),(x_0,y_0)) \in \mathcal{S},$ 
we set 
$i:=(\zeta,\tilde{\mu})$
and $j:=(x_0,y_0) \in
 \mathcal{S}^i_{00}.$
\begin{proposition}\label{fou-1}
Let the notation be as above.
Then, in the stable reduction of the Lubin-Tate space 
$\mathcal{X}(\pi^2),$ there exist $q(q-1) 
\times q(q^2-1)$ irreducible components 
$\{X^{i,c}_j\}_{(i,j) \in \mathcal{S}}$ such that
each $X^{i,c}_j$ has an affine model with an equation
$a^q+a=\zeta z_1^{q+1}-c$ with some 
$c \in \mathbb{F}_q.$
These components attach to the connected component of 
$\overline{\mathbf{Y}}_{2,2}$
at $q(q^2-1)$ singular points. 
\end{proposition}

\subsection{Computation of the reduction of 
the space $\mathbf{Z}_{1,1,\ast}\ 
(\ast=0,\infty)$}\label{zo1}
In this subsection, we calculate the reduction 
of the space $\mathbf{Z}_{1,1,\ast}$ with 
$\ast =0,\infty$ by using 
Lemma \ref{all}.1.
First, we check 
that the reduction
$\overline
{\mathbf{Z}}_{1,1,0}$
has $q(q-1)$ 
connected components.
Secondly, we show that
each component is defined by 
the following equation
\[Z^q=X^{q^2-1}+X^{-(q^2-1)}\]
with genus $0.$
Since we have 
\[
(X_2,Y_2,u) \in \mathbf{Z}_{1,1,0} \Leftrightarrow 
(Y_2,X_2,u) \in \mathbf{Z}_{1,1,\infty} 
\]
by Definition \ref{fd}.5, 
the same things happen for 
$\mathbf{Z}_{1,1,\infty}.$

Let $(X_2,Y_2)$ be a
 Drinfeld $\pi^2$-basis of the universal formal 
$\mathcal{O}_F$-module 
$\mathcal{F}^{\rm univ}.$
We set 
$X_1:=[\pi]_u(X_2)$ and 
$Y_1:=[\pi]_u(Y_2).$
Recall that we 
have the 
followings on the space 
$(X_2,Y_2,u) \in \mathbf{Z}_{1,1,0}$
as in Definition \ref{fd}
\[v(u)=\frac{1}{2},\ 
v(X_1)=\frac{1}{2(q-1)},\ 
v(Y_1)=\frac{1}{2q(q-1)},\ 
v(X_2)=\frac{1}{2q^2(q-1)},\ 
v(Y_2)=\frac{1}{2q^3(q-1)}.
\]

We choose an element $\kappa_1$
such that
$\kappa_1^{2q^4(q-1)}=\pi$ with 
$v(\kappa_1)
=1/2q^4(q-1).$ We set 
$\kappa:=\kappa_1^q.$
We set $\gamma:=\kappa^{q(q-1)^2}$ with 
$v(\gamma)=(q-1)/2q^2.$
We write $\gamma^{\frac{1}{q^2(q-1)}}$
for $\kappa_1^{q-1}$.
We change variables as follows
$u=\kappa^{q^3(q-1)} u_0,\ 
X_1=\kappa^{q^3}x_1,\ 
Y_1=\kappa^{q^2}y_1,\ X_2=\kappa^q x$
and $Y_2=\kappa y.$
Considering $[\pi]_u(X_1)=[\pi]_u(Y_1)=0,$
we acquire the 
following 
congruence 
by using 
Lemma \ref{all}.1
\begin{equation}\label{aq1}
u_0 \equiv -\frac{1}{x_1^{q-1}} 
\equiv -y_1^{q(q-1)}
-\frac{\gamma^q}{y_1^{q-1}}\ 
({\rm mod}\ 1+).
\end{equation} 
By (\ref{aq1}), 
the following 
congruence
holds
\begin{equation}\label{aq4}
x_1 \equiv \frac{\zeta}{y_1^q}
\biggl(1+\frac{\gamma^q}{y_1^{q^2-1}}\biggr)\ 
({\rm mod}\ 1+)
\end{equation}
with $\zeta \in 
\mu_{q-1}(\mathcal{O}_F).$
By considering 
$X_1=[\pi]_u(X_2)$ and $Y_1=[\pi]_u(Y_2),$
we obtain 
the following 
congruences again by Lemma \ref{all}.1
\begin{equation}\label{aq2}
x_1 \equiv x^{q^2}+\gamma^q u_0 x^q
+\gamma^{2q+1}x\ ({\rm mod}\ 1+),\  
y_1 \equiv y^{q^2}
+\gamma^{q+1}u_0y^q\ ({\rm mod}\ ((q+1)/2q)+).
\end{equation}
We set
\begin{equation}\label{aq6}
\mathcal{Z}:=(xy^q)^q
-\gamma\biggl
(xy^{q^3}+\frac{\zeta}{y^{q(q^2-1)}}\biggr).
\end{equation}
By substituting (\ref{aq2}) 
to $(\ref{aq4}) \times y^{q^3}$, we acquire
 the following congruence by using (\ref{aq1})
%\begin{equation}\label{aq5}
%(xy^q)^{q^2}-\gamma^q\biggl(xy^{q^3}+\frac{\zeta}{y^{q(q^2-1)}}\biggr)^q
%-\gamma^{2q}\{(xy^q)^q-\gamma\biggl(xy^{q^3}+\frac{\zeta}{y^{q(q^2-1)}}\biggr)\}
%\equiv \zeta\ ({\rm mod}\ (1/2)+).
%\end{equation}
%Then, $\mathcal{Z}$ satisfies the following by (\ref{aq5})
\begin{equation}\label{olo}
\mathcal{Z}^q-\gamma^{2q}\mathcal{Z}
=\zeta\ ({\rm mod}\ 1+).
\end{equation}
We choose an element $\gamma_0$ such that
$\gamma_0^q-\gamma^{2q}\gamma_0=\zeta.$
Then, by (\ref{olo}), if we set
$\mathcal{Z}=\gamma_0-\gamma^{\frac{2q}{q-1}}\mu,$
we obtain $\mu^q \equiv \mu\ ({\rm mod}\ 0+).$
%Hence, the set of
% geometrically connected components of  
% $\mathbf{Z}_{1,1,0}$ is identified with 
% $(\bar{\zeta},\bar{\mu}\bar{\zeta}^{-1}) \in 
% \mathbb{F}^{\times}_q \times \mathbb{F}_q.$
Similarly as (\ref{fc}), 
we fix the following identification
\begin{equation}\label{fc2}
\pi_0(\mathbf{Z}_{1,1,0}) 
\simeq 
(\mathcal{O}_F/\pi^2)^{\times}
\simeq \mathbb{F}^{\times}_q
\times 
\mathbb{F}_q \in 
(\bar{\zeta},
\frac{\bar{\mu}}{\bar{\zeta}}).
\end{equation}
Now, we choose an element $\mu \in
 \mu_{q-1}(\mathcal{O}_F) \cup\{0\}.$
We have the following congruence 
by (\ref{aq6})
\begin{equation}\label{aq8}
\mathcal{Z}=(xy^q)^q-
\gamma\biggl(xy^{q^3}+\frac{\zeta}{y^{q(q^2-1)}}\biggr)
\equiv
\gamma_0+\gamma^{\frac{2q}{q-1}}\mu\ 
({\rm mod}\ (1/q)+).
\end{equation}
We choose an element $\tilde{\gamma}_0$
such that
$\tilde{\gamma}_0^q=\gamma_0+\gamma^{\frac{2q}{q-1}}\mu.$
Then, 
we set 
\begin{equation}\label{aq7}
xy^q=\tilde{\gamma}_0+\gamma^{1/q}Z_1.
\end{equation}
Substituting this to 
(\ref{aq8}),
 and dividing it by $\gamma,$
the following congruence holds
\begin{equation}\label{aq9}
Z_1^q \equiv xy^{q^3}+
\zeta y^{-q(q^2-1)}\ ({\rm mod}\ (1/2q^2)+).
\end{equation}
Substituting $x=\frac{\tilde{\gamma}_0
+\gamma^{1/q}Z_1}{y^q}$ (\ref{aq7}) to 
the right hand side of the congruence (\ref{aq9}),
we obtain the following congruence
\begin{equation}\label{aq10}
\bigl(Z_1-y^{q^2-1}\tilde{\gamma}_0^{1/q}-
\zeta y^{-(q^2-1)}\bigr)^q 
\equiv \gamma^{1/q}y^{q(q^2-1)}Z_1\ ({\rm mod}\ (1/2q^2)+).
\end{equation}
We introduce a new parameter $Z_2$ such that
\begin{equation}\label{aq11}
y^{q^2-1}\gamma^{1/q^2}Z_2
=Z_1-y^{q^2-1}
\tilde{\gamma}_0^{1/q}-\zeta y^{-(q^2-1)}.
\end{equation}
Substituting this to (\ref{aq10}),
 and dividing it by $\gamma^{1/q}y^{q(q^2-1)},$
we obtain the following
congruence
$Z_2^q \equiv Z_1\ ({\rm mod}\ (1/2q^3)+).$
By substituting this to (\ref{aq11}), 
the following congruence holds 
\begin{equation}\label{aq13}
Z_2^q \equiv \zeta(y^{q^2-1}+y^{-(q^2-1)})
+y^{q^2-1}\gamma^{1/q^2}Z_2\ ({\rm mod}\ (1/2q^3)+).
\end{equation}
Note that we have $\tilde{\gamma}_0^{1/q}
=\gamma_0^{1/q}=\zeta\ ({\rm mod}\ (1/2q^3)+)$.
Therefore, we have proved the following proposition.
\begin{proposition}
The reduction $\overline{\mathbf{Z}}_{1,1,0}$
and $\overline{\mathbf{Z}}_{1,1,\infty}$
has $q(q-1)$ connected components,
 and each component is defined by the following equation
\[Z^q=\zeta(y^{q^2-1}+y^{-(q^2-1)}).\]
\end{proposition}
\begin{proof}
The required assertion for $\overline{\mathbf{Z}}_{1,1,0}$
 follows from (\ref{aq13}).
 Switching the 
 roles of $X_2$ and $Y_2$, we obtain the required assertion
 for $\overline{\mathbf{Z}}_{1,1,\infty}$.
\end{proof}

\subsection{Analysis of the singular 
residue classes of $\mathbf{Z}_{1,1,\ast}\ 
(\ast=0,\infty)$}\label{zo2}
In this subsection, we analyze the singular 
residue classes in the space
$\mathbf{Z}_{1,1,\ast}$ 
with $\ast =0,\infty.$ 
We keep the same 
notation as in the previous subsection.
Let $\{\overline{\mathbf{Z}}^i_{1,1,\ast}\}
_{i \in \mathbb{F}_q^{\times} \times \mathbb{F}_q}$
denote the connected components of 
the reduction $\overline{\mathbf{Z}}_{1,1,\ast}.$
Recall that the reduction 
$\overline{\mathbf{Z}}^i_{1,1,\ast}$ 
is defined by
the following equation
\[Z^q=\zeta (y^{q^2-1}+y^{-(q^2-1)}).\]
This affine curve has singular points at $y 
\in \mu_{2(q^2-1)}.$
We set 
\[\mathcal{S}_0:=(\mathcal{O}_F/\pi^2)^{\times} 
\times \mu_{2(q^2-1)} \ni (i,j).\]
For $(i,j) \in \mathcal{S}_0,$ 
we denote by $\mathbf{W}^i_{0,j}$ 
(resp.\ $\mathbf{W}^i_{\infty,j}$)
the underlying affinoid of the 
singular residue class at $y=j$ 
of the space 
$\mathbf{Z}^i_{1,1,0}.$ (resp.\ 
$\mathbf{Z}^i_{1,1,\infty}.$)
Furthermore, for $\ast=0,\infty$,
  we denote by $W^i_{\ast,j}$
the reduction of the affinoid 
$\mathbf{W}^i_{\ast,j}.$

We choose elements 
$\tilde{\gamma}_1$ and 
$y_0$ such that
$\tilde{\gamma}_1^q=
\iota2\zeta \{1+\gamma^{1/q^2}
\bigl(\frac{\tilde{\gamma}_1}{\zeta}\bigr)\}^{1/2}$
and $y_0^{q^2-1}=
\iota/\{1+\gamma^{1/q^2}\bigl
(\frac{\tilde{\gamma}_1}{\zeta}\bigr)\}^{1/2}$
with $\iota \in \{\pm1\}.$
Set $w:=y_0^{q+1}\gamma^{\frac{1}{q^2(q-1)}}.$
Furthermore, we choose an element $w_1$
such that
$y_0^{q^2-3}(\zeta+\gamma^{1/q^2}\tilde{\gamma}_1)w_1^2=w^q.$
Then, we have $v(w)=1/2q^4$ and $v(w_1)=1/4q^3.$ 

We change variables as follows
\begin{equation}\label{st}
Z_2=\tilde{\gamma}_1+w a,\ y=y_0+w_1 y_1.
\end{equation}
Substituting them to (\ref{aq13}), we acquire the following
\[w^q(a^q-a) \equiv 
y_0^{q^2-3}(\zeta+\gamma^{1/q^2}\tilde{\gamma}_1)w_1^2 y_1^2\ 
({\rm mod}\ (1/2q^3)+).
\]
Hence, by dividing this by $w^q,$ we acquire the following 
$a^q-a=y_1^2\ ({\rm mod}\ 0+).$
Hence, the reduction $W^i_{0,j}$ is defined by $a^q-a=y_1^2.$
Therefore, we have proved the following proposition.
\begin{proposition}\label{fou0}
In the stable reduction of the Lubin-Tate space $\mathcal{X}(\pi^2),$
there exist $q(q-1) \times 2(q^2-1)$ 
irreducible components 
$\{W^{i,c}_{0,j}\}
_{(i,j) \in \mathcal{S}_0}$ (resp.\ 
$\{W^{i,c}_{\infty,j}\}_{(i,j) \in \mathcal{S}_0}$) 
with an affine model $a^q-a=s^2.$
For each $i \in (\mathcal{O}_F/\pi^2)^{\times},$ 
the components $\{W^i_{0,j}\}_{j \in \mu_{2(q^2-1)}}$
(resp.\  $\{W^i_{\infty,j}\}_{j \in \mu_{2(q^2-1)}}$)
 attach 
to the component
 $\overline{\mathbf{Z}}^i_{1,1,0}$
(resp.\ $\overline{\mathbf{Z}}^i_{1,1,\infty}$)
at $2(q^2-1)$ singular points.
\end{proposition}

\subsection{Computation of 
the reduction of 
the space $\mathbf{Z}_{1,1,c}$}\label{zo3}
In this subsection, 
we compute the reduction of the space $\mathbf{Z}_{1,1,c}.$
We prove that the space $\mathbf{Z}_{1,1,c}$
has $q(q-1)^2$ connected components,
 and ech component
is defined by
$Z^q=X^{q^2-1}+X^{-(q^2-1)}.$

Recall that we have the followings 
on the space $\mathbf{Z}_{1,1,c}$
as in Definition \ref{fd}
\[v(u)=\frac{1}{2},\ v(X_1)=v(Y_1)=\frac{1}{2q(q-1)},\ 
v(X_2)=v(Y_2)=\frac{1}{2q^3(q-1)}.\]
Let $\kappa$ and $\gamma^{1/q^2(q-1)}$ be 
as in subsection \ref{zo1}.
We change variables as follows
$u=\kappa^{q^3(q-1)}u_0,\ 
X_1=\kappa^{q^2}x_1,\ 
Y_1=\kappa^{q^2}y_1,\ 
X_2=\kappa x$ and 
$Y_2=\kappa y.$
We put $i:=(2q+1)/2q.$

By $[\pi]_u(X_1)=[\pi]_u(Y_1)=0$ 
and Lemma \ref{all}.1, 
we acquire the following congruence
\begin{equation}\label{aqq0}
u_0=-x_1^{q(q-1)}
-\frac{\gamma^q}{x_1^{q-1}} 
\equiv 
-y_1^{q(q-1)}
-\frac{\gamma^q}{y_1^{q-1}}\ 
({\rm mod}\ (3/2)+).
\end{equation}
Note that we have $v(\gamma^q)+i=3/2.$
Let $f(X,Y)$ be a polynomial 
such that $X^q-Y^q=(X-Y)^q+pf(X,Y).$
By considering $(\ref{aqq0}) \times (x_1y_1)^q$, 
we acquire the following congruence
\begin{equation}\label{aaq1}
x_1^qy_1-x_1y_1^q \equiv 
\zeta \gamma^{\frac{q}{q-1}}+
\biggl(\frac{p}{\gamma^q}\biggr)
f(x_1^qy_1,x_1y_1^q)\ ({\rm mod}\ i+)
\end{equation}
with some $\zeta \in \mu_{q-1}(\mathcal{O}_F).$
Note that we have $(x_1^qy_1-x_1y_1^q)|
f(x_1^qy_1,x_1y_1^q)$
 and $v(x_1^qy_1-x_1y_1^q)=1/2q>0.$
Furthermore, by $X_1=[\pi]_u(X_2), Y_1=[\pi]_u(Y_2)$
and Lemma \ref{all}.1, 
we obtain the following congruence
\begin{equation}\label{aqq2}
x_1 \equiv x^{q^2}
+\gamma^{q+1}u_0x^q
+\gamma^{\frac{2q^2+2q+1}{q}}x
-q \gamma
 u^2_0x^{q(q^2-q+1)} 
 ({\rm mod}\ i+).
\end{equation}
The same congruence as 
(\ref{aqq2}) holds for $(y,y_1).$
On the right hand side of (\ref{aaq1}), 
 we have the following congruence 
\begin{equation}\label{asoi1}
\biggl(\frac{p}{\gamma^q}\biggr)
f(x_1^qy_1,x_1y_1^q) \equiv 
\biggl(\frac{p}{\gamma^q}\biggr)
\{f(x^{q^3}y^{q^2},x^{q^2}y^{q^3})
%+\gamma^{q+1}u_0\bigl(\partial_X
%f(x^{q^3}y^{q^2},x^{q^2}y^{q^3})x^{q^3}y^q
%\partial_Yf(x^{q^3}y^{q^2},x^{q^2}y^{q^3})x^qy^{q^3}\bigr) 
-(q/p)\gamma^{q+1}u_0^2(xy)^{q^2(q-1)}(x^qy-xy^q)^q\}\ 
({\rm mod}\ i+).
\end{equation}
Note that, if $e_{F/\mathbb{Q}_p} \geq 2,$
this term (\ref{asoi1}) 
is congruent to zero modulo $i+.$ 
We set
\begin{equation}\label{aqq4}
\mathcal{Z}:=(x^qy-xy^q)^q
-\gamma^{\frac{q+1}{q}}(x^{q^3}y-xy^{q^3})
-(p/\gamma^q)^{1/q}f(x^{q^2}y^q,x^qy^{q^2}).
\end{equation}
On the other hand, 
we have the following congruence,  
by substituting (\ref{aqq0}) and (\ref{aqq2})
to $x_1^qy_1-x_1y_1^q,$
\begin{equation}\label{asoi2}
x_1^qy_1-x_1y_1^q \equiv 
\mathcal{Z}^q-\gamma^{2q+1}\mathcal{Z}
+(p/\gamma^q)
f(x^{q^3}y^{q^2},x^{q^2}y^{q^3})
-q\gamma u_0^2
(xy)^{q^2(q-1)}(x^qy-xy^q)^q\
 ({\rm mod}\ i+).
\end{equation}
Hence, (\ref{aaq1}) induces the following 
congruence by (\ref{asoi1}) and (\ref{asoi2}) 
\begin{equation}\label{olo2}
\mathcal{Z}^q
-\gamma^{2q+1}\mathcal{Z} 
\equiv \gamma^{\frac{q}{q-1}}\zeta\ 
({\rm mod}\ i+).
\end{equation}
We choose an element ${\gamma}_0$
such that
$\gamma_0^q-\gamma^{2q+1}\gamma_0
=\gamma^{\frac{q}{q-1}}\zeta.$
Clearly, we have $v(\gamma_0)=1/2q^2.$
Then, if we set $\mathcal{Z}
=\gamma_0+\gamma^{\frac{2q+1}{q-1}}\mu,$
we acquire $\mu^q \equiv \mu\ ({\rm mod}\ 0+)$ 
by (\ref{olo2}).
Hence, by (\ref{aqq4}) and
 $v(f(x^{q^2}y^q,x^qy^{q^2}))=1/2q^2>0$,
we obtain the following congruence
\begin{equation}\label{aqq6}
\mathcal{Z}=(x^qy-xy^q)^q
-\gamma^{\frac{q+1}{q}}(x^{q^3}y-xy^{q^3})
\equiv
\gamma_0+\gamma^{\frac{2q+1}{q-1}}\mu\ 
({\rm mod}\ \biggl(\frac{q+1}{2q^2}\biggr)+).
\end{equation}
We choose an element $\tilde{\gamma}_0$
such that $\tilde{\gamma}_0^q
=\gamma_0+\gamma^{\frac{2q+1}{q-1}}\mu$.
Then, we introduce a new parameter $Z$
 as follows
\begin{equation}\label{t0}
x^qy-xy^q=\tilde{\gamma}_0+
y^{q^2-1}\gamma^{\frac{q+1}{q^2}}
\tilde{\gamma}_0^{1/q}
+y^{q^2-1}\gamma^{\frac{q^2+q-1}{q^2(q-1)}} Z.
\end{equation}
By substituting (\ref{t0})
to the left hand side of the congruence (\ref{aqq6}), 
we obtain the following congruence 
in the same way as in \cite[(4.21), (4.25)]{T}
\begin{equation}\label{zw}
Z^q \equiv \zeta(y^{q^2-1}+y^{-(q^2-1)})
+\gamma_1y^{q^2-1}Z\ ({\rm mod}\ (1/2q^3)+).
\end{equation}
By (\ref{t0}), we have $x^{q-1} \equiv
 y^{q-1}\ ({\rm mod}\ 0+)$. Therefore, we acquire 
 $x \equiv \zeta_1 y\ ({\rm mod}\ 0+)$ with 
 some $\zeta_1 
 \in \mu_{q-1}(\mathcal{O}_F).$
 Hence, the reduction 
 $\overline{\mathbf{Z}}_{1,1,c}$ has 
 $q(q-1)^2$ connected components, which are parametrized
 by $(\bar{\zeta},(\bar{\mu},\bar{\zeta}_1)) \in 
 \mathbb{F}^{\times}_q \times 
 \mathbb{F}^2_q.$
 Furthermore, each component is defined by
 $Z^q=\bar{\zeta}(y^{q^2-1}+y^{-(q^2-1)})$ by (\ref{zw}).
Therefore, we obtain the following propositions.
\begin{proposition}
The space $\mathbf{Z}_{1,1,c}$
has $q(q-1)^2$ connected components and each component 
reduces to the following curve
\begin{equation}\label{cur}
Z^q=\zeta(y^{q^2-1}+y^{-(q^2-1)})
\end{equation}
with some $\zeta \in 
\mathbb{F}^{\times}_q.$
This affine curve has $2(q^2-1)$ singular points at 
$y \in \mu_{2(q^2-1)}.$
\end{proposition}
Let $i \in (\mathcal{O}_F/\pi^2)^{\times}.$
We write $\{\mathbf{Z}^i_{1,1,j}\}
_{j \in \mathbb{F}^{\times}_q}$
for the connected components of 
$\mathbf{Z}^i_{1,1,c}.$
For each $k \in \mu_{2(q^2-1)}$,
 let $\mathbf{W}^i_{j,k} 
 \subset \mathbf{Z}^i_{1,1,j}$
 denote the underlying affinoid of 
 the singular residue class 
 at $y=k.$ 
 Then, by the same computations as the ones 
 in subsection \ref{zo1},
 we acquire the following proposition.
\begin{proposition}\label{fou}
In the stable reduction of $\mathcal{X}(\pi^2)$,
there exist $q(q-1)^2 \times 2(q^2-1)$ 
irreducible components $\{W^{i,c}_{j,k}\}_{(i,j,k) \in 
(\mathcal{O}_F/\pi^2)^{\times} 
\times \mathbb{F}^{\times}_q 
\times \mu_{2(q^2-1)}},$
 with an affine model $a^q-a=s^2.$
For each $(i,j) \in 
(\mathcal{O}_F/\pi^2)^{\times} \times \mathbb{F}^{\times}_q,$  
the components $\{W^{i,c}_{j,k}\}
_{k \in \mu_{2(q^2-1)}}$ attach to the
 component $\overline{\mathbf{Z}}^i_{1,1,j}$
at each singular point.
\end{proposition}

%%%%%%%%%%%%%%%%%%%%%%%%%%%%%%%%%%%%%%%%%%%%%%%%%%%%%%

\subsection{Computation of
 the reduction of $\mathbf{Y}_{3,1,0}$}\label{j1}
In this subsection, under a technical assumption
$e_{F/\mathbb{Q}_p} \geq 2$, 
we compute 
the reduction of the space $\mathbf{Y}_{3,1,0}.$
The reduction $\overline{\mathbf{Y}}_{3,1,0}$
has $q(q-1)$ connected components,
 and each component is defined by 
 $x^qy-xy^q=\bar{\zeta}$
with some $\bar{\zeta} 
\in \mathbb{F}^{\times}_q.$
In the proceeding sections, 
we do
 not use results 
 in subsections \ref{j1} and \ref{j2}.

Recall that we have the following on the space $\mathbf{Y}_{3,1,0}$
in Definition \ref{fd}
\[v(u)=\frac{1}{q+1},\ v(X_1)=\frac{q}{q^2-1},\ v(Y_1)=\frac{1}{q(q^2-1)},\ 
v(X_2)=\frac{1}{q(q^2-1)},\ 
v(Y_2)=\frac{1}{q^3(q^2-1)}.\]
We choose an element $\kappa$ such that
$\kappa^{q^3(q^2-1)}=\pi.$
We set $\gamma:=\kappa^{q(q-1)(q^2-1)}.$
Then, we have $v(\kappa)=1/q^3(q^2-1)$ 
and $v(\gamma)=(q-1)/q^2.$
Furthermore, we set $\gamma^{\frac{1}{q(q-1)}}:=\kappa^{q^2-1}.$
We change variables as follows
$u=\kappa^{q^3(q-1)}u_0,\ X_1=\kappa^{q^4}x_1,\ Y_1=\kappa^{q^2}y_1,
X_2=\kappa^{q^2}x$ and 
$Y_2=\kappa y.$
From now on, we assume $e_{F/\mathbb{Q}_p} \geq 2.$
Then, we acquire the following congruences 
by $[\pi]_u(X_1)=[\pi]_u(Y_1)=0$ and 
Lemma \ref{all}.2
\begin{equation}\label{bak1}
u_0 \equiv -\frac{1}{x_1^{q-1}} \equiv 
-y_1^{q(q-1)}-\frac{\gamma^q}{y_1^{q-1}}\ 
({\rm mod}\ 1+).
\end{equation}
Therefore, we acquire the following congruence
\begin{equation}\label{bak2}
x_1y_1^q \equiv \zeta\biggl(
1+\frac{\gamma^q}{y_1^{q^2-1}}\biggr)\ 
({\rm mod}\ 1+)
\end{equation}
with some $\zeta \in \mu_{q-1}(\mathcal{O}_F).$
On the other hand, by $Y_1=[\pi]_u(Y_2)$,  
$X_1=[\pi]_u(X_2)$ and Lemma \ref{all}.2, 
we acquire the followings
\begin{equation}\label{bak3}
y_1 \equiv y^{q^2}+\gamma u_0y^q\ ({\rm mod}\ (1/q)+),\ 
x_1 \equiv x^{q^2}+u_0 x^q+\gamma^q x\ ({\rm mod}\ 1+).
\end{equation}
By substituting (\ref{bak3}) 
to the right hand side of (\ref{bak1}), 
we acquire the following
\begin{equation}\label{bak5}
u_0 \equiv
-y^{q^3(q-1)}-\gamma^q\biggl(y^{q^2(q^3-2q+1)}
+\frac{1}{y^{q^2(q-1)}}\biggr)
-\gamma^{q+1}y^{q^4-2q^3+q}\ ({\rm mod}\ 1+).
\end{equation}
We set 
\[\mathcal{Z}:=(x^qy^{q^2}-xy^{q^3})-\gamma 
\bigl(x^qy^{q^4-q^3+q}+(\zeta /y^{q(q^2-1)})\bigr).\]
Substituting (\ref{bak3}) and (\ref{bak5}) 
to the term $x_1y_1^q$ in the left hand side of 
the congruence (\ref{bak2}), we acquire the following
by $e_{F/\mathbb{Q}_p} \geq 2$
%\begin{equation}\label{bak6}
%\{(x^qy^{q^2}-xy^{q^3})-\gamma x^qy^{q^4-q^3+q}\}^q
%-\gamma^q \{(x^qy^{q^2}-xy^{q^3})-\gamma x^qy^{q^4-q^3+q}\}
%\equiv \zeta 
%\biggl(1+
%\frac{\gamma^q}{y^{q^2(q^2-1)}}
%-\frac{\gamma^{q+1}}{y^{q(q^2-1)}}\biggr)
%\end{equation}
%modulo $1+.$
\begin{equation}\label{bak7}
\mathcal{Z}^q-\gamma^q\mathcal{Z} 
\equiv \zeta\ ({\rm mod}\ 1+).
\end{equation}
We choose an element $\gamma_0$ 
such that $\gamma_0^q-\gamma^q\gamma_0=\zeta.$
Then, we set 
\begin{equation}\label{bakk1}
\mathcal{Z}=\gamma_0+\gamma^{\frac{q}{q-1}}c.
\end{equation}
By substituting this to the congruence 
(\ref{bak7}), and 
dividing it by $\gamma^{\frac{q^2}{q-1}},$ 
we acquire the following congruence
$c^q \equiv c\ ({\rm mod}\ 0+).$
Therefore, we obtain $\bar{c} \in \mathbb{F}_q.$
Hence, $\overline{\mathbf{Y}}_{3,1,0}$
has $q(q-1)$ connected components, 
which are parametrized by 
$(\bar{\zeta},\bar{c}) 
\in \mathbb{F}_q^{\times} \times \mathbb{F}_q$.
Furthermore, each component 
is defined by $x^qy^{q^2}-xy^{q^3}=\bar{\zeta},$
because we have $x^qy^{q^2}-xy^{q^3} 
\equiv \mathcal{Z} \equiv \gamma_0 \equiv 
\zeta\ ({\rm mod}\ 0+)$ by 
(\ref{bakk1}).
By setting as follows
$z:=\frac{xy^q-\bar{\zeta}(1+y^{q^2-1})}{y^{q^2+q-1}},$
we acquire $z^qy-zy^q=\bar{\zeta}.$
Hence, we have proved the following lemma.
\begin{lemma}
We assume $e_{F/\mathbb{Q}_p} \geq 2.$
Then, the reduction $\overline{\mathbf{Y}}_{3,1,0}$
has $q(q-1)$ connected components 
and each component is defined by 
$x^qy-xy^{q}=\bar{\zeta}$ with some $\bar{\zeta} \in 
\mathbb{F}^{\times}_q$.
\end{lemma}

%%%%%%%%%%%%%%%%%%%%%%%%%%%%%%%%
\subsection{Computation of the reduction 
of the space $\mathbf{Y}_{3,1,c}$}\label{j2}
In this subsection, we calculate the reduction
$\overline{\mathbf{Y}}_{3,1,c}.$
By a technical reason, only in this subsection,
we assume that the absolute 
ramification index 
$e_{F/\mathbb{Q}_p} \geq 3.$
Under this assumption, we show that 
the reduction $\overline{\mathbf{Y}}_{3,1,c}$
has $q(q-1)^2$ connected components,
 and each component is defined by
$x^{q}y-xy^{q}=\bar{\zeta}$
with some $\bar{\zeta} \in
 \mathbb{F}^{\times}_q.$
However, we believe that 
the same phenomenon happens when 
$e_{F/\mathbb{Q}_p} \leq 2.$

Recall that
we have on the space 
$\mathbf{Y}_{3,1,c}$
in Definition \ref{fd}
\[v(u)=\frac{1}{q+1},\ 
v(X_1)=v(Y_1)=\frac{1}{q(q^2-1)},\ 
v(X_2)=v(Y_2)=\frac{1}{q^3(q^2-1)}.
\]
%We choose an element $\kappa$ such that
%$\kappa^{q^3(q^2-1)}=\pi.$
%We set $\gamma:=\kappa^{q(q-1)(q^2-1)}.$
%Then, we have $v(\kappa)=1/q^3(q^2-1)$ and 
%$v(\gamma)=(q-1)/q^2.$
Let $\kappa$, $\gamma$ and 
$\gamma^{\frac{1}{q(q-1)}}$ be as in the previous subsection.
We change variables as follows
$u=\kappa^{q^3(q-1)}u_0,\ X_1=\kappa^{q^2}x_1,\ 
Y_1=\kappa^{q^2}y_1,\ 
X_2=\kappa x$ and $Y_2=\kappa y.$
We set $k=(q+1)/q.$ 
From now on, we assume $e_{F/\mathbb{Q}_p} \geq 3.$
Then, we acquire the following, 
by $[\pi]_u(X_1)=[\pi]_u(Y_1)=0$ and 
Lemma \ref{all}.3,  
\begin{equation}\label{rob1}
(1-\pi^{q-1})u_0 \equiv 
-x_1^{q(q-1)}-\frac{\gamma^q}{x_1^{q-1}} \equiv 
-y_1^{q(q-1)}-\frac{\gamma^q}{y_1^{q-1}}\ 
({\rm mod}\ 2+).
\end{equation}
Note that we have
 $\pi^{q-1}u_0 \equiv 0\ ({\rm mod}\ 2+)$
 on the left hand side of (\ref{rob1})
unless $q=3.$
By considering $(\ref{rob1}) \times (x_1y_1)^q$,
 we obtain $v(x_1^qy_1-x_1y_1^q)=1/2q.$
Furthermore, we acquire 
the following congruence
\begin{equation}\label{rob2}
x_1^qy_1-x_1y_1^q \equiv 
\zeta \gamma^{\frac{q}{q-1}}\ 
({\rm mod}\ k+)
\end{equation}
with some $\zeta \in 
\mu_{q-1}(\mathcal{O}_F)$.
On the other hand, we 
have the following congruences, 
by $X_1=[\pi]_u(X_2)$ 
and $Y_1=[\pi]_u(Y_2)$ and
 Lemma \ref{all}.3,  
\begin{equation}\label{rob3}
x_1 \equiv x^{q^2}+\gamma u_0 x^q+
\gamma^{\frac{q^2+q+1}{q}}x,\ 
y_1 \equiv 
y^{q^2}+\gamma u_0 y^q+
\gamma^{\frac{q^2+q+1}{q}}y\ ({\rm mod}\ k+).
\end{equation}
Substituting (\ref{rob3}) to the term 
$x_1^qy_1-x_1y_1^q$
in the left hand side of 
the congruence
(\ref{rob2}), we acquire 
the following congruence, 
by the assumption 
$e_{F/\mathbb{Q}_p} \geq 3$, 
\begin{equation}\label{rob5}
x_1^qy_1-x_1y_1^q
\equiv
(x^{q}y-xy^{q})^{q^2}
+\gamma u_0(x^{q^2}y-xy^{q^2})^q
+\gamma^{q+1}u_0^{q+1}(x^{q}y-xy^{q})^q
+\gamma^{\frac{q^2+q+1}{q}}(x^{q^3}y-xy^{q^3})
\end{equation}
$({\rm mod}\ k+).$
By (\ref{rob1}) and (\ref{rob3}), 
we obtain the following congruence
\begin{equation}\label{rob6}
u_0 \equiv -x^{q^3(q-1)}
+\gamma^q u_0^q x^{q^2(q-1)^2}
-\frac{\gamma^q}{x^{q^2(q-1)}}
-\frac{\gamma^{q+1}u_0}{x^{q(q^2-1)}}
\end{equation}
\begin{equation}\label{rob7}
 \equiv
 -y^{q^3(q-1)}
+\gamma^q u_0^q y^{q^2(q-1)^2}
-\frac{\gamma^q}{y^{q^2(q-1)}}
-\frac{\gamma^{q+1}u_0}{y^{q(q^2-1)}}
\ ({\rm mod}\ \biggl(\frac{q^2+1}{q^2}\biggr)+). 
\end{equation}
Hence, we acquire the following congruence 
on the right hand side of the congruence
(\ref{rob5})
\begin{equation}\label{rob8}
\gamma u_0(x^{q^2}y-xy^{q^2})^q
+\gamma^{q+1}u_0^{q+1}(x^{q}y-xy^{q})^q
\equiv
-\gamma(x^{q^3}y-xy^{q^3})^q
-\gamma^{q+1}(x^{q}y-xy^{q})^q
\end{equation}
modulo $k+.$
We set as follows
\begin{equation}\label{rob9}
\mathcal{Z}:=(x^qy-xy^q)^q
-\gamma^{1/q}(x^{q^3}y-xy^{q^3}).
\end{equation}
Then, by (\ref{rob2}), 
(\ref{rob5}) 
and (\ref{rob8}), the following congruence holds
\begin{equation}\label{rob11}
\zeta \gamma^{\frac{q}{q-1}} \equiv 
x_1^qy_1-x_1y_1^q \equiv 
\mathcal{Z}^q-\gamma^{q+1}\mathcal{Z}\ ({\rm mod}\ 
k+).
\end{equation}
%By (\ref{rob2}) and (\ref{rob10}), we acquire the following
%\begin{equation}\label{rob11}
%\mathcal{Z}^q-\gamma^{q+1}\mathcal{Z}=\zeta \gamma^{\frac{q}{q-1}}\ 
%({\rm mod}\ k+).
%\end{equation}
We choose an element $\gamma_0$ such that
$\gamma_0^q-\gamma^{q+1}\gamma_0=\zeta\gamma^{\frac{q}{q-1}}.$
Then, we have $v(\gamma_0)=1/q^2.$
If we set 
\begin{equation}\label{rob12}
\mathcal{Z}=\gamma_0+\gamma^{\frac{q+1}{q-1}}c,
\end{equation}
by (\ref{rob11}), $c$ satisfies $c^q \equiv 
c\ ({\rm mod}\ 0+).$ Therefore, we have $\bar{c} \in
 \mathbb{F}_q$.
Set $\mathcal{Z}_1:=x^qy-xy^q.$
Note that we have $v(\mathcal{Z}_1)=1/q^3$ by (\ref{rob12}).
Then, the congruence (\ref{rob12}) induces the following congruence
\begin{equation}\label{rob13}
\mathcal{Z}_1^q-\gamma^{1/q}y^{q(q^2-1)}
\mathcal{Z}_1=\zeta \gamma^{1/(q-1)}\ ({\rm mod}\ (1/q^2)+).
\end{equation}
We set as follows
\begin{equation}\label{rob14}
\mathcal{Z}_1=x^qy-xy^q=\gamma^{1/q(q-1)}z.
\end{equation}
By substituting (\ref{rob14}) to (\ref{rob13}) and dividing it by 
$\gamma^{1/(q-1)},$ we acquire the following 
\begin{equation}\label{rob15}
z^q-y^{q(q^2-1)}z=\zeta\ ({\rm mod}\ 0+).
\end{equation}
If we set $z:=\bar{\zeta}(1+y^{q^2-1})+y^{q^2+q-1}w_1,$
the curve (\ref{rob15}) is 
isomorphic to a curve 
defined by the following 
$yw_1^q-y^qw_1=\bar{\zeta},$
because the function $y\ 
({\rm mod}\ 0+)$ 
is an invertible function.
Furthermore, we have $(x/y)
 \in \mathbb{F}^{\times}_q,$
because we have 
$\mathcal{Z}_1=x^qy-xy^q=0$
 modulo $0+$ by (\ref{rob14}).
Therefore, the reduction $\overline{\mathbf{Y}}_{3,1,c}$
has $q(q-1)^2$ connected components by (\ref{rob12})
and each component is defined by 
$x^qy-xy^q=\bar{\zeta}$
with some $\bar{\zeta} \in \mathbb{F}^{\times}_q.$
Hence, we have proved the following lemma.
\begin{lemma}
We assume $e_{F/\mathbb{Q}_p} \geq 3.$
Then, 
the reduction of the space $\mathbf{Y}_{3,1,c}$
has $q(q-1)^2$ connected components 
and each component is defined by $x^qy-xy^q=\bar{\zeta}$
with some $\bar{\zeta} \in 
\mathbb{F}^{\times}_q$.
\end{lemma}
%\begin{remark}
%If we do not assume $e_{F/\mathbb{Q}_p} \geq 2,$
%the congruences (\ref{rob1}), (\ref{rob3}) are not correct.
%In case $e_{F/\mathbb{Q}_p}=1,$ for example, we have the following congruences
%$$x_1 \equiv x^{q^2}+\gamma u_0 x^q+\gamma^{\frac{q^2+q+1}{q}}x
%+u_0\frac{(x^{q^3}+\gamma^qu_0^qx^{q^2})-x_1^q}{\gamma^q}\ ({\rm mod}\ k+)$$
%and 
%$$(1-\pi^{q-1})u_0 \equiv -x_1^{q(q-1)}
%-\frac{\gamma^q}{x_1^{q-1}}+\frac{u_0}{x_1^{q}}\{x_1^{q^3}-
%(x_1^{q^2}+\gamma^qx_1)^q\}\ ({\rm mod}\ \biggl(\frac{2q+1}{q}\biggr)+).$$
%Hence, the computation of the reduction $\overline{\mathbf{Y}}_{3,1,c}$ 
%in the case $e_{F/\mathbb{Q}_p}=1$
%seems to be much more complicated than the case $e_{F/\mathbb{Q}_p} \geq 2.$
%\end{remark}
%%%%%%%%%%%%%%%%%%%%%%%%%

\section{Action of the division 
algebra $\mathcal{O}^{\times}_D$ 
on the 
components in the stable reduction 
of $\mathcal{X}(\pi^2)$}\label{acd1}
Recall that we set $\mathcal{S}_0
=(\mathcal{O}_F/\pi^2)^{\times} \times 
\mathbb{P}^1(\mathbb{F}_q) 
\in (i,j)$ in \ref{zo2}.
Moreover, we set $
\mathcal{S}_1:=\mathcal{S}_0 \times  
\mu_{2(q^2-1)} \ni (i,j,k).$
In this section, we determine the right 
action of the division algebra
$\mathcal{O}^{\times}_D$
on the components $\overline{\mathbf{Y}}_{2,2}$, $X^{i}_j$ 
for $(i,j) \in \mathcal{S}$,
$\overline{\mathbf{Z}}_{1,1,\ast }\ (\ast =e_1,c)$, 
$W^i_{j,k}$ for $(i,j,k) \in \mathcal{S}_1$, 
which appear 
in the stable reduction of 
$\mathcal{X}(\pi^2).$ 
To compute the 
$\mathcal{O}_D^{\times}$-action, 
we use the description of 
the action of $\mathcal{O}^{\times}_D$ on the 
 Lubin-Tate space given in 
(\ref{dac1}). 
Then, the group $\mathcal{O}^{\times}_D$
acts on the stable reduction of 
$\mathcal{X}(\pi^2)$ by factoring through 
$\mathcal{O}^{\times}_3.$

First, we prepare some notations.
%We set $U_D^n:=1+(\varphi^n) \subset 
%\mathcal{O}^{\times}_D$, and
 %$\mathcal{O}^{\times}_n:
 %=\mathcal{O}^{\times}_D/U_D^n.$
The reduced norm ${\rm Nrd}_{D/F}:
\mathcal{O}^{\times}_D \to \mathcal{O}^{\times}_F$
induces ${\rm Nrd}_{D/F}:
\mathcal{O}^{\times}_3 \to 
(\mathcal{O}_F/\pi^2)^{\times}.$
For $i \in (\mathcal{O}_F/\pi^2)^{\times}$
and $b \in 
\mathcal{O}^{\times}_3,$
we set $ib:={\rm Nrd}^{-1}_{D/F}(b) \times i 
\in (\mathcal{O}_F/\pi^2)^{\times}.$
It is well-known that, under some 
identification
$\pi_0(\mathcal{X}(\pi^2))\simeq 
(\mathcal{O}_F/\pi^2)^{\times},$ 
 the group $\mathcal{O}^{\times}_3 \ni b$ acts on 
 $\pi_0(\mathcal{X}(\pi^2))$
 by 
$i \mapsto ib.$ 
See Theorem \ref{caq}.
For an element 
$b \in \mathcal{O}^{\times}_3,$
we write 
$b=a_0+\varphi b_0+\pi a_1$ with
$a_0 \in 
\mu_{q^2-1}(\mathcal{O}_E)$
and $b_0,a_1 
\in \mu_{q^2-1}(\mathcal{O}_E) \cup \{0\}.$
%The action of $\mathcal{O}^{\times}_D$
%on the components in the stable reduction 
%$\overline{\mathcal{X}(\pi^2)}$, which we will
%compute, is {\it right} action.

\subsection{The action of 
$\mathcal{O}^{\times}_D$ 
on the reduction 
$\overline{\mathbf{Y}}_{2,2}$ 
and 
$\{X^i_j\}_{(i,j) \in \mathcal{S}}$}
In this subsection, 
first, we write down 
the action of $\mathcal{O}_D^{\times}$
on $\pi_0(\mathbf{Y}_{2,2})$ 
in Lemma \ref{conn}.
Secondly, we determine 
the action of $\mathcal{O}^{\times}_D$
on each connected component 
$\overline{\mathbf{Y}}^i_{2,2}$ 
with $i \in 
(\mathcal{O}_F/\pi^2)^{\times}$ 
in Proposition \ref{bb2}.
Thirdly, we calculate the 
$\mathcal{O}^{\times}_3$-action on
the components 
$\{X^i_j\}_{(i,j) \in \mathcal{S}}$ 
with each defined by $X^q+X=Y^{q+1}$
in Proposition \ref{giu}.
In the following computations, we freely use
the notations 
 in subsections \ref{yo1} and \ref{yo2}.

Recall that we have the following 
on the space $\mathbf{Y}_{2,2}$
\[v(u) \geq \frac{q}{q+1},\ 
v(X_1)=v(Y_1)=\frac{1}{q^2-1},\ 
v(X_2)=v(Y_2)=\frac{1}{q^2(q^2-1)}.
\]
We choose an element 
$\kappa_1$ such that
$\kappa_1^{q^3(q^2-1)}=\pi$ 
with $v(\kappa_1)=1/q^3(q^2-1).$ 
We set $\kappa:=\kappa_1^q$
 and 
$\gamma:=\kappa^{(q-1)(q^2-1)}$
as in subsection \ref{yo2}.
We write $\gamma^{\frac{1}{q(q^2-1)}}$
for $\kappa_1^{q-1}.$
In subsection \ref{yo2},
we change variables as follows
$u=\kappa^{q^3(q-1)} u_0,\ 
X_1=\kappa^{q^2}x_1,\ 
Y_1=\kappa^{q^2}y_1,\ 
X_2=\kappa x$
and $Y_2=\kappa y.$ 

Recall that the reduction 
$\overline{\mathbf{Y}}_{2,2}$
has $q(q-1)$ connected components. 
Each component is defined by 
\[Z_1^q=x^{q^3}y-xy^{q^3},\ 
x^qy-xy^q=\bar{\zeta}\]
with some 
$\bar{\zeta} \in 
\mathbb{F}_q^{\times}.$
See subsection 3.2 for more details.
In the following,
 we determine the action of 
$\mathcal{O}^{\times}_D$ 
on the parameters $(x,y,Z_1)$.

Let $b=a_0+\varphi b_0+\pi a_1
 \in \mathcal{O}^{\times}_3$.
Recall the 
definition of $b^\ast $ in (\ref{diq}).
Then, we have 
the following by (\ref{cop})
and $j^{-1}(X) \equiv X\ ({\rm mod}\ (\pi,u))$
\begin{equation}\label{div2}
b^\ast (x) \equiv 
\frac{x-(b_0/a_0)^q x^q \gamma^{1/(q^2-1)}-
((a_1/a_0)-(b_0/a_0)^{q+1})x^{q^2}
\gamma^{1/(q-1)}}{a_0}\
 ({\rm mod}\ (1/q^2)+).\ 
%b^\ast (y) \equiv 
%\frac{\alpha^q y-\beta^q y^q \gamma_2^{1/q^2}}{\alpha^{q+1}}\ 
%({\rm mod}\ (1/q^2(q+1))+).
\end{equation} 
The same formula
 as (\ref{div2}) holds for $y.$
Recall that we have 
$x_1^qy_1-x_1y_1^q 
\in \mu_{q-1}(\mathcal{O}_F)$
modulo $1+$ as
proved in (\ref{aw1}).
We choose an element $\zeta \in 
\mu_{q-1}(\mathcal{O}_F).$
In the following, 
we assume that 
$x_1^qy_1-x_1y_1^q
=\zeta$ modulo $1+.$
We calculate 
$b^\ast (x_1)^q 
b^\ast (y_1)-b^\ast 
(x_1)b^\ast (y_1)^q$
in the following.
We acquire the following congruence, 
by $b^\ast (x_1) \equiv x_1/a_0$
and $b^\ast (y_1) \equiv y_1/a_0$ 
modulo $0+,$
\begin{equation}\label{div3}
b^\ast (x_1)^q b^\ast 
(y_1)-b^\ast (x_1)b^\ast (y_1)^q
\equiv \frac{x_1^qy_1
-x_1y_1^q}{a_0^{q+1}} 
\equiv \frac{\zeta}{a_0^{q+1}}\ ({\rm mod}\ 0+).
\end{equation}
%Therefore, $\mathcal{O}^{\times}_D$
%acts on the group $(\mathcal{O}_F/\pi \mathcal{O}_F)^{\times}$
%of the connected components of $\overline{\mathbf{Y}}_{1,1}$
%by ${\rm can.} \circ {\rm Nrd}^{-1}:\mathcal{O}_D^{\times} 
%\to \mathcal{O}^{\times}_F
%\to (\mathcal{O}_F/\pi \mathcal{O}_F)^{\times}.$
%This is a well-known fact. See subsection 2.5 for more details.

Recall that we set in (\ref{aw3'})
\begin{equation}\label{yui0}
\mathcal{Z}:=(x^qy-xy^q)^q-\gamma(x^{q^3}y-xy^{q^3}).
\end{equation}
This parameter $\mathcal{Z}$
satisfies $\mathcal{Z}^q-
\gamma^q \mathcal{Z} \equiv 
\zeta+\pi(f^q-f)\ 
({\rm mod}\ 1+)$
as in (\ref{aw4}).
We fix an element $\gamma_0$ such that 
$\gamma_0^q-\gamma^q\gamma_0=\zeta.$
We set
$\mathcal{Z}=\gamma_0
-\gamma^{\frac{q}{q-1}}c$ and 
$\mu:=c+f.$
 Then, 
we obtain 
$\mu^q \equiv \mu\ 
({\rm mod}\ 0+).$
Now, we have chosen
 the following identification in (\ref{fc})
\begin{equation}\label{fix}
\pi_0({\mathbf{Y}}_{2,2})
\simeq \mathbb{F}^{\times}_q 
\times \mathbb{F}_q \ni (\bar{\zeta},
\frac{\bar{\mu}}{\bar{\zeta}}).
\end{equation}
Then, we choose a pair 
$(\bar{\zeta},
\tilde{\mu}) 
\in (\mathcal{O}_F/\pi^2)^{\times}.$
We put 
\begin{equation}\label{yui1}
b^\ast (\mathcal{Z}):=(b^\ast (x)^qb^\ast (y)
-b^\ast (x)b^\ast (y)^q)^q
-\gamma(b^\ast (x)^{q^3}b^\ast (y)
-b^\ast (x)b^\ast (y)^{q^3}).
\end{equation}
 We acquire the following congruences 
 by using (\ref{div2})
\begin{equation}\label{div4}
b^\ast (x)^q b^\ast (y)-b^\ast (x)b^\ast (y)^q
\equiv
\frac{(x^qy-xy^q)-
(b_0/a_0) \gamma^{\frac{q}{q^2-1}}
(x^{q^2}y-xy^{q^2})
+\gamma^{\frac{1}{q-1}}
(a_1/a_0)(x^qy-xy^q)^q}{a_0^{q+1}}
\end{equation}
and 
\[b^\ast (x)^{q^3} 
b^\ast (y)-b^\ast (x)b^\ast (y)^{q^3}
\equiv \]\begin{equation}\label{div5}
\frac{x^{q^3}y-xy^{q^3}
-(b_0/a_0)^q(x^{q^2}y-xy^{q^2})^q\gamma^{\frac{1}{q^2-1}}
+\gamma^{\frac{1}{q-1}}
\{(b_0/a_0)^{q+1}-(a_1/a_0)\}(x^qy-xy^q)^{q^2}}
{a_0^{q+1}}\
\end{equation}
modulo $(1/q^2)+.$
%Note that we have $\gamma_2=\gamma \gamma_2^{1/q^2}$ and 
%$\gamma_2^{\frac{q+1}{q}}=\gamma^{\frac{q}{q-1}}.$
Hence, we obtain the following by (\ref{yui0}), 
(\ref{yui1}), 
(\ref{div4}) and (\ref{div5})
$$
b^\ast(\mathcal{Z})=
(b^\ast (x)^q b^\ast (y)
-b^\ast (x)b^\ast (y)^q)^q
-\gamma(b^\ast (x)^{q^3} 
b^\ast (y)-b^\ast (x)b^\ast (y)^{q^3}
) $$
\begin{equation}\label{div6}
\equiv
\frac{\mathcal{Z}}{a_0^{q+1}}+\gamma^{\frac{q}{q-1}}
\biggl(\frac{(a_1/a_0)^q+(a_1/a_0)-(b_0/a_0)^{q+1}}
{a_0^{q+1}}\biggr)(x^qy-xy^q)^{q^2}
\ ({\rm mod}\ (1/q)+).
\end{equation}
We set $f_b:=
f((b^\ast x)^{q^2}(b^\ast y)^q,
(b^\ast x)^q(b^\ast y)^{q^2}).$
Then, we obtain the following
$\pi f_b \equiv 
\pi 
\bigl(\frac{f}{a_0^{q+1}}\bigr)$ 
modulo $1+.$
Therefore, we acquire  
$(b^\ast \mathcal{Z})^q-
\gamma^q (b^\ast \mathcal{Z})=
\zeta/a_0^{q+1}
+\pi(f_b^q-f_b)\ ({\rm mod}\ 1+).$
We set
\[
h:=-(a_1/a_0)^q-(a_1/a_0)+(b_0/a_0)^{q+1}.
\]
Note that we have $\bar{h}
 \in \mathbb{F}_q.$
If we write 
$\mathcal{Z}=\gamma_0-\gamma^{q/(q-1)}c,$ 
we acquire the following
by (\ref{div6})
\begin{equation}\label{div7}
b^\ast \mathcal{Z}=
\frac{\gamma_0}{a_0^{q+1}}+
\frac{\gamma^{q/(q-1)}}{a_0^{q+1}}(
c+h\zeta)\ 
({\rm mod}\ (1/q)+).
\end{equation}
We write $b^\ast \mathcal{Z}
=(\gamma_0/a_0^{q+1})
-\gamma^{\frac{q}{q-1}}b^\ast c$
with $b^\ast c=(c+h\zeta)/a_0^{q+1}.$
Hence, we have proved the following lemma.
\begin{lemma}\label{conn}
Let $b=a_0+\varphi b_0+\pi a_1
 \in \mathcal{O}^{\times}_3$ 
 with $a_0 \in 
 \mu_{q^2-1}(\mathcal{O}_E)$
 and $b_0,a_1 \in 
 \mu_{q^2-1}(\mathcal{O}_
 E) 
 \cup \{0\}.$ We set
 $h:=-(a_1/a_0)^q-(a_1/a_0)+(b_0/a_0)^{q+1}.$
We consider the identification (\ref{fix}).
Then, the element $b$
 acts on the group 
$\pi_0(\overline{\mathbf{Y}}_{2,2})$
as follows
\[
b:\pi_0(\overline{\mathbf{Y}}_{2,2})
\to \pi_0(\overline{\mathbf{Y}}_{2,2})\ 
;\ i:=(\zeta,\tilde{\mu}) 
\mapsto 
ib:=(\zeta \bar{a}_0^{-(q+1)},
\tilde{\mu}+\bar{h}).
\]
\end{lemma}
\begin{proof}
The required assertion follows from (\ref{div7}).
\end{proof}
\begin{remark}The group 
$\mathcal{O}^{\times}_D$ acts on 
the set of connected components
of the Lubin-Tate space
according to the inverse 
of the reduced norm. See subsection 2.5
for more details.
Let $\sigma \neq 1 \in {\rm Gal}(E/F)$.
Then, 
we have ${\rm Nrd}_{D/F}(\alpha+\varphi \beta)
=\alpha\alpha^{\sigma}-\pi \beta\beta^{\sigma}.$
Let $b=a_0+\varphi b_0+\pi a_1 
\in \mathcal{O}_3^{\times}$ with $a_0 \in 
\mu_{q^2-1}(\mathcal{O}_E)$ and $b_0,a_1
 \in \mu_{q^2-1}(\mathcal{O}_E) \cup \{0\}$.
Then, we have
 ${\rm Nrd}_{D/F}^{-1}(b)=\frac{1}{a_0^{q+1}}\bigl(1+\pi 
 h\bigr) \in 
 \bigl(\mathcal{O}_F/\pi^2\mathcal{O}_F)^{\times}.$
 %We identify $(\mathcal{O}_F/\pi^2\mathcal{O}_F)^{\times}$
 %with $\mathbb{F}^{\times}_q \times \mathbb{F}_q$
 %by a map $a+b\pi \mapsto (\bar{a},\frac{\bar{b}}{\bar{a}}).$
 Hence, 
 as observed in 
 Lemma \ref{conn}, 
 $\mathcal{O}^{\times}_D$
 acts on $\pi_0(\overline{\mathbf{Y}}_{2,2})$ 
 according to
 the inverse of the reduced norm. 
\end{remark}

Secondly, we write down the action 
of $\mathcal{O}^{\times}_D$
 on the reduction 
 $\overline{\mathbf{Y}}_{2,2}.$ 
We write 
$\{\overline{\mathbf{Y}}^i_{2,2}\}_
{i=(\zeta,\tilde{\mu}) \in 
\mathbb{F}^{\times}_q 
\times \mathbb{F}_q}$
for the connected 
components of 
$\overline{\mathbf{Y}}_{2,2}.$
%We choose an element $\zeta
% \in \mu_{q-1}(\mathcal{O}_F).$
Now, we compute the action of $b$ 
on the component 
$\overline{\mathbf{Y}}^i_{2,2}$
with
$i=(\bar{\zeta},\tilde{\mu}) 
\in (\mathcal{O}_F/\pi^2)^{\times}.$

Let $\tilde{\gamma}_{0},
\tilde{\gamma'}_0$ be elements such that
$\tilde{\gamma}_{0}^q=\gamma_0$ and 
$\tilde{\gamma'}_0=\tilde{\gamma}_0a_0^{-(q+1)}.$
Recall that we introduce new parameters $Z_1$ and 
$b^\ast Z_1$ as in (\ref{aw6}),
\begin{equation}\label{div8}
x^qy-xy^q=\tilde{\gamma}_{0}+\gamma^{1/q}Z_1,\ 
b^\ast (x)^qb^\ast (y)-b^\ast (x) b^\ast (y)^q
=\tilde{\gamma'}_0+\gamma^{1/q}(b^\ast Z_1).
\end{equation} 
%Then, as in (\ref{aw7}), we obtain
%\begin{equation}\label{div8'}
%Z_1^q \equiv x^{q^3}y-xy^{q^3}
%-\gamma^{1/(q-1)}c\ ({\rm mod}\ (1/q^2)+).
%\end{equation}
%By 
%$(b^\ast \mathcal{Z})^q-\gamma^{q}(b^\ast \mathcal{Z})=
%\zeta a_0^{-(q+1)}
%+p(f_b^q-f_b)\ ({\rm mod}\ 1+)$
%and (\ref{yui1}),
%we obtain the following
%\begin{equation}\label{div8'''}
%(b^\ast Z_1)^q \equiv 
%b^\ast (x)^{q^3}b^\ast (y)-b^\ast (x)b^\ast (y)^{q^3}-
%\gamma^{\frac{1}{q-1}}b^\ast c\ ({\rm mod}\ (1/q^2)+).
%\end{equation}
As computed in subsection \ref{yo1},
 the components $\overline{\mathbf{Y}}^i_{2,2}$ and 
$\overline{\mathbf{Y}}^{ib}_{2,2}$ 
are defined by the following equations
respectively
\[x^qy-xy=\bar{\zeta},\ 
Z_1^q=x^{q^3}y-xy^{q^3},\]
\begin{equation}\label{rn1}
b^\ast (x)^qb^\ast (y)-b^\ast (x)b^\ast (y)^q
=\bar{\zeta} \bar{a}_0^{-(q+1)},\ 
(b^\ast Z_1)^q=b^\ast (x)^{q^3}b^\ast (y)-b^\ast (x) 
b^\ast (y)^{q^3}.
\end{equation}
\begin{proposition}\label{bb2}
Let $b \in 
\mathcal{O}^{\times}_D$.
We write $\bar{b}$ for the image of $b$
by $\mathcal{O}^{\times}_D \to
 \mathbb{F}^{\times}_{q^2}.$
We choose an element $i 
\in (\mathcal{O}_F/\pi^2)^{\times}.$
See (\ref{rn1}) for the 
defining equations of 
$\overline{\mathbf{Y}}^i_{2,2}$
and 
$\overline{\mathbf{Y}}^{ib}_{2,2}$.
Then, $b$ induces the 
following morphism
\[
b:\overline{\mathbf{Y}}^i_{2,2} \to
\overline{\mathbf{Y}}^{ib}_{2,2}\ ;\ 
(x,y,Z_1) \mapsto (\bar{b}^{-1}x,
\bar{b}^{-1}y,\bar{b}^{-(q+1)}Z_1).
\]
\end{proposition}
\begin{proof}
By (\ref{div2}) and (\ref{div8}),
we acquire $b^\ast Z_1
\equiv Z_1/\bar{b}^{q+1}$ modulo $0+.$
Hence, the required assertion follows.
\end{proof}

Thirdly, we write down the action of $\mathcal{O}^{\times}_D$
on the irreducible components 
$\{X^i_j\}_{(i=(\zeta,\tilde{\mu}),j) 
\in \mathcal{S}}$.

%%%%%%%%%%%%%%%%%%%%%%%%%%%
%We change variables as follows
%$z:=x/y,$ $t=1/y.$
%The congruence (\ref{div8'})
%has the following form under the variables $(Z_1,t)$, by (\ref{div2}), 
%\begin{equation}\label{div8''}
%Z_1^q \equiv 
%\frac{\{(z^q-z)^{q}+(z^q-z)\}^q+(z^q-z)}{t^{q^3+1}}+\gamma^{1/(q-1)}c 
%\equiv
%\tilde{\gamma}_0^q\frac{(t^{q^2-1}\tilde{\gamma}_0^{q-1}+1)^q}{t^{(q-1)(q^2-1)}}
%+\frac{\tilde{\gamma}_0+\gamma^{1/q}Z_1}{t^{q(q^2-1)}}+\gamma^{1/(q-1)}c 
%\end{equation}
%modulo $1/q^2+.$
Let $y_0$ be an element such that
$y_0^{q^2-1}+\tilde{\gamma}_0^{q-1}=0.$
%%%%%%%%%%%%%%%%%%%%%%%%%%%%%%%%%
We choose an element $\tilde{\gamma}_1$
such that $\tilde{\gamma}_1^q
+\tilde{\gamma}_0
+\gamma^{1/q}\tilde{\gamma}_1=0.$
Set $\tilde{\gamma'}_1
:=\tilde{\gamma}_1/a_0^{q+1}.$
Let $x_0$ be an element such that 
$x_0^qy_0-x_0y_0^q=
\tilde{\gamma}_0+\gamma^{1/q}\tilde{\gamma}_1.$
We choose elements $w,w_1$ such that
$w=\gamma^{1/q(q-1)}$
and
$w_1=y_0^q\gamma^{1/(q^2-1)}.$
Then, we have $v(w)=1/q^3$ and 
$v(w_1)=1/q^2(q+1).$
%We set $b^\ast w:=\frac{w}{\alpha^{q+1}}$ and 
We set
$b^\ast (w_1)=w_1/a_0^q.$
Furthermore, for 
$j:=(\bar{x}_0,\bar{y}_0) 
\in \mathcal{S}^i_{00},$
we set  
$jb:=(\bar{a}_0^{-1}\bar{x}_0,\bar{a}_0^{-1}\bar{y}_0) 
\in \mathcal{S}^{ib}_{00}.$
Now, we determine 
a morphism $X^i_j \to 
X^{ib}_{jb}$ which is induced by $b$. 

We change variables as follows as in (\ref{ss})
\begin{itemize}
\item
$Z_1=\tilde{\gamma}_1+wa,\ 
y=y_0+w_1z_1$
\item
$b^\ast Z_1=\tilde{\gamma'}_1+w (b^\ast a),\ 
b^\ast y=\frac{y_0}{a_0}+w_1 (b^\ast z_1).$
\end{itemize}
See subsection \ref{yo2} for 
the definition of 
$c_0$.
We set $c'_0:=
(c_0+h\zeta) a_0^{-(q+1)}.$
Then, we acquire the following by (\ref{aw10})
\begin{equation}\label{ky}
a^q+a \equiv \zeta z_1^{q+1}-c_0,\ 
(b^\ast a)^q+b^\ast a 
\equiv a_0^{-(q+1)}\zeta(b^\ast z_1)^{q+1}-
c'_0\ ({\rm mod}\ 0+).
\end{equation}
These equations define the components 
$X^i_j$ and $X^{ib}_{jb}$ respectively.
%We have the following
%\begin{equation}\label{go1}
%b^\ast (x)^q b^\ast (y)-b^\ast (x)b^\ast (y)^q \equiv
%-\tilde{\gamma'}_1^q+\gamma^{1/q}w b^\ast (a)\ ({\rm mod}\ (1/q^2)+).
%\end{equation}
On the term 
$x^{q^2}y-xy^{q^2}$ 
in the right hand side of the congruence (\ref{div4}), 
we acquire the following congruence
\begin{equation}\label{div9}
x^{q^2}y-xy^{q^2} \equiv 
%\tilde{\gamma}_0
%\biggl(\frac{\tilde{\gamma}_0^{q-1}t^{q^2-1}+1}{t^{q(q-1)}}\biggr)
\zeta\biggl(\frac{y^{q^2-1}+1}{y^{q-1}}\biggr)
\equiv \frac{\zeta}{y_0^q}w_1z_1\ ({\rm mod}\ (1/q^2(q+1))+).
\end{equation}
Therefore, we acquire the following by (\ref{div4})
\begin{equation}\label{go2}
b^\ast (x)^q b^\ast (y)-b^\ast (x)b^\ast (y)^q \equiv
\frac{-\tilde{\gamma}_1^q+\gamma^{1/q}w a
-\gamma_2^{1/q}(b_0/a_0)(w_1\zeta z_1/y_0^q)
+\gamma_2^{\frac{q+1}{q^2}}(a_1/a_0)\zeta}{a_0^{q+1}}\ 
\end{equation}
modulo $1/q^2+.$
On the other hand, 
we have the following by (\ref{div8})
\begin{equation}\label{go0}
b^\ast (x)^q b^\ast (y)-b^\ast (x)b^\ast (y)^q=
\tilde{\gamma'}_0+\gamma^{1/q}b^\ast Z_1=
-\tilde{\gamma'}_1^q+\gamma^{1/q} w (b^\ast a).
\end{equation}
Hence, by (\ref{go2}) and (\ref{go0}), 
we acquire the following
congruence
\begin{equation}\label{go3}
b^\ast a \equiv 
\frac{a-(b_0/a_0)
\zeta z_1+(a_1/a_0)\zeta}
{a_0^{q+1}}\ ({\rm mod}\ 0+).
\end{equation}

On the other hand, by considering 
the definitions of $z_1$ and $b^\ast z_1$,
 we obtain the 
following congruence by (\ref{div2})
\begin{equation}\label{div11}
b^\ast z_1 \equiv
a_0^{q-1}\biggl(z_1
-\biggl(\frac{b_0}{a_0}\biggr)^q\biggr)\ ({\rm mod}\ 0+).
\end{equation}
%(Then, we can check directly the following 
%congruence by the relationship
%$c'_0=\frac{c_0+h\zeta}{a_0^{q+1}}
%,$ 
%(\ref{go3}) and (\ref{div11})
%\[
%(b^\ast a)^q+(b^\ast a) \equiv 
%\frac{\zeta}{a_0^{q+1}}(b^\ast z_1)^{q+1}-
%c'_0\ ({\rm mod}\ 0+).)
%\]

%Let $b=\alpha \in \mathcal{O}^{\times}_D.$
In the following proposition, we describe 
the action of $b$ on the irreducible components
$\{X^i_j\}_
{(i,j) \in \mathcal{S}}$
in subsection 3.3.
%We set $s_1:=z_1/y_0^{q}$ and 
%$b^\ast s_1:=\alpha^q b^\ast z_1/y_0^q.$
%Then, by (\ref{ky}), 
%the components $X^i_j$ and $X^{b(i)}_{b(j)}$
%are defined by
%$a^q+a=\zeta s_1^{q+1}+\frac{c_0}{y_0^{q+1}}$ and 
%$(b^\ast a)^q-(b^\ast a)=
%\frac{\zeta}{\alpha^{q+1}}(b^\ast s_1)^{q+1}-\frac{\alpha^{q+1}c'_0}{y_0^{q+1}}.$

\begin{proposition}\label{giu}
Let $b=a_0+\varphi b_0+\pi a_1 \in \mathcal{O}^{\times}_3.$
We choose elements $i=(\zeta,\tilde{\mu}) \in 
(\mathcal{O}_F/\pi^2)^{\times}$ and 
 $j=(x_0,y_0) \in \mathcal{S}^i_{00}.$ 
We set $jb:=(x_0\bar{a}_0^{-1},y_0\bar{a}_0^{-1}) 
\in \mathcal{S}^{ib}_{00}.$
See (\ref{ky}) for the defining equations
of $X^i_j$ and $X^{ib}_{jb}$.
Then, the element $b$ induces the
 following morphism
\[
b:X^i_{j} \to 
X^{ib}_{jb}\ ;\ 
(a,z_1) \mapsto 
\biggl(\bar{a}_0^{-(q+1)}
(a-(\bar{b}_0/\bar{a}_0)
\zeta z_1
+(\bar{a}_1/\bar{a}_0)\zeta),
\bar{a}_0^{q-1}(z_1
-(\bar{b}_0/\bar{a}_0)^q)\biggr).
\]
\end{proposition}
\begin{proof}
The required assertion follows from 
(\ref{go3}) and (\ref{div11}).
\end{proof}

%Let $b=\alpha \in \mathcal{O}^{\times}_D$
%with $\alpha^{q+1}=1.$
%Then, $b$ fixes each connected component of $\overline{\mathcal{X}(\pi^2)}.$ 
%We acquire the following corollary as a special case of Proposition \ref{giu}.
%\begin{corollary}
%Let the notation be as above.
%Let $b=\alpha \in \mathcal{O}^{\times}_D$ with $\alpha^{q+1}=1.$
%Then, the element $b$ acts on the irreducible components $\{X^i_j\}_
%{(i=(\zeta,\mu),j) \in (\mathbb{F}^{\times}_q \times \mathbb{F}_q) \times \mathcal{S}^i_{00}}$
%as follows
%\[
%b:X^i_{j} \to
% X^i_{b(j)};\ 
%(a,s_1) \mapsto (a,\alpha^{2q-1} s_1).\]
%\end{corollary}

\subsection{The action of $\mathcal{O}_D^{\times}$ 
on the reduction $\overline{\mathbf{Z}}_{1,1,\ast}$
and $\{W^i_{\ast,j}\}_{ 
(i,j) \in \mathcal{S}_0}\ (\ast=0,\infty)$}
In this subsection, we 
compute the action of $\mathcal{O}^{\times}_D$
on the reduction $\overline{\mathbf{Z}}_{1,1,\ast}\ 
(\ast =0,\infty)$
and $\{W^i_{\ast,j}\}
_{\ast=0,\infty, 
(i,j) \in \mathcal{S}_0}$.
See Lemma \ref{op0} and Proposition \ref{op1}
for precise statements.
In the following, we use freely 
the notations 
in subsections \ref{zo1} and \ref{zo2}.
%As a result, the action of $\mathcal{O}^{\times}_D$ on
%$\overline{\mathbf{Z}}_{1,1,\infty}$ factors through 
%$\mathcal{O}^{\times}_3.$
But, we briefly recall 
them used in this subsection.
Recall that we have the followings 
on the space $\mathbf{Z}_{1,1,0}$
\[v(u)=\frac{1}{2},\ 
v(X_1)=\frac{1}{2(q-1)},\ 
v(X_2)=\frac{1}{2q^2(q-1)},\ 
v(Y_1)=\frac{1}{2q(q-1)},\ v(Y_2)=\frac{1}{2q^3(q-1)}.
\]
We choose an element 
$\kappa_1$
 such that 
 $\kappa_1^{2q^4(q-1)}=\pi$,
 We set $\kappa:=\kappa_1^q$ 
 and 
$\gamma:=\kappa^{q(q-1)^2}.$
We write $\gamma^{\frac{1}{q^2(q-1)}}$
for $\kappa_1^{q-1}.$
Then, we have $v(\kappa_1)=1/2q^4(q-1)$ 
and $v(\gamma)=(q-1)/2q^2.$
We change variables as follows
$u=\kappa^{q^3(q-1)}u_0,\ 
X_1=\kappa^{q^3}x_1,\ 
Y_1=\kappa^{q^2}y_1$
and $X_2=\kappa^{q}x,\ Y_2=\kappa y.$
We choose an element $\zeta 
\in \mu_{q-1}(\mathcal{O}_F).$
We set $r:=(q+1)/2q^2.$

Let $b=a_0+\varphi b_0+\pi a_1
 \in \mathcal{O}^{\times}_3.$
Recall the definition of 
$b^\ast $ in (\ref{diq}).
Then, we have 
the following congruences
by (\ref{cop})
and $j^{-1}(X) \equiv X\ ({\rm mod}\ (\pi,u))$
\begin{equation}\label{dz1}
b^\ast(x) \equiv 
\frac{x-(b_0/a_0)^q 
\gamma^{1/(q-1)}x^q+
(-(a_1/a_0)
+(b_0/a_0)^{q+1})
\gamma^{\frac{q+1}{q-1}}
x^{q^2}}{a_0}\ 
({\rm mod}\ r+),
\end{equation}
\begin{equation}\label{dz2} 
b^\ast (y) \equiv 
\frac{y
-(b_0/a_0)^q\gamma^{\frac{1}{q(q-1)}}y^q
+(-(a_1/a_0)+(b_0/a_0)^{q+1})
\gamma^{\frac{q+1}{q(q-1)}}y^{q^2}}
{a_0}\ 
({\rm mod}\ (r/q)+).
\end{equation}
Then, the congruences 
(\ref{dz1}) 
induce the following congruence
\begin{equation}\label{dz3}
b^\ast (x)b^\ast (y)^q \equiv
\frac{xy^q-((b_0/a_0)^q (xy)^q
+(b_0/a_0) 
xy^{q^2})\gamma^{\frac{1}{q-1}}
+(b_0/a_0)^{q+1}
\gamma^{\frac{2}{q-1}}
(xy^q)^q}{a_0^{q+1}}
\end{equation}
modulo $(1/q^2)+.$
Recall that we set as in (\ref{aq6})
\[\mathcal{Z}=(xy^q)^q-
\gamma
\biggl(xy^{q^3}
+\frac{\zeta}{y^{q(q^2-1)}}\biggr),
\]
\begin{equation}\label{22}
 b^\ast \mathcal{Z}
 =(b^\ast 
 (x)b^\ast (y)^q)^q
 -\gamma
\biggl
(b^\ast (x)b^\ast (y)^{q^3}
+\frac{\zeta}{a_0^{q+1}
b^\ast (y)^{q(q^2-1)}}\biggr).
\end{equation}
By (\ref{dz1}), we have 
the following congruences
 on the right hand side of the 
above equality (\ref{22})
\begin{equation}\label{dz5'}
b^\ast (x)b^\ast (y)^{q^3} \equiv 
\frac{xy^{q^3}-(b_0/a_0)^q 
\gamma^{1/(q-1)}(xy^{q^2})^q+
(-(a_1/a_0)
+(b_0/a_0)^{q+1})
\gamma^{\frac{q+1}{q-1}}
(xy^q)^{q^2}}{a_0^{q+1}}, 
\end{equation}
and 
\begin{equation}\label{dz5}
\frac{\zeta}{a_0^{q+1}b^\ast (y)^{q(q^2-1)}}
\equiv
\zeta\frac
{y^q-(b_0/a_0)\gamma^{1/(q-1)}y^{q^2}
+(-(a_1/a_0)^q
+(b_0/a_0)^{q+1})
\gamma^{\frac{q+1}{q-1}}y^{q^3}}
{a_0^{q+1} y^{q^3}}
\end{equation}
modulo $r+.$
By substituting 
the congruences (\ref{dz3}), 
(\ref{dz5'}) and (\ref{dz5})
to the right hand
 side of the equality (\ref{22}) 
above, $b^\ast \mathcal{Z}$ 
is congruent to the following 
\begin{equation}\label{dz60}
\frac{\mathcal{Z}}{a_0^{q+1}}
+\gamma^{\frac{q}{q-1}}\frac{b_0}{a_0^{q+2}}
\biggl(\frac{\zeta-xy^q}{y^{q-1}}\biggr)^{q^2}
+\gamma^{\frac{2q}{q-1}}
\biggl(
\biggl(\frac{b_0}{a_0^2}\biggr)^{q+1}
(xy^q)^{q^2}
+
\frac{(a_1/a_0)^q+(a_1/a_0)
-2(b_0/a_0)^{q+1}}
{a_0^{q+1}}\zeta
\biggr)
\end{equation}
modulo $(1/q)+.$
Since we set 
$xy^q=\tilde{\gamma}_0+\gamma^{1/q}Z_1$
 in (\ref{aq7}),
we have 
$(\zeta-xy^q)^{q^2} \equiv 0\ 
({\rm mod}\ (1/2q)+)$ by the choice of 
$\tilde{\gamma}_0$ in \ref{zo2}.
Note that we have $\tilde{\gamma}_0
 \equiv \zeta\ 
({\rm mod}\ 0+).$
Hence, (\ref{dz60}) has the following form
\begin{equation}\label{dz6}
b^\ast \mathcal{Z} \equiv
\frac{\mathcal{Z}}{a_0^{q+1}}
+\gamma^{\frac{2q}{q-1}}
\biggl(
\frac{(a_1/a_0)^q+(a_1/a_0)
-(b_0/a_0)^{q+1}}{a_0^{(q+1)}}\biggr)\zeta\ 
({\rm mod}\ (1/q)+).
\end{equation}
Recall that we have set 
$\mathcal{Z}^q
-\gamma^{2q}\mathcal{Z}
=\zeta\ ({\rm mod}\ 1+)$
in (\ref{olo}).
On the other hand, we have 
$(b^\ast \mathcal{Z})^q
-\gamma^{2q} b^\ast \mathcal{Z}
=\zeta/a_0^{q+1}\ ({\rm mod}\ 1+).$
We choose an element 
$\gamma_0$ such that $\gamma_0^{q}
-\gamma^{2q}\gamma_0=\zeta.$
Then, we set 
$\gamma'_0:=\gamma_0/a_0^{q+1}.$
Then, we set as follows
$\mathcal{Z}:=\gamma_0
-\gamma^{\frac{2q}{q-1}}\mu$
and $b^\ast \mathcal{Z}:=\gamma'_0
-\gamma^{\frac{2q}{q-1}}b^\ast \mu.$
Hence, by (\ref{dz6}), we acquire
$b^\ast \mu=\frac{\mu}{a_0^{q+1}}
-\bigl(
\frac{(a_1/a_0)^q+(a_1/a_0)
-(b_0/a_0)^{q+1}}
{a_0^{(q+1)}}\bigr)\zeta\ 
({\rm mod}\ 0+).$ 
Therefore,  $\mathcal{O}^{\times}_D$ 
acts on $\pi_0(\mathbf{Z}_{1,1,0})$
according to 
the inverse of the reduced norm, under the identification
(\ref{fc2}).

Now, fixing $i=(\bar{\zeta},\tilde{\mu}) 
\in (\mathcal{O}_F/\pi^2)^{\times}
=\pi_0(\mathbf{Z}_{1,1,0})$,
 we will determine the morphism 
$\overline{\mathbf{Z}}^i_{1,1,0} \to 
\overline{\mathbf{Z}}^{ib}_{1,1,0},$ 
which is induced by $b$ 
in Lemma \ref{op0} below.

Let $\tilde{\gamma}_0,\tilde{\gamma'}_0$
be elements such that
$\tilde{\gamma}_0^q=\gamma_0
-\gamma^{\frac{2q}{q-1}}\mu$ 
and $\tilde{\gamma'}_0^q
=\gamma'_0-\gamma^{\frac{2q}{q-1}}
b^\ast \mu.$
See subsection \ref{zo2} for more details.
By definitions of $\tilde{\gamma}_0$
 and $\tilde{\gamma'}_0$
and (\ref{dz6}), we obtain
$
v(\tilde{\gamma'}_0
-\tilde{\gamma}_0a_0^{-(q+1)}) \geq 1/q^2.
$
Recall that, in (\ref{aq7}), 
we set as follows
\begin{equation}\label{dz7}
xy^q=\tilde{\gamma}_0+\gamma^{1/q}Z_1,\ 
b^\ast (x)b^\ast (y)^q
=\tilde{\gamma'}_0+\gamma^{1/q}b^\ast Z_1.
\end{equation}
Furthermore, we put
\[h(x,y):=\frac{b_0^q a_0 
(xy)^q+a_0^qb_0 xy^{q^2}}{a_0^{2(q+1)}}.\]
Then, by (\ref{dz3}) 
 and (\ref{dz7}), 
we acquire the following congruence
\begin{equation}\label{dz9}
b^\ast Z_1 \equiv 
\frac{Z_1}{a_0^{q+1}}
-\gamma^{\frac{1}{q(q-1)}}h(x,y)\ ({\rm mod}\ (1/2q^3)+).
\end{equation}
%Similarly, we set
%\begin{equation}\label{dz11}

%\end{equation}
%Then, we obtain the following congruences as in (\ref{aq11})
%\[
%Z_2^q \equiv \zeta (y^{q^2-1}+y^{-(q^2-1)})
%+y^{q^2-1}\gamma^{1/q^2}Z_2,
%\]
%\[ 
%(b^\ast Z_2)^q \equiv \frac{\zeta}{a_0^{q+1}} 
%((b^\ast y)^{q^2-1}+(b^\ast y)^{-(q^2-1)})
%+(b^\ast y)^{q^2-1}
%\gamma^{1/q^2}b^\ast Z_2\ 
%({\rm mod}\ (1/2q^3)+).\]
Recall that we have
introduced new 
parameters $Z_2$ and $b^\ast Z_2$
as follows, as in (\ref{aq10}), 
\begin{itemize}
\item $\gamma^{1/q^2}
y^{q^2-1}Z_2=
Z_1-\zeta(y^{q^2-1}+y^{-(q^2-1)}),$
\item $\gamma^{1/q^2}
(b^\ast y)^{q^2-1}
b^\ast Z_2=
b^\ast Z_1-
\frac{\zeta}{a_0^{q+1}}
((b^\ast y)^{q^2-1}
+(b^\ast y)^{-(q^2-1)}).$
\end{itemize}
By considering (\ref{dz9}) 
and the definitions of $Z_2$ and 
$b^\ast Z_2$ above, 
we acquire the following congruence
\begin{equation}\label{dz12'}
b^\ast Z_2 \equiv
\frac{Z_2}{a_0^{q+1}}
-\gamma^{\frac{1}{q^2(q-1)}}
\biggl(\frac{h(x,y)}{y^{q^2-1}}
-\frac{b_0^q}{a_0^{2q+1}}(Z_1y^{-q(q-1)}
-2y^{-(q-1)(2q+1)})
\biggr)
\ ({\rm mod}\ (1/2q^4)+).
\end{equation}
We set as follows
\begin{equation}\label{h_1}
h_1(y)=\zeta \{(b_0^q/a_0^{2q+1})y^{-q(q-1)}
+(b_0/a_0^{q+2})y^{q(q-1)}\}.
\end{equation}
Since we have 
$Z_1 \equiv \zeta
(y^{q^2-1}
+y^{-(q^2-1)})\ 
({\rm mod}\ 0+),$
the congruence
(\ref{dz12'}) has the following form
\begin{equation}\label{dz12}
b^\ast Z_2 \equiv
\frac{Z_2}{a_0^{q+1}}
-\gamma^{1/q^2(q-1)}
h_1(y)
\ ({\rm mod}\ (1/2q^4)+).
\end{equation}

As computed in subsection \ref{zo1}, 
the components 
$\overline{\mathbf
{Z}}^i_{1,1,0}$ 
and $\overline
{\mathbf{Z}}^{ib}_{1,1,0}$ are 
defined by the 
following equations respectively 
\[
Z_2^q=\bar{\zeta}
 (y^{q^2-1}+y^{-(q^2-1)}),\ 
(b^\ast Z_2)^q=\frac{\bar{\zeta}}
{\bar{a}_0^{q+1}} 
((b^\ast y)^{q^2-1}+(b^\ast y)^{-(q^2-1)}).
\]
Then, we have the following lemma.
\begin{lemma}\label{op0}
Let $b \in \mathcal{O}_D^{\times}$
 and $\bar{b}$
the image of $b$ by 
$\mathcal{O}^{\times}_D \to
 \mathbb{F}^{\times}_{q^2}.$
 We choose an element $i 
\in (\mathcal{O}_F/\pi^2)^{\times}.$
Let $\ast=0,\infty.$
Then, $b$ induces 
the following morphism 
\[
b:\overline{\mathbf{Z}}^i_{1,1,\ast}
 \to 
\overline
{\mathbf{Z}}^{ib}
_{1,1,\ast}\ ;\ 
(y,Z_2) 
\mapsto 
\biggl(\frac{y}{\bar{b}},
\frac{Z_2}
{\bar{b}^{q+1}}\biggr).
\]
\end{lemma}
\begin{proof}
The required assertion follows from
 (\ref{dz1}) and (\ref{dz12}) immediately. 
\end{proof}
In the following, 
we determine the action 
of $\mathcal{O}^{\times}_{D}$
on the components 
$\{W^{i'}_{j'}\}_{(i',j') \in \mathcal{S}_0}$
where 
each $W_{j'}^{i'}$ 
is defined by the 
Artin-Schreier equation $a^q-a=s^2.$

Let $\iota \in \{\pm1\}.$
Recall that we choose $\tilde{\gamma}_1$
such that
$\tilde{\gamma}_1=
\iota2\zeta\{1+\gamma^{1/q^2}
(\tilde{\gamma}_1/\zeta)\}^{1/2}.$
Similarly, we 
choose $b^\ast \tilde{\gamma}_1$
 such that
\begin{equation}\label{zd1}
b^\ast \tilde{\gamma}_1^q
=\iota\frac{2\zeta}{a_0^{q+1}}
\{1+\gamma^{1/q^2}
a_0^{q+1}\frac{b^\ast 
\tilde{\gamma}_1}{\zeta}\}^{1/2}.
\end{equation}
Let $y_0$ and $x_0$ be elements 
such that $y_0^{q^2-1}=\iota/\{1+
\gamma^{1/q^2}
\tilde{\gamma}_1\zeta^{-1}\}^{1/2}$ 
and
$x_0y_0^q=
\tilde{\gamma}_0
+\gamma^{1/q}\tilde{\gamma}_1.$
We set $w:=y_0^{q+1}
\kappa^{(q-1)/q}$
and $b^\ast w:=w/a_0^{q+1}.$
Then, we have 
$\bar{y}_0 \in \mu_{2(q^2-1)}.$
We set 
$\bar{y}_0b:=\bar{y}_0/\bar{a}_0 
\in \mu_{2(q^2-1)}.$
Now, we determine the 
morphism 
$W^i_{\bar{y}_0} 
\to W^{ib}_{\bar{y}_0b}$
which is induced by $b$ 
in Proposition \ref{op1}.

In (\ref{st}), we have 
changed variables  
as follows
\begin{equation}\label{zd2}
Z_2=\tilde{\gamma}_1+w a,\ 
b^\ast Z_2=b^\ast \tilde{\gamma}_1+b^\ast w (b^\ast a).
\end{equation}
Then, by (\ref{dz12}) and (\ref{zd2}), 
we acquire the following congruence
\begin{equation}\label{zd3}
b^\ast (a) \equiv a
-a_0^{q+1}y_0^{-(q+1)}h_1(y_0)\ ({\rm mod}\ 0+).
\end{equation}
%We have 
%\[
%a_0^{q+1}h_1(y_0) 
%\equiv 
%\{(b_0/a_0)^qy_0^{-q(q-1)}
%+(b_0/a_0)y_0^{q(q-1)}
%\}\zeta\ ({\rm mod}\ 0+).
%\] 
%By $x_0y_0^q=\zeta,$ we acquire the following congruence
%\begin{equation}\label{zd3'}
%a_0^{q+1}\frac{h_1(y_0) \equiv
%\biggl(\frac{b_0^q\zeta}{a_0^{q}y_0^2}\biggr)
%+\biggl(\frac{b_0^q\zeta}{a_0^{q}y_0^2}\biggr)^q
%={\rm Tr}_{\mathbb{F}_{q^2}/\mathbb{F}_q}
%\biggl(\frac{b_0}{a_0}\biggr)\zeta\ ({\rm mod}\ 0+).
%\end{equation}
We choose an element $w_1$ such that
$y_0^{q^2-3}
(\zeta+\gamma^{1/q^2}
\tilde{\gamma}_1)w_1^2=w^q.$
We set $b^\ast y_0:=y_0/a_0$ 
and $b^\ast w_1:=w_1/a_0.$
As in (\ref{st}), 
we change variables 
 as follows
\[
y=y_0+w_1y_1,\ 
b^\ast (y)=b^\ast y_0+b^\ast w_1
b^\ast (y_1).
\]
Then, 
the congruence (\ref{dz1}) 
induces the following congruence
\begin{equation}\label{zd6}
b^\ast (y_1) \equiv y_1\ ({\rm mod}\ 0+).
\end{equation}

As proved in subsection \ref{zo2}, 
the components $W_{0,j}^i$ 
and $W^{ib}_{0,jb}$ are defined by the following 
equations respectively 
\begin{equation}\label{w_e}
a^q-a=y_1^2,\ 
b^\ast(a)^q-b^\ast (a)=b^\ast (y_1)^2.
\end{equation}
Then, we acquire the following proposition.
\begin{proposition}\label{op1}
Let $b=a_0+\varphi b_0+\pi a_1 
\in \mathcal{O}^{\times}_3.$
We choose elements 
$i=(\bar{\zeta},\tilde{\mu})
 \in \mathbb{F}^{\times}_q \times \mathbb{F}_q$ and 
 $\bar{y}_0 \in \mu_{2(q^2-1)}.$
 We set
 $\bar{y}_0b:=\bar{y}_0/\bar{a}_0.$
 Let $\ast=0,\infty.$
 See (\ref{w_e}) for the defining 
 equations of 
 $W^{i}_{\ast,\bar{y}_0}$ and $W^{ib}_{\ast,
\bar{y}_0b}.$
Then, the element $b$ induces the
 following morphism
\[
b:W^{i}_{\ast,\bar{y}_0} \to W^{ib}_{\ast,
\bar{y}_0b}\ ;\ (a,y_1) \mapsto
 \biggl(a-{\rm Tr}_{\mathbb{F}_{q^2}/\mathbb{F}_q}
 \biggl(\frac{\bar{b}_0}{\bar{a}_0
 {\bar{y}_0}^{2q}}\biggr)\bar{\zeta},y_1\biggr).
 \] 
\end{proposition}
\begin{proof}
By (\ref{h_1}), we 
have the following 
$\bar{a}_0^{q+1}
\bar{y}_0^{-(q+1)}
h_1(\bar{y}_0) 
\equiv 
{\rm Tr}
_{\mathbb{F}_{q^2}/\mathbb{F}_q}
 \bigl(\bar{b}_0(\bar{a}_0
 {\bar{y}_0}^{2q})^{-1}\bigr)
 \bar{\zeta}$ in 
 $\mathbb{F}^{\times}_q.$
Hence, the required assertion follows from 
(\ref{zd3}) and (\ref{zd6}).
\end{proof}

\subsection{Action of 
$\mathcal{O}_D^{\times}$ 
on the components 
$\overline{\mathbf{Z}}_{1,1,c}$ and 
$\{W^i_{k,j}\}_{(i,j) \in \mathcal{S}_0,k 
\in \mathbb{F}^{\times}_q}$}
In this subsection,
we compute the action of $\mathcal{O}_D^{\times}$
on the reduction $\overline{\mathbf{Z}}_{1,1,c}$
 and $\{W^i_{k,j}\}_{(i,j) \in \mathcal{S}_0,k 
\in \mathbb{F}^{\times}_q}$
explicitly. 
Then, we obtain Propositions \ref{zd0}
and \ref{zd1} similar to Lemma \ref{op0}
and Proposition \ref{op1}.
In the following computations, 
we freely use the notations 
in subsection \ref{zo3}.
%The action
%f $\mathcal{O}^{\times}_D$
%on the reduction $\overline{\mathbf{Z}}_{1,1,c}$
%factors through $\mathcal{O}^{\times}_D/U_D^2$.

First, recall that we have 
the following on the space $\mathbf{Z}_{1,1,c}$
\[v(u)=\frac{1}{2},\ 
v(X_1)=v(Y_1)=\frac{1}{2q(q-1)},\ 
v(X_2)=v(Y_2)=\frac{1}{2q^3(q-1)}.
\]
We choose an element 
$\kappa_1$
 such that 
 $\kappa_1^{2q^4(q-1)}=\pi.$
We set $\kappa:=\kappa_1^q$ 
and 
$\gamma:
=\kappa^{q(q-1)^2}.$
We write $\gamma^{\frac{1}{q^2(q-1)}}$
for an element $\kappa_1^{q-1}.$
Then, we have
 $v(\kappa)=1/2q^3(q-1)$ 
and $v(\gamma)=(q-1)/2q^2.$
We change variables as follows
$u=\kappa^{q^3(q-1)}u_0,\ 
X_1=\kappa^{q^2}x_1,\ 
Y_1=\kappa^{q^2}y_1,\ 
X_2=\kappa x,\ 
Y_2=\kappa y.$
We set $m=(q^2+q+1)/2q^3.$
As in the previous subsection, we write
$b=a_0+\varphi b_0+\pi a_1 \in 
\mathcal{O}^{\times}_3$. 
Recall the definition of $b^\ast $ 
in (\ref{diq}).
Then, we acquire the following congruences
by (\ref{cop}) and $j^{-1}(X) \equiv X\ 
({\rm mod}\ (\pi,u))$
\begin{equation}\label{ggd1}
b^\ast (x) \equiv 
\frac{x-(b_0/a_0)^q \gamma^{1/q(q-1)}x^q
-((a_1/a_0)-(b_0/a_0)^{q+1})
\gamma^{\frac{q+1}{q(q-1)}}x^{q^2}
+c(b)\kappa^{q^3-1}x^{q^3}}
{a_0^{q+1}}\ ({\rm mod}\ m+)
\end{equation}
with some element $c(b) \in \mathcal{O}_{E}$.
The same congruence holds for $y.$
The congruence (\ref{ggd1})
induces the following congruence
\begin{equation}\label{ggd3}
b^\ast (x)^qb^\ast (y)-b^\ast (x) b^\ast (y)^q
\equiv
\frac{x^qy-xy^q-(b_0/a_0) 
\gamma^{1/(q-1)}(x^{q^2}y-xy^{q^2})
%+\beta^{q+1}\gamma^{\frac{q+1}{q(q-1)}}
%(x^qy-xy^q)^q
+(a_1/a_0)\gamma^{\frac{q+1}{q(q-1)}}(x^qy-xy^q)^q
}{a_0^{q+1}}
\end{equation}
modulo $(2q+1)/2q^3+.$
On the other hand, we acquire the following 
by (\ref{ggd1})
\[
b^\ast (x)^{q^3}b^\ast (y)-b^\ast (x)b^\ast (y)^{q^3} \equiv
\]
\begin{equation}\label{ggd4}
\frac{(x^{q^3}y-xy^{q^3})-(b_0/a_0)^q
\gamma^{1/q(q-1)}
(x^{q^2}y-xy^{q^2})^q-
((a_1/a_0)-(b_0/a_0)^{q+1})
\gamma^{\frac{q+1}{q(q-1)}}(x^qy-xy^q)^{q^2}}{a_0^{q+1}}
\end{equation} modulo 
$m+.$
We set as follows
\begin{equation}\label{gdd1}
\mathcal{Z}:=(x^qy-xy^q)^q-
\gamma^{\frac{q+1}{q}}(x^{q^3}y-xy^{q^3})
\end{equation}
and
\begin{equation}\label{gdd1'}
b^\ast \mathcal{Z}:=((b^\ast x)^qb^\ast y-b^\ast x(b^\ast y)^q)^q-
\gamma^{\frac{q+1}{q}}((b^\ast x)^{q^3}b^\ast y-b^\ast x(b^\ast y)^{q^3}).
\end{equation}
Recall that we have 
$\mathcal{Z}^q-\gamma^{2q+1}\mathcal{Z} \equiv \gamma^{\frac{q}{q-1}}\zeta\ 
({\rm mod}\ \bigl(\frac{2q+1}{2q}\bigr)+).$
Hence, by (\ref{ggd3}) and (\ref{ggd4}), the following holds
\begin{equation}\label{gddd1}
b^\ast \mathcal{Z} \equiv \frac{\mathcal{Z}+
((a_1/a_0)^q+(a_1/a_0)
-(b_0/a_0)^{q+1})\gamma^{\frac{q+1}{q-1}}
(x^qy-xy^q)^{q^2}}{a_0^{q+1}}\ 
({\rm mod}\ \biggl(\frac{2q+1}{2q^2}\biggr)+).
\end{equation}
Therefore, we have
$(b^\ast \mathcal{Z})^q-\gamma^{2q+1}b^\ast \mathcal{Z} \equiv 
\gamma^{\frac{q}{q-1}}\frac{\zeta}{\alpha^{q+1}}\ 
({\rm mod}\ \bigl(\frac{2q+1}{2q}\bigr)+).$
Then, we have $v(b^\ast \mathcal{Z})=v(\mathcal{Z})=1/2q^2.$

We choose an element $\tilde{\gamma}_0$
such that
$\tilde{\gamma}_0^q-\gamma^{(2q+1)/q}\tilde{\gamma}_0
=\gamma^{1/(q-1)}\zeta.$
%and $\tilde{\gamma'}_0^q
%-\gamma^{(2q+1)/q}\tilde{\gamma'}_0=\gamma^{1/(q-1)}(\zeta/\alpha^{q+1}).$
We set $\tilde{\gamma'}_0=\tilde{\gamma}_0/a_0^{q+1}.$
We write  
$\tilde{\gamma}_0^q=\mathcal{Z}$ 
and $\tilde{\gamma'}_0^q=b^\ast \mathcal{Z}.$
We set 
\[\mathcal{Z}_1=x^qy-xy^q,\ 
b^\ast \mathcal{Z}_1=b^\ast (x)^qb^\ast (y)-b^\ast (x) b^\ast (y)^q.
\]
Note that 
we have $v(b^\ast \mathcal{Z}_1)=v(\mathcal{Z}_1)=1/2q^3$
by (\ref{gdd1}) and (\ref{gdd1'}), 
because of $v(\mathcal{Z})=v(b^\ast \mathcal{Z})=1/2q^2.$
We have the following
\[
x^{q^2}y-xy^{q^2}=\frac{\mathcal{Z}_1^q}{y^{q-1}}
+y^{q(q-1)}\mathcal{Z}_1 \equiv 
y^{q(q-1)}\mathcal{Z}_1\ ({\rm mod}\ (1/2q^3)+).
\]
Therefore, we acquire the following by (\ref{ggd3})
\begin{equation}\label{ggr1}
b^\ast \mathcal{Z}_1 \equiv 
\frac{(1
-(b_0/a_0) \gamma^{1/(q-1)}y^{q(q-1)})\mathcal{Z}_1}{a_0^{q+1}}\ 
({\rm mod}\ ((q+1)/2q^3)+).
\end{equation}

Furthermore, we introduce new parameters $Z$
and $b^\ast Z$ as follows
\begin{equation}\label{ty0}
\mathcal{Z}_1=\tilde{\gamma}_0+
y^{q^2-1}\gamma^{\frac{q+1}{q^2}}
\tilde{\gamma}_0^{1/q}
+y^{q^2-1}\gamma^{\frac{q^2+q-1}{q^2(q-1)}} Z
\end{equation}
and 
\begin{equation}\label{ty1}
b^\ast \mathcal{Z}_1=\tilde{\gamma'}_0+
(b^\ast y)^{q^2-1}\gamma^{\frac{q+1}{q^2}}
\tilde{\gamma'}_0^{1/q}
+(b^\ast y)^{q^2-1}\gamma^{\frac{q^2+q-1}{q^2(q-1)}}b^\ast Z.
\end{equation}
Then, by (\ref{ggd1}), on the right hand side of (\ref{ty1}), 
we acquire the following congruence
\[
\gamma^{(q+1)/q^2} \tilde{\gamma'}_0^{1/q} 
(b^\ast y)^{q^2-1}
\equiv 
\frac{\tilde{\gamma}_0\gamma^{(q+1)/q^2}y^{q^2-1}}
{a_0^{q+1}}
+\zeta \gamma^{\frac{q+1}{q(q-1)}}
%+\frac{\zeta \gamma^{\frac{q+1}{q(q-1)}}
\biggl(\frac{b_0^q}{a_0^{2q+1}}\biggr)
y^{q^2+q-2}\ 
({\rm mod}\ ((q+1)/2q^3)+).
\]
Hence, by (\ref{ggr1}), we obtain
\begin{equation}\label{noo1}
b^\ast Z \equiv
\frac{Z}{a_0^{q+1}}-
\gamma^{1/q^2(q-1)}\biggl(\frac{y^{q-1}b_0^q \zeta}{a_0^{2q+1}}+
\frac{b_0 \zeta}{a_0^{q+2}y^{q-1}}\biggr)\ ({\rm mod}\ (1/2q^4)+).
\end{equation}

%Let $\gamma_1$ be an element such that $\gamma_1^q=\gamma.$
%Then, by (\ref{ty0}) and (\ref{ty1}), we acquire 
% the followings
%\begin{enumerate} 
%\item $(b^\ast Z)^q \equiv
%\frac{\zeta}{a_0^{q+1}}
%(b^\ast (y)^{q^2-1}+b^\ast (y)^{-(q^2-1)})
%+\gamma_1 (b^\ast Z) b^\ast (y)^{q^2-1}\ ({\rm mod}\ (1/2q^3)+)
%,$
%\item
%$Z^q \equiv
%\zeta (y^{q^2-1}+y^{-(q^2-1)})
%+\gamma_1 Zy^{q^2-1}\ ({\rm mod}\ (1/2q^3)+).$
%\end{enumerate}
Since $\mathcal{Z}_1=x^qy-xy^q \equiv
 0\ ({\rm mod}\ 0+)$ by (\ref{ty0}),
we acquire $x=\tilde{\zeta} y$ 
with $\tilde{\zeta} \in \mathbb{F}^{\times}_q.$
Note that $\overline{\mathbf{Z}}_{1,1,c}$
has $q(q-1)^2$ connected components.
Let $\{\overline{\mathbf{Z}}^i_{1,1,j}\}_{(i,j) 
\in (\mathbb{F}^{\times}_q 
\times \mathbb{F}_q) 
\times \mathbb{F}^{\times}_q}$ denote the 
connected components of 
$\overline{\mathbf{Z}}_{1,1,c}.$
The group $\mathcal{O}^{\times}_D$
 acts on $j \in \mathbb{F}^{\times}_q$
trivially.
The group $\mathcal{O}^{\times}_D$
 acts on $i=(\zeta,\mu) 
\in \mathbb{F}^{\times}_q 
\times \mathbb{F}_q$ according
 to the inverse of the reduced norm.
Assume that we have 
$x=\tilde{\zeta}y$ with 
$\tilde{\zeta} \in \mathbb{F}^{\times}_q.$
As computed in subsection \ref{zo3}, 
the components
$\overline{\mathbf{Z}}^i_{1,1,j}$ 
and $\overline{\mathbf{Z}}^{ib}_{1,1,j}$
are defined by the following equations
\[
Z^q=\bar{\zeta}(y^{q^2-1}+y^{-(q^2-1)}),\ 
(b^\ast Z)^q=
\frac{\bar{\zeta}}{\bar{a}_0^{q+1}}
(y^{q^2-1}+y^{-(q^2-1)})
\]
respectively.
%Then, we acquire the following congruence by (\ref{ggd3}), (\ref{ty0}) and (\ref{ty1})
Therefore, we acquire the 
following proposition by (\ref{noo1}).
\begin{proposition}\label{zd0} 
Let $b \in \mathcal{O}^{\times}_D$ and 
$\bar{b}$ the image of $b$ by 
$\mathcal{O}^{\times}_D \to 
\mathbb{F}^{\times}_{q^2}.$
We choose
 elements $i 
\in (\mathcal{O}_F/\pi^2)^{\times}$
and $j \in \mathbb{F}^{\times}_q.$
Then, the element $b$ induces
 the following morphism
\[
b:\overline{\mathbf{Z}}^i_{1,1,j} \to
 \overline{\mathbf{Z}}^{ib}_{1,1,j}\ ;\ 
(y,Z) \mapsto \biggl(\frac{y}{\bar{b}},
\frac{Z}{\bar{b}^{q+1}}\biggr).
\]
\end{proposition}
Let $\{W^i_{j,k}\}_{i \in \mathbb{F}^{\times}_q 
\times \mathbb{F}_q,j \in \mathbb{F}^{\times}_q,
k \in \mu_{2(q^2-1)}}$ be the 
irreducible components defined by 
the Artin-Schreier equation
$a^q-a=s^2.$
In the same way as Proposition \ref{op1}, 
we acquire the following proposition
 by (\ref{ggd1}) and (\ref{noo1}).
\begin{proposition}\label{zd1}
Let $b=a_0+\varphi b_0+a_1 \pi 
\in \mathcal{O}^{\times}_3.$
We choose elements $j \in
 \mathbb{F}^{\times}_q$, 
$i=(\zeta,\tilde{\mu})
 \in \mathbb{F}^{\times}_q \times \mathbb{F}_q$
 and $\bar{y}_0 \in \mu_{2(q^2-1)}$. 
 We set $\bar{y}_0b:=\bar{y}_0/a_0.$
Then, the element $b$ induces 
the following morphism
\[
b:W^{i}_{j,\bar{y}_0} 
\to 
W^{ib}_{j,\bar{y}_0b}\ ;\ (a,y_1) \mapsto
 \biggl(a-{\rm Tr}_{\mathbb{F}_{q^2}/\mathbb{F}_q}
 \biggl(\frac{\bar{b}_0}{\bar{a}_0{\bar{y}_0}^{2q}}
 \biggr)\zeta,y_1\biggr).
 \]
\end{proposition}

\section{The action of $G_2^F$ on the components in the 
stable reduction of $\mathcal{X}(\pi^2)$}
In this section, we determine the right 
action of 
$G_2^F$ on the 
components $\overline{\mathbf{Y}}_{2,2}$, 
$X^i_j\ ((i,j) \in \mathcal{S}),$ 
$\overline{\mathbf{Z}}_{1,1,0}$, $W^i_{j,k}$
for $(i,j,k) \in \mathcal{S}_1.$
First, we prepare some notations used through
 this section.
For $i \in (\mathcal{O}_F/\pi^2)^{\times}$
and $g \in G_2^F$, we set 
$ig:={\rm det}(g) \times i$.
It is well-known that, under some 
identification
$\pi_0(\mathcal{X}(\pi^2)) \simeq 
(\mathcal{O}_F/\pi^2)^{\times},$ 
 the group $G_2^F \ni g$ acts on 
 $\pi_0(\mathcal{X}(\pi^2))$
 by 
$i \mapsto ig.$ See Theorem \ref{caq}. 
For an element $\alpha \in \mathcal{O}_F/\pi^2$,
we write $\alpha=a_0+a_1\pi$ with $a_0 \in 
\mu_{q-1}(\mathcal{O}_F), a_1 \in 
\mu_{q-1}(\mathcal{O}_F) \cup \{0\}.$
%The action of $G_2^F$ on the components 
%in the stable reduction 
%$\overline{\mathcal{X}(\pi^2)}$ is
% {\it right} action.

\subsection{Action of 
$G_2^F$ 
on $\overline{\mathbf{Y}}_{2,2}$
 and 
$\{X^i_j\}_{(i,j) \in \mathcal{S}}$
}\label{gl1}
In this subsection, for each 
$i \in 
(\mathcal{O}_F/\pi^2)^{\times},$  
we determine 
the action of $G_2^F$
on the components 
$\overline{\mathbf{Y}}^i_{2,2}$ 
and $\{X^i_j\}
_{j \in \mathcal{S}^i_{00}}$
explicitly.
See Lemma \ref{yoo1} and 
Proposition \ref{actt} for 
 precise statements.
In the following, we use freely the notations 
in subsections \ref{yo1} and \ref{yo2}.

%In the following, we determine the action of
%$G_2^F$ on the components $\overline{\mathbf{Y}}^i_{2,2}$ 
%and $\{X^i_j:a^q+a=t^{q+1}\}_{j \in \mathcal{S}^i_{00}}.$
Let $g:=\left(
\begin{array}{cc}
a_0+a_1 \pi & b_0+b_1\pi \\
c_0+c_1\pi & d_0+d_1\pi
\end{array}
\right) \in G_2^F.$
Let $(X_2,Y_2,u) \in \mathbf{Y}_{2,2}$
and $X_1:=[\pi]_u(X_2),Y_1:=[\pi]_u(Y_2).$
Then, the element $g$ acts on 
the Drinfeld basis $(X_2,Y_2)$ 
as follows
\[X_2 \mapsto [a_0]_u(X_2)+_u[c_0]_u(Y_2)
+_u[a_1]_u(X_1)+_u[c_1]_u(Y_1)\]
\[Y_2 \mapsto [b_0]_u(X_2)
+_u[d_0]_u(Y_2)+_u
[b_1]_u(X_1)+_u[d_1]_u(Y_1).\]
We choose an element 
$\kappa_1$ 
such that
$\kappa_1^{q^3(q^2-1)}=\pi$ with 
$v(\kappa_1)=1/q^3(q^2-1).$
We set $\kappa:=\kappa_1^q$ and 
$\gamma:=\kappa^{(q-1)(q^2-1)}$ with 
$v(\gamma)=(q-1)/q^2.$
We write 
$\gamma^{\frac{1}{q(q^2-1)}}$
for an element 
$\kappa_1^{q-1}.$
In subsection 2.3,
we change variables as follows
$u=\kappa^{q^3(q-1)} u_0,\ 
X_1=\kappa^{q^2}x_1,\ 
Y_1=\kappa^{q^2}y_1,\ X_2=\kappa x$
and $Y_2=\kappa y.$ 
For $a,b \in \mathcal{O}_F$, we set as follows
\[
f_{a,b}(X,Y):
=\frac{aX^{q^2}+bY^{q^2}-(aX+bY)^{q^2}}{\pi}
\in \mathcal{O}_F[X,Y].
\]
Then, we acquire the followings by 
(\ref{sos})
\[
g^\ast (x) \equiv a_0x+c_0y
+\kappa^{q^2-1}\{(a_1x+c_1y)^{q^2}
+f_{a_0,c_0}(x,y)
\}\ ({\rm mod}\ (1/q^2)+),
\]
\begin{equation}\label{act1}
g^\ast (y) \equiv b_0x+d_0y
+\kappa^{q^2-1}
\{(b_1x+d_1y)^{q^2}
+f_{b_0,d_0}(x,y)
\}\ ({\rm mod}\ (1/q^2)+).
\end{equation}
Note that if ${\rm char}\ F=p>0$, 
or ${\rm char}\ F=0$ 
and $e_{F/\mathbb{Q}_p} \geq 2,$
the terms $f_{\ast,\ast}(x,y)$ 
in (\ref{act1})
do not appear.
We set as follows
\[
g(x,y):=\{(a_0x+c_0y)(b_1x+d_1y)^q
-(a_1x+c_1y)^q(b_0x+d_0y)\}^q \in
 \mathcal{O}_F[x,y]
\]
and
 \[
h(x,y):=(a_0x+c_0y)^{q^3}(b_1x+d_1y)^{q^2}-
(a_1x+c_1y)^{q^2}(b_0x+d_0y)^{q^3} \in \mathcal{O}_F[x,y].
\]
Furthermore, we put
\[
G_0(x,y):=(a_0x+c_0y)^qf_{b_0,d_0}(x,y)
-(b_0x+d_0y)^qf_{a_0,c_0}(x,y) \in \mathcal{O}_F[x,y],
\] 
\[
G_1(x,y):=(a_0x+c_0y)^{q^3}f_{b_0,d_0}(x,y)
-(b_0x+d_0y)^{q^3}f_{a_0,c_0}(x,y),\ 
g_1(x,y):=G_0(x,y)^q-G_1(x,y)
\] 
\[
f_0(x,y):=\frac{x^{q^2}y^q-x^qy^{q^2}
-(x^qy-xy^q)^q}{\pi},\ 
{\rm det}(\bar{g}):=a_0d_0-b_0c_0,\ 
k_0:=\frac{{\rm det}
(\bar{g})-{\rm det}(\bar{g})^q}{p}.
\]
Then, 
we acquire the following congruence
by (\ref{act1})
\begin{equation}\label{act2}
g^\ast(x)^qg^\ast (y)-g^\ast (x)g^\ast (y)^q
 \equiv {\rm det}(\bar{g})(x^qy-xy^q)+\kappa^{q^2-1}
\{g(x,y)+G_0(x,y)\}\ ({\rm mod}\ (1/q^2)+).
\end{equation}
%We recall the computations in 
%subsection \ref{yo1} briefly.
We choose an element 
$\zeta \in \mu_{q-1}(\mathcal{O}_F).$
Recall that we set in (\ref{aw3'})
\[
\mathcal{Z}
:=(x^qy-xy^q)^q
-\gamma(x^{q^3}y-xy^{q^3}).\]
Then, we acquire the following
$\mathcal{Z}^q-\gamma^q\mathcal{Z}=\zeta+
p(f_0(x^q,y^q)^q-f_0(x^q,y^q))\
 ({\rm mod}\ 1+).$
We put $\mathcal{Z}
=\gamma_0-\gamma^{q/(q-1)}c.$
Furthermore, we set
$\mu:=
f_0(x^q,y^q)+c.$ 
Then, we obtain 
$\mu^q \equiv \mu\ 
({\rm mod}\ 0+).$
Similarly as above, 
we set 
\[
g^\ast \mathcal{Z}:=
(g^\ast (x)^qg^\ast (y)-g^\ast (x)g^\ast (y)^q)^q
-\gamma(g^\ast (x)^{q^3}g^\ast (y)-g^\ast (x)g^\ast (y)^{q^3})
\]
and $g^\ast \mathcal{Z}:={\rm det}(\bar{g})\gamma_0
-\gamma^{q/(q-1)}g^\ast c.$
Moreover, we put 
$g^\ast \mu:=
g^\ast c
+f_0((a_0x+c_0y)^q,
(b_0x+d_0y)^q)$.
Then, we obtain
$(g^\ast \mu)^q \equiv 
g^\ast \mu\ ({\rm mod}\ 0+).$
By (\ref{act1}) and (\ref{act2}),
we acquire the following
 by a direct computation
\begin{equation}\label{cnn0}
g^\ast \mathcal{Z}=
{\rm det}(\bar{g})\mathcal{Z}
+\gamma^{q/(q-1)}\{g(x,y)^q-h(x,y)
+g_1(x,y)
\}\ 
({\rm mod}\ (1/q)+).
\end{equation}
We set 
\[
\gamma(g):=-a_1d_0-a_0d_1+b_1c_0+b_0c_1.
\]
We can easily check that
\begin{equation}\label{c_x}
g(x,y)^q-h(x,y) \equiv \gamma(g)(x^qy-xy^q)^{q^2}
\equiv \gamma(g)\zeta\ ({\rm mod}\ 0+)
\end{equation}
and 
\begin{equation}\label{koq}
g_1(x,y)=f_0((a_0x+c_0y)^q,(b_0x+d_0y)^q)
-f_0(x^q,y^q)
-k_0(x^qy-xy^q)^{q^2}\ 
({\rm mod}\ 0+).
\end{equation}
Hence, we acquire by (\ref{cnn0}), 
(\ref{c_x}) and (\ref{koq})
%\begin{equation}\label{cn1}
%g^\ast c={\rm det}(\bar{g})c
%-\gamma(g)\zeta\ ({\rm mod}\ 0+).
%\end{equation}
%We set $g^\ast c=g^\ast \mu.$
%By (\ref{cn1}), we obtain
\begin{equation}\label{c_y}
g^\ast \mu \equiv 
{\rm det}(\bar{g})\mu-\{
\gamma(g)-k_0\}\zeta\ 
({\rm mod}\ 0+).
\end{equation}
\begin{lemma} 
W choose the identification
$\pi_0(\mathbf{Y}_{2,2}) \simeq 
(\mathcal{O}_F/\pi^2)^{\times}$ 
in (\ref{fc}).
Then, the group 
$G_2^F$ acts 
on the set 
$\pi_0(\mathbf{Y}_{2,2})$
by ${\rm det}:G_2^F \to 
(\mathcal{O}_F/\pi^2)^{\times}$.
\end{lemma}
\begin{proof}
The required assertion follows from 
(\ref{c_y}) 
and the identification (\ref{fc})
immediately.
\end{proof}

%If $g \in {\rm SL}_2(A)$, we have ${\rm det}(\bar{g})=1$ 
%and $\gamma(g)=0.$
%Hence, each connected component 
%is stable under the action of $g \in {\rm SL}_2(A).$

Now, fixing an element
 $i=(\bar{\zeta},
\tilde{\mu}) 
\in (\mathcal{O}_F/\pi^2)^{\times},$
we determine the morphism 
$\overline{\mathbf{Y}}^i_{2,2} \to 
\overline{\mathbf{Y}}^{ig}_{2,2}$ which is induced by $g$
in Lemma \ref{yoo1}.

We choose an element $\gamma_0$ such that
$\gamma_0^q-\gamma^q \gamma_0=\zeta.$ 
Furthermore, we choose an element
$\tilde{\gamma}_0$ such 
that $\tilde{\gamma}_0^q=\gamma_0.$ 
We recall that
we set in (\ref{aw6})
\[
x^qy-xy^q=\tilde{\gamma}_0+\gamma^{1/q}Z_1,\ 
g^\ast (x)^qg^\ast (y)-g^\ast (x)g^\ast (y)^q
={\rm det}(\bar{g})
\tilde{\gamma}_0+\gamma^{1/q}g^\ast Z_1.
\]
Then, by (\ref{act2}),
we obtain the following congruence
\begin{equation}\label{act2'}
g^\ast Z_1 \equiv 
{\rm det}(\bar{g})Z_1+\gamma^{1/q(q-1)}
(g(x,y)+G_0(x,y))\ 
({\rm mod}\ (1/q^3)+).
\end{equation}
Recall that the components 
$\overline{\mathbf{Y}}^i_{2,2}$
and $\overline{\mathbf{Y}}^{ig}_{2,2}$
are defined by the following 
equations respectively
\begin{itemize}
 \item $x^qy-xy^q=\bar{\zeta},\ 
 Z_1^q=x^{q^3}y-xy^{q^3},$
 \item $g^\ast(x)^qg^\ast(y)-g^\ast (x)g^\ast (y)^q=
 {\rm det}(\bar{g})\bar{\zeta},\ 
 (g^\ast Z_1)^q=g^\ast(x)^{q^3}
 g^\ast(y)-g^\ast (x)g^\ast (y)^{q^3}.$
\end{itemize}
Then, we have 
the following lemma. 
\begin{lemma}\label{yoo1}
Let $g=
\left(
\begin{array}{cc}
a_0+a_1\pi & b_0+b_1\pi \\
c_0+c_1\pi & d_0+d_1\pi
\end{array}
\right)
 \in G_2^F.$
We choose an element $i 
\in (\mathcal{O}_F/\pi^2)^{\times}.$
Then, $g$ induces 
the following 
morphism
\[
\overline{\mathbf{Y}}^i_{2,2} \to 
\overline{\mathbf{Y}}^{ig}_{2,2}\ :\ 
(x,y,Z_1) \mapsto 
(\bar{a}_0x+\bar{c}_0y,\bar{b}_0x+\bar{d}_0y,
\overline{{\rm det}(\bar{g})}Z_1).
\]
\end{lemma}
\begin{proof}
The required assertion follows from
 (\ref{act1}) and (\ref{act2'}) immediately.
\end{proof}
We choose an element $\tilde{\gamma}_1$
such that 
$\tilde{\gamma}_1^q+\tilde{\gamma}_0
+\gamma^{1/q}\tilde{\gamma}_1=0.$
Let $y_0$ be an element such that
$y_0^{q^2-1}+\tilde{\gamma}_0^{q-1}=0.$
We choose an element $x_0$ such that 
$x_0^qy_0-x_0y_0^q=\tilde{\gamma}_0
+\gamma^{1/q}\tilde{\gamma}_1.$ 
Recall that we set $w:=\gamma^{1/q(q-1)}$
and $w_1:=y_0^q\gamma^{1/(q^2-1)}.$
Then, we have $v(w)=1/q^3,\ 
v(w_1)=1/q^2(q+1).$
We set $j:=(\bar{x}_0,\bar{y}_0) \in 
\mathcal{S}^i_{00}$ and 
$jg:=(\bar{a}_0\bar{x}_0
+\bar{c}_0\bar{y}_0,
\bar{b}_0\bar{x}_0
+\bar{d}_0\bar{y}_0) 
\in \mathcal{S}^{ig}_{00}.$
Now, we determine 
the morphism 
$X^i_j 
\to X^{ig}_{jg}$, 
which is induced by $g$
in Proposition \ref{actt}
 below.

As 
in (\ref{ss}), 
we change variables 
$Z_1=\tilde{\gamma}_1+w a,\ 
y=y_0+w_1 z_1.$
%with $\tilde{\gamma}_1^q+\tilde{\gamma}_0
%+\gamma^{1/q} \tilde{\gamma}_1=0.$
%Let $x_0$ be an element such that
%$x_0^qy_0-x_0y_0^q=\tilde{\gamma}_0
%+\gamma^{1/q}\tilde{\gamma}_1.$ 
Then, we acquire the following defining equation of 
$X^i_j$
%with $i=(\bar{\zeta},\bar{\mu}_0) 
%\in 
%(\mathcal{O}_F/\pi^2)^{\times} 
%\simeq 
%\mathbb{F}^{\times}_q 
%\times \mathbb{F}_q$ and 
%$j=(\bar{x}_0,\bar{y}_0) 
%\in \mathcal{S}^{i}_{00}$ 
by (\ref{aw10})
\begin{equation}\label{act3}
a^q+a=\bar{\zeta} z_1^{q+1}
-\bar{c}_0 
\end{equation}
where $c_0:
=\mu-f(x_0^q,y_0^q).$
%where we have 
%$c_0=\mu_0-f_0(x_0^q,y_0^q)$
%with $\mu_0^q \equiv \mu_0\ 
%({\rm mod}\ 0+)$ when $v(p)=1.$
%When $v(p)>1,$  
%$\mu_0:=c_0$ 
%satisifes $\mu_0^q \equiv \mu_0\
% ({\rm mod}\ 0+).$ 
%We set
%as follows
%$g^\ast(x)^qg^\ast (y)-g^\ast (x)g^\ast (y)^q=
%\tilde{\gamma}_0+\gamma^{1/q}g^\ast Z_1.$
%Then, by (\ref{act2}), we acquire the following congruence
%$g^\ast Z_1 \equiv Z_1+\gamma^{1/q(q-1)}g(x,y)\ ({\rm mod}\ (1/q^3)+).$

Let $g^\ast (y_0):
=b_0x_0+d_0y_0$ 
and $g^\ast w_1:=
g^\ast (y_0)^qw_1.$
Furthermore, we set 
$g^\ast Z_1={\rm det}(\bar{g})
\tilde{\gamma}_1+w g^\ast (a)$
and $g^\ast (y)
:=g^\ast (y_0)+g^\ast w_1 g^\ast (z_1).$
Then, similarly as (\ref{act3}), 
we acquire 
the following defining 
equation of $X^{ig}_{jg}$
\begin{equation}\label{act3'}
g^\ast (a)^q+g^\ast (a)=
\overline{{\rm det}(\bar{g})}
\bar{\zeta} 
(g^\ast z_1)^{q+1}
-\overline{g^\ast c_0}.
\end{equation}
%where $g^\ast c_0=g^\ast \mu_0-
%f_0((a_0x_0+c_0y_0)^q,(b_0x_0+d_0y_0)^q)$ when
%$v(p)=1$
%and 
%$g^\ast c_0=g^\ast \mu_0$ with 
%$g^\ast \mu_0={\rm det}(\bar{g})
%\mu_0
%-(\gamma(g)-k_0)\zeta$.
Here, 
note that 
we have the following 
\[
\overline{g^\ast c_0} 
= 
\overline{{\rm det}(\bar{g})}\bar{c}_0
+\overline{{\rm det}(\bar{g})}
f_0(\bar{x}_0^q,\bar{y}_0^q)
-f_0((\bar{a}_0\bar{x}_0+\bar{c}_0\bar{y}_0)^q,
(\bar{b}_0\bar{x}_0+\bar{d}_0\bar{y}_0)^q)
-(\bar{\gamma(g)}-\bar{k}_0)\bar{\zeta}.
\]

By (\ref{act1}) and (\ref{act2'}), we acquire 
the following congruence
\begin{equation}\label{act4}
g^\ast (a) \equiv 
{\rm det}(\bar{g})a+g(x_0,y_0)
+G_0(x_0,y_0),\ 
g^\ast (z_1) \equiv z_1\ 
({\rm mod}\ 0+).
\end{equation}
%We choose elements $a_0$ and $\zeta_1$ 
%such that $a_0^q+a_0=c_0$ and 
%$\zeta_1^{q+1}=\zeta.$
%Finally, we change variables 
%as follows
%$a=a_0+a',\ 
%g^\ast (a)={\rm det}(\bar{g})a_0+g^\ast (a'),\ 
%s_1=\zeta_1\frac{z_1}{y_0^q}$
%and $g^\ast (s_1)=\zeta_1\frac{g^\ast z_1}{g^\ast (y_0)^q}.$
%Then, $X^i_j$ and $X^{ig}_{jg}$ are defined by the following 
%equations
%by (\ref{act3}) and (\ref{act3'})
%\[{a'}^q+a'=s_1^{q+1},\ 
%g^\ast (a')^q+g^\ast (a')=
%{\rm det}(\bar{g})g^\ast (s_1)^{q+1}+\gamma(g)\zeta.
%\]
%Furthermore,
%we acquire the followings by (\ref{act4}) and (\ref{act5})
%\begin{equation}\label{actoo}
%g^\ast (s_1)=\biggl(\frac{y_0}{g^\ast (y_0)}\biggr)^{q-1}s_1,\
% g^\ast (a')=a'+g(x_0,y_0)\ ({\rm mod}\ 0+).
%\end{equation}
Note that we have $g(x_0,y_0)
 \equiv -h(x_0,y_0)$
and hence $\gamma(g)\zeta=g(x_0,y_0)^q+g(x_0,y_0)$
modulo $0+.$
We also have $G_1(x_0,y_0)\equiv -G_0(x_0,y_0)$ and
 hence $G^q_0(x_0,y_0)+G_0(x_0,y_0)=-k_0\zeta+
 f_0((a_0x_0+c_0y_0)^q,
 (b_0x_0+d_0y_0)^q)-f_0(x_0^q,y_0^q)\ 
({\rm mod}\ 0+)$ by (\ref{koq}).
Hence, by (\ref{act4}),
 we easily check that 
$(g^\ast (a),g^\ast (z_1))$
 satisfies 
(\ref{act3'}).

As a result of the above computations,
 we acquire the following proposition.
\begin{proposition}\label{actt}
Let $g:=\left(
\begin{array}{cc}
a_0+a_1 \pi & b_0+b_1\pi \\
c_0+c_1\pi & d_0+d_1\pi
\end{array}
\right) \in G_2^F.$
Let $\bar{g}(x,y):=\{(\bar{a}_0x
+\bar{c}_0y)(\bar{b}_1x+\bar{d}_1y)^q
-(\bar{a}_1x+\bar{c}_1y)^q(\bar{b}_0x+\bar{d}_0y)\}^q
 \in \mathbb{F}_q[x,y].$
We define two polynomials
with coefficients 
in $\mathcal{O}_F$
\[f_{a,b}(X,Y):=\frac{ax^{q^2}+by^{q^2}-(ax+by)^{q^2}}{\pi},\ 
G_0(x,y):=(a_0x+c_0y)^qf_{b_0,d_0}(x,y)
-(b_0x+d_0y)^qf_{a_0,c_0}(x,y)
\]
for $a,b \in \mathcal{O}_F.$ 
% For $i \in (\mathcal{O}_F/\pi^2)^{\times},$
% we set $gi:={\rm det}(g) \times i.$
Let $j=(x_0,y_0) \in \mathcal{S}^i_{00}.$
We set
$jg:=(\bar{a}_0x_0+\bar{c}_0y_0,
\bar{b}_0x_0+\bar{d}_0y_0) 
\in \mathcal{S}^{ig}_{00}.$
See (\ref{act3}) and (\ref{act3'})
 for the defining equations 
of $X^i_j$ and $X^{ig}_{jg}.$
We set $f(j,g):=
\bar{g}(x_0,y_0)
+\overline{G}_0(x_0,y_0).$
%Set $g^\ast (y_0):=\bar{b}_0x_0+\bar{d}_0y_0.$
Then, $g$ induces the following morphism
\[g:X^i_j \to
X^{ig}_{jg}\ ;\ (a,z_1) \mapsto
\bigl(\overline{{\rm det}(\bar{g})}a+f(j,g), 
z_1
\bigr).\]
In particular, 
we have $\overline{G}_0(x,y)=0$ if ${\rm char}\ F>0$, 
or ${\rm char}\ F=0$ and $e_{F/\mathbb{Q}_p} \geq 2.$
\end{proposition}
\begin{proof}
The required assertion follows from (\ref{act4})
immediately.
\end{proof}
\begin{remark}
For $f(j,g),$ we have 
$f(j,g_1g_2)=\overline{{\rm det}
(\bar{g}_2)}f(j,g_1)+f(jg_1,g_2)$
for $g_1,g_2 \in G_2^F.$
\end{remark}

\subsection{Action of $G_2^F$ on 
 $\overline{\mathbf{Z}}_{1,1,0}$ 
 and $\{W^i_{0,j}\}_{(i,j) \in 
 \mathcal{S}_0}$}
In this subsection, we compute the action of 
$G_2^F$
on the components $\overline{\mathbf{Z}}_{1,1,0}$
and $W^i_{0,j}$ for ${(i,j) \in \mathcal{S}_0}$ 
explicitly in 
Lemma \ref{zoo1} and Proposition \ref{zg1}.
In the following computations, 
we freely use the notations in 
\ref{zo1} and \ref{zo2}.
%Set $A:=\mathcal{O}_F/\pi^2\mathcal{O}_F$
% $G:={\rm SL}_2(A).$
Recall that
we have on the space $\mathbf{Z}_{1,1,0}$
\[v(u)=\frac{1}{2},\ 
v(X_1)=\frac{1}{2(q-1)},\ 
v(X_2)=\frac{1}{2q^2(q-1)},\ 
v(Y_1)=\frac{1}{2q(q-1)},\ 
v(Y_2)=\frac{1}{2q^3(q-1)}.
\]
We choose an element $\kappa_1$ such that
$\kappa_1^{2q^4(q-1)}=\pi.$ We set 
$\kappa:=\kappa_1^q$ and
 $\gamma:=\kappa_1^{q^2(q-1)^2}.$
We have $v(\kappa_1)=1/2q^4(q-1)$
and $v(\gamma)=(q-1)/2q^2.$
We write $\gamma^{\frac{1}{q^2(q-1)}}$
for $\kappa_1^{q-1}.$
We change variables as follows
$u=\kappa^{q^3(q-1)}u_0,\ 
X_1=\kappa^{q^3}x_1,\ 
Y_1=\kappa^{q^2}y_1$
and $X_2=\kappa^{q}x,\
 Y_2=\kappa y.$
We choose an element $\zeta 
\in \mu_{q-1}(\mathcal{O}_F).$
%We choose an element $\gamma_0$
%such that 
%$\gamma_0^q-\gamma^{2q} \gamma_0=\zeta$.
See subsection \ref{zo2} for
 an element $\tilde{\gamma}_0.$

Let $B \subset G_1^F$ be a subgroup
 of upper triangular matrices.
Let  
$\overline{\mathfrak{U}}^{\times}
 \subset G_2^F$
 be the inverse image of 
$B$ by $G_2^F 
\to G_1^F.$
The stabilizer 
of $\overline{\mathbf{Z}}_{1,1,0}$
 in $G_2^F$ is equal to
$\overline{\mathfrak{U}}^{\times}.$

In the following, for an element 
$i=(\bar{\zeta},\tilde{\mu}) 
\in (\mathcal{O}_F/\pi^2)^{\times},$ we determine 
the morphism 
$\overline{\mathbf{Z}}^i_{1,1,0} \to 
\overline{\mathbf{Z}}^{ig}_{1,1,0}$
which is induced by $g \in
 \overline{\mathfrak{U}}^{\times}$
  in Lemma \ref{zoo1} below.

Let $\left(
\begin{array}{cc}
a_0+a_1\pi & b_0+b_1\pi \\
c_1\pi & d_0+d_1\pi
\end{array}
\right) \in \overline{\mathfrak{U}}^{\times}.
$
Then, we acquire the following congruences
\begin{equation}\label{zoi1}
g^\ast (x) \equiv a_0x+
c_1\gamma^{1/(q-1)}y^{q^2}\ 
({\rm mod}\ (1/2q^2)+),\ 
g^\ast (y) \equiv d_0y
+b_0\gamma^{1/q(q-1)}x\ 
({\rm mod}\ (1/2q^3)+).
\end{equation}
The following congruence 
holds by (\ref{zoi1})
\begin{equation}\label{zoi3}
g^\ast (x) g^\ast (y)^q
\equiv
a_0d_0xy^q+\gamma^{1/(q-1)}(c_1d_0y^{q(q+1)}
+a_0b_0x^{q+1})\ ({\rm mod}\ (1/2q^2)+).
\end{equation}
Recall that we have set 
 $xy^q=\tilde{\gamma}_0+\gamma^{1/q}Z_1$ and 
 $g^\ast (x)g^\ast (y)^q
 =a_0d_0\tilde{\gamma}_0+\gamma^{1/q}
 (g^\ast Z_1)$
in (\ref{aq7}). Moreover, we put
$h(x,y):=c_1d_0y^{q(q+1)}+a_0b_0x^{q+1}.$
Then, 
the congruence (\ref{zoi3}) 
induces the following congruence
\begin{equation}\label{zoi4}
g^\ast Z_1 \equiv 
a_0d_0Z_1+\gamma^{1/q(q-1)}h(x,y)\
 ({\rm mod}\ (1/2q^3)+).
\end{equation}
We recall the following equality 
in (\ref{aq11})
\begin{equation}\label{e_1}
y^{q^2-1}\gamma^{1/q^2}Z_2=Z_1-
\tilde{\gamma}_0^{1/q}y^{q^2-1}
-\zeta y^{-(q^2-1)}.
\end{equation}
We have a similar equality for 
$(g^\ast Z_1,g^\ast Z_2,
g^\ast (y)).$ 
By (\ref{zoi1}), 
 (\ref{zoi4}) and (\ref{e_1}), 
we obtain the following
\begin{equation}\label{zoi5}
g^\ast Z_2 \equiv a_0d_0Z_2
+\gamma^{1/q^2(q-1)}g(x,y)\ 
({\rm mod}\ (1/2q^4)+)
\end{equation}
where we set
\[g(x,y):=y^{-(q^2-1)}
\{h(x,y)+\frac{\zeta b_0x}{d_0y}
\bigl(y^{q^2-1}-
y^{-(q^2-1)}\bigr)\}.
\]
%Recall that the following congruences hold by (\ref{aq13})
%\[Z_2^q \equiv \zeta (y^{q^2-1}+y^{-(q^2-1)})
%+y^{q^2-1}\gamma^{1/q^2}Z_2, \]
%\begin{equation}\label{zoi6}
%(g^\ast Z_2)^q \equiv 
%a_0d_0\zeta ((g^\ast y)^{q^2-1}+(g^\ast y)^{-(q^2-1)})
%+(g^\ast y)^{q^2-1}\gamma^{1/q^2}g^\ast Z_2
%\ ({\rm mod}\ (1/2q^3)+).
%\end{equation}
As computed in \ref{zo1},
 the components 
$\overline{\mathbf{Z}}^i_{1,1,0}$
and $\overline{\mathbf{Z}}^{ig}_{1,1,0}$
are defined by
the following equations respectively 
\[Z_2^q=\bar{\zeta} (y^{q^2-1}+y^{-(q^2-1)}),\ 
(g^\ast Z_2)^q =
\bar{a}_0\bar{d}_0\bar{\zeta} 
((g^\ast y)^{q^2-1}+(g^\ast y)^{-(q^2-1)}).
\]
Then, we have the following lemma.
\begin{lemma}\label{zoo1}
Let $g=
\left(
\begin{array}{cc}
a_0+a_1\pi & b_0+b_1\pi \\
c_1\pi & d_0+d_1\pi
\end{array}
\right)
 \in \overline{\mathfrak{U}}^{\times}.$
 We choose an element $i 
\in (\mathcal{O}_F/\pi^2)^{\times}.$
Then, $g$ induces
 the following morphism
\[
\overline{\mathbf{Z}}^i_{1,1,0} \to 
\overline{\mathbf{Z}}^{ig}_{1,1,0}\ :\ 
(y,Z_2) \mapsto (\bar{d}_0y,\bar{a}_0\bar{d}_0Z_2).
\]
\end{lemma}
\begin{proof}
The required assertion 
follows 
from (\ref{zoi1}) and (\ref{zoi5}) 
immediately.
\end{proof}
We choose elements $\tilde{\gamma}_1,y_0$
such that
$\tilde{\gamma}^q_1=\iota 2\zeta 
\{1+\gamma^{1/q^2}(\tilde{\gamma}_1/\zeta)\}^{1/2}$
and
$y_0^{q^2-1}=\frac{\iota}
{(1+\gamma^{1/q^2}(\tilde{\gamma}_1/\zeta))^{1/2}}$
where $\iota \in \{\pm1\}.$
We set $w:=y_0^{q+1}\kappa^{(q-1)/q}$ and choose 
$w_1$ such that
$y_0^{q^2-3}(\zeta+\gamma^{1/q^2}
\tilde{\gamma}_1)w_1^2=w^q.$
Then, we have $v(w)=1/2q^4$ and $v(w_1)=1/4q^3.$
Set $x_0:=\tilde{\gamma}_0/y_0^q.$
We choose a $2$-th root $(a_0/d_0)^{1/2}$
of $a_0/d_0.$
Furthermore, we put $g^\ast w:=d_0^2w$ and 
$g^\ast w_1:=d_0(d_0/a_0)^{1/2}w_1.$
We have $\bar{y}_0 \in \mu_{2(q^2-1)}$.
Put $\bar{y}_0g:=\bar{d}_0\bar{y}_0 
\in \mu_{2(q^2-1)}.$
Now, we determine the morphism
 $W^i_{\bar{y}_0} \to W^{ig}_{\bar{y}_0g}$
 which is induced by $g$ in 
 Proposition \ref{zg1} below.

As in (\ref{st}), 
we change variables as follows
\[Z_2=\tilde{\gamma}_1+wa,\ y=y_0+w_1s_1.
\]
(resp.\ 
\[
g^\ast Z_2=a_0d_0\tilde{\gamma}_1+g^\ast w(g^\ast a),\ 
g^\ast (y)
=d_0y_0+g^\ast w_1 (g^\ast s_1).)
\]
%Similarly, we set $g^\ast (Z_2):=\tilde{\gamma}_1+g^\ast a.$
%Then, by (\ref{zoi5}), we obtain the following
%\begin{equation}
%g^\ast (a) \equiv a+\frac{g(x_0,y_0)}{y_0^{q+1}}\ ({\rm mod}\ 0+).
%\end{equation}
By (\ref{zoi1}) and (\ref{zoi5}),
the following congruences hold
\begin{equation}\label{zzoi}
g^\ast a \equiv 
\biggl(\frac{a_0}{d_0}\biggr)a+\frac{g(x_0,y_0)}{d_0^2y_0^{q+1}}\ 
({\rm mod}\ 0+),\ g^\ast s_1 \equiv 
(a_0/d_0)^{1/2}s_1\ ({\rm mod}\ 0+).
\end{equation}
Furthermore, we obtain the following 
$g(x_0,y_0) \equiv c_1d_0y_0^{q+1}
+a_0b_0y_0^{-(q+1)}\zeta^2\ ({\rm mod}\ 0+).$
Hence, we acquire the following
by (\ref{zzoi})
\begin{equation}\label{d_a}
g^\ast a \equiv
(a_0/d_0)
\bigl(a+(c_1/a_0)+(b_0/d_0)(\zeta/y_0^{q+1})^2\bigr),\ 
g^\ast s_1 \equiv (a_0/d_0)^{1/2}s_1\ ({\rm mod}\ 0+).
\end{equation}
As computed 
in subsection \ref{zo2}, 
the components 
$W^i_{0,\bar{y}_0}$ and $W^{ig}_{0,\bar{y}_0g}$ with 
$\bar{y}_0g=\bar{d}_0\bar{y}_0$ are defined by
the following 
equations 
\begin{equation}\label{zzoi1}
a^q-a=s^2,\ (g^\ast a)^q-g^\ast a=
(g^\ast s_1)^2
\end{equation} 
respectively.
Then, we have the following proposition.
\begin{proposition}\label{zg1}
Let $g=
\left(
\begin{array}{cc}
a_0+a_1\pi & b_0+b_1\pi \\
c_1\pi & d_0+d_1\pi
\end{array}
\right)
 \in \overline{\mathfrak{U}}^{\times}$. 
 We consider the 
 components 
 $\{W^{i'}_{0,\bar{y}'_0}\}_{(i',\bar{y}'_0)
  \in \mathcal{S}_0}.$
 %For $i \in (\mathcal{O}_F/\pi^2)^{\times},$
 %we set $gi:={\rm det}(g) \times i.$
  We choose elements $i=(\zeta,\tilde{u}) \in 
  \mathbb{F}^{\times}_q \times \mathbb{F}_q$
  and $\bar{y}_0 \in \mu_{2(q^2-1)}.$
  Furthermore, 
we set $\bar{y}_0g:=\bar{d}_0\bar{y}_0.$
See (\ref{zzoi1}) for the defining equations of
$W^i_{0,\bar{y}_0}$ and $W^{ig}_{0,\bar{y}_0g}$.
Then, $g$ induces the following morphism 
\[g:W^i_{0,\bar{y}_0} 
\to 
W^{ig}_{0,\bar{y}_0g}\ ;\ (a,s) \mapsto
\biggl(
(\bar{a}_0/\bar{d}_0)\bigl
(a+(\bar{c}_1/\bar{a}_0)
+(\bar{b}_0/\bar{d}_0)(\zeta/\bar{y}_0^{q+1})^2\bigr), 
 (\bar{a}_0/\bar{d}_0)^{1/2}s_1
\biggr).\]
\end{proposition}
\begin{proof}
The required assertion follows
from
(\ref{d_a}).
\end{proof}

\begin{remark}
The canonical map 
$G_2^F \to G_1^F$
induces the 
following
bijective 
$G_2^F/
\overline
{\mathfrak{U}}
^{\times} 
\simeq G_1^F/B.$
Let 
$g:=\left(
\begin{array}{cc}
a_0 &  b_0\\
c_0 & d_0
\end{array}
\right)
 \in G_1^F.$
Then, we consider 
the right 
 $G_1^F$-action on 
$\mathbb{P}^1
(\mathbb{F}_q) 
\in [x:y]$
by $[x:y] \mapsto 
[a_0x+c_0y:
b_0x+d_0y].$ 
The stabilizer of 
$[0:1]$ in $G_1^F$
 is 
equal to $B.$
Hence, we obtain bijectives 
$G_2^F/
\overline{\mathfrak{U}}
^{\times} \simeq G_1^F/B \simeq
 \mathbb{P}^1(\mathbb{F}_q)$
 by $[g] \mapsto 
 [0:1]g=[c_0:d_0].$
 We set
$\mathcal{S}_2:=(\mathcal{O}_F/\pi^2)
 ^{\times} \times
  \mu_{2(q^2-1)}.$ 
In the following, we consider 
the components $
\overline{\mathbf{Z}}^i_{1,1,\ast}$
and $\{W^i_
{\ast,y_0}\}
_{(i,y_0) \in 
\mathcal{S}_2}$ for $\ast 
\in \mathbb{P}^1(\mathbb{F}_q)$
in \ref{zo1}, \ref{zo2} and \ref{zo3}.
Then, $G_2^F$ acts on the index set 
$\ast \in \mathbb{P}^1(\mathbb{F}_q)$
by $\ast \mapsto \ast g.$
Clearly,
 this action is transitive.
As mentioned 
at the 
beginning 
of this subsection, for each 
$i \in (\mathcal{O}_F/\pi^2)^{\times}$,  
the stabilizer 
of the components 
$\overline{\mathbf{Z}}
^i_{1,1,0}$ and 
$\{W^{i,c}_{0,y_0}\}
_{(i,y_0) 
\in \mathcal{S}_2}$
in $G_2^F$ is equal to 
$\overline{\mathfrak{U}}^{\times}.$
Hence,  
 we acquire the 
following isomorphism
as a $G_2^F$
-representation
\begin{equation}
\label{ind}
\bigoplus
_{(i,j,y_0) \in
 \mathcal{S}_1}
 H^1(W^{i,c}_{j,y_0},
 \overline{\mathbb{Q}}_l)
 \simeq \bigoplus_{(i,y_0) 
 \in \mathcal{S}_2}{\rm Ind}^{G_2^F}
 _{\overline{\mathfrak{U}}^{\times}}
 H^1(W^{i,c}_{0,y_0},
 \overline{\mathbb{Q}}_l)
\end{equation}
with $l \neq p.$
Hence, to understand 
the \'{e}tale
 cohomology group
in the left
 hand side
 of (\ref{ind}), 
 it suffices to 
 understand  
 $H^1(W^{i,c}_{0,y_0},
 \overline{\mathbb{Q}}_l)$
 as a 
 $\overline{\mathfrak{U}}
 ^{\times}$-representation. 
\end{remark}

%%%%%%%%%%%%%%%%%%%%%%%%%%%

\section{Inertia action on
 the components in the 
stable reduction 
of $\mathcal{X}(\pi^2)$}\label{aci1}
In this section, we determine the right 
action of 
inertia on the components $\overline{\mathbf{Y}}_{2,2}$, 
$X^i_j\ ((i,j) \in \mathcal{S}),$ 
$\overline{\mathbf{Z}}_{1,1,0}$, $W^i_{j,k}$
for $(i,j,k) \in \mathcal{S}_1.$
First, we prepare some notations used through
 this section.
For a finite extension $L/F,$ 
we have the Artin reciprocity map
$\mathbf{a}_L:W_L^{\rm ab} 
\overset{\sim}{\to} 
L^{\times}$, which is normalized such that 
the geometric Frobenius 
is sent to 
 a prime element $\pi_L$
 by this map.
 Then, for an extension $L/K,$
 we have $\mathbf{a}_K=
 {\rm Nr}_{L/K} \circ \mathbf{a}_L.$
Let ${\rm LT}_L$ denote the formal
$\mathcal{O}_L$-module
over 
$\mathcal{O}_{\hat{L}^{\rm ur}}$,  
with ${\rm LT}_L 
\otimes k^{\rm ac}$
of height $1$. 
For $n \geq 1$, 
let $\pi_{n,L} \in {\rm LT}_L[{\pi}_L^n](\mathbf{C})$
be primitive $\pi^n$-torsion points.
Then, by the classical 
Lubin-Tate theory, we have the following equality
\begin{equation}\label{clt}
\sigma(\pi_{n,L})=
[\mathbf{a}_L(\sigma)]_{{\rm LT}_L}(\pi_{n,L})
\end{equation}
for any $n \geq 1$, $\sigma \in I^{\rm ab}_L$.

We consider a case $L=F.$
We denote by the same letter 
$\mathbf{a}_F$ for the composite
\begin{equation}\label{van0}
\mathbf{a}_F:\ I^{\rm ab}_F 
\overset{\mathbf{a}_F|_{I_F^{\rm ab}}}{\to} 
\mathcal{O}^{\times}_F \to
(\mathcal{O}_F/\pi^2)^{\times}.
\end{equation}
We write
$
\mathbf{a}_F(\sigma)=(\zeta_0(\sigma),\lambda_0(\sigma))
\in (\mathcal{O}_F/\pi^2)^{\times}
 \simeq \mathbb{F}^{\times}_{q} 
\times \mathbb{F}_q.
$
Then, by (\ref{clt}),
 we have the followings 
\[
\zeta_0(\sigma)=
\overline{\sigma(\pi_1)/\pi_1},\ 
\lambda_0(\sigma)
=\overline{\biggl(\frac{\pi_1\sigma(\pi_2)
-\sigma(\pi_1)\pi_2}{\sigma(\pi_1)\pi_1}\biggr)}.
\]
For $i \in (\mathcal{O}_F/\pi^2)^{\times}$
and $\sigma \in I_F,$ 
we set $i \sigma:=
\mathbf{a}_F(\sigma)^{-1} \times i.$
It is well-known that, under some 
identification
$\pi_0(\mathcal{X}(\pi^2))\simeq 
(\mathcal{O}_F/\pi^2)^{\times},$ 
 the inertia $I_F \ni \sigma$ acts on 
 $\pi_0(\mathcal{X}(\pi^2))$
 by 
$i \mapsto i \sigma.$ See Theorem \ref{caq}.

\subsection{Action of Inertia}
We will recall the right action of 
inertia on the stable model of 
a curve over $F$
from \cite[Section 6]{CM2}.
Note that 
the inertia action
 considered in loc.\ cit.\ 
 is a left action.
In this paper, 
we want to consider 
the right action. 
Hence, the (right) action of inertia
is characterized as in (\ref{ine1}).

If $Y/F$
is a curve, and 
$\mathcal{Y}$ 
its stable model
 over $\mathbf{C},$
there is 
a homomorphism
$w_Y$
\begin{equation}\label{mc}
w_Y:I_F={\rm Gal}
(\mathbf{C}/F^{\rm ur})
\to 
{\rm Aut}(\overline{\mathcal{Y}}).
\end{equation}
It is characterized 
by the fact that for each $P \in Y(\mathbf{C})$
and $\sigma \in I_F,$
\begin{equation}\label{ine1}
\overline{\sigma^{-1}(P)}=(\overline{P})w_Y(\sigma).
\end{equation}
We have something similar if $\mathbf{Y}$
is a reduced affinoid over $F.$
Namely, we have a homomorphism
$w_Y:I_F \to
 {\rm Aut}(\overline{\mathbf{Y}}_{\mathbf{C}})$
characterized by (\ref{ine1}).
This follows from the fact that $I_F$
preserves 
$A(\mathbf{Y}_{\mathbf{C}})^{0}$
(power bounded elements of 
$A(\mathbf{Y}_{\mathbf{C}})$)
and $A(\mathbf{Y}_{\mathbf{C}})^v$
(topologically nilpotent elements of 
$A(\mathbf{Y}_{\mathbf{C}})$).

%Moreover, inertia action behaves well with respect to morphisms
%in the following sense.
%\begin{lemma}(\cite[Lemma 6.1]{CM2})\label{ine2}
%If $f:X \longrightarrow Y$ is a morphism of reduced affinoids over $K$
%and $\sigma \in I_K,$ then
%$w_Y(\sigma) \circ \bar{f}=\bar{f} \circ w_X(\sigma).$

\subsection{Inertia action on 
$\overline{\mathbf{Y}}_{2,2}$ and 
$\{X^i_j\}_{(i,j) \in \mathcal{S}}$}
We determine the right 
action of the 
inertia on the components 
$\overline{\mathbf{Y}}_{2,2}$ and 
$\{X^i_j\}_{(i,j) \in \mathcal{S}}$ 
using (\ref{ine1}).
 These components are computed in 
 subsections in \ref{yo1} and \ref{yo2}.
 
In the following, we use  
the notations in \ref{yo1} and \ref{yo2}.
We briefly recall them.
Let $\kappa_1$ 
be an element such that 
$\kappa_1^{q^3(q^2-1)}=\pi.$
We set $\kappa:=\kappa_1^q$ 
and 
$\gamma:
=\kappa^{(q-1)(q^2-1)}.$
We write 
$\gamma^{\frac{1}{q(q^2-1)}}$
for $\kappa_1^{q-1}.$
We fix an element $\zeta \in 
\mu_{q-1}(\mathcal{O}_F).$
Moreover, we choose elements
 ${\gamma}_0$, 
 $\tilde{\gamma}_0$ and 
 $\tilde{\gamma}_1$
such that 
 $\gamma_0^q-\gamma^q\gamma_0=\zeta$, 
 $\tilde{\gamma}_0^q=\gamma_0$
 and $\tilde{\gamma}^q_1
 +\tilde{\gamma}_0+\gamma^{1/q}\tilde{\gamma}_1=0$
Let $y_0$ be an element such that
$y_0^{q^2-1}+\tilde{\gamma}_0^{q-1}=0.$ 
We set $w:=\gamma^{1/q(q-1)}$
and $w_1:=y_0^q\gamma^{1/(q^2-1)}.$
Then, we have $v(w)=1/q^3,\ 
v(w_1)=1/q^2(q+1).$
Let $x_0$ be an element such that
$x_0^qy_0-x_0y_0^q=
\tilde{\gamma}_0+\gamma^{1/q}
\tilde{\gamma}_1.$ 

%\begin{lemma}\label{ci1}
%Let $\sigma \in I_F.$
%We write $\sigma(\kappa)=\zeta_\sigma \kappa$
%for some $\zeta_{\sigma} \in \mathbb{F}^{\times}_{q^2}.$
%We set 
%\[
%c_0:=\frac{\sigma(\gamma_0)
%-\gamma_0}{\gamma^{\frac{q}{q-1}}\zeta}.
%\]
%Then, $\sigma$ acts on the group of 
% the connected components 
%$\pi_0(\mathbf{Y}_{2,2}) \simeq (\mathcal{O}_F/\pi^2)^{\times} \simeq 
%\mathbb{F}^{\times}_{q} 
%\times \mathbb{F}_q \ni (\zeta,\tilde{\mu})$
%as follows
%\[
%(\zeta,\tilde{\mu}) \mapsto
%(\zeta\zeta_{\sigma}^{q+1}, \tilde{\mu}-\bar{c}_0).
%\]
%\end{lemma}
%\begin{proof}
%By $\mathcal{Z}=(x^qy-xy^q)^q-\gamma(x^{q^3}y-xy^{q^3}),$
%we acquire
%$\mathcal{Z}(\sigma(P))=\zeta_{\sigma}^{q+1}\sigma(\mathcal{Z}(P)).$
%By $\mathcal{Z}=\tilde{\gamma}_0-\gamma^{\frac{q}{q-1}}c$, 
%we obtain
%$c(\sigma(P))=\zeta_{\sigma}^{q+1}(\sigma(c(P))-c_0).$
%Hence, the required assertion follows from this and 
%$\sigma(c(P)) \equiv c(P)\ ({\rm mod}\ 0+).$
%\end{proof}

\begin{lemma}\label{ci2}
Let the notation be as above.
Let $\sigma \in I_F.$
We write $\sigma(\kappa)=\zeta_\sigma \kappa$
with 
some $\zeta_{\sigma} \in \mu_{q^2(q^2-1)}.$
We choose elements 
$i=(\bar{\zeta},\tilde{\mu}) 
\in (\mathcal{O}_F/\pi^2)^{\times}$ 
and 
$j=(\bar{x}_0,\bar{y}_0) \in \mathcal{S}^i_{00}.$
We set
$j \sigma:=(\bar{\zeta}_{\sigma}^{-1}
\bar{x}_0,
\bar{\zeta}_{\sigma}^{-1} 
\bar{y}_0) \in 
\mathcal{S}^{i \sigma}_{00}.$
Furthermore, we set as follows
\[
a_0:=\frac{\sigma(\tilde{\gamma}_1)-\tilde{\gamma}_1}
{w},\ 
d_0:=\frac{\sigma(\tilde{\gamma}_0)
-\tilde{\gamma}_0}{\gamma^{1/(q-1)}}.
\]
Note that we have 
$\bar{a}_0 \in \mathbb{F}_{q^2},\ 
\bar{d}_0 \in 
\mathbb{F}_q$ and $\bar{a}_0^q+\bar{a}_0+\bar{d}_0=0.$
%For $i=(\zeta,\tilde{\mu}) 
%\in \mathbb{F}_q^{\times} \times \mathbb{F}_q$, 
%we set $i \sigma:=(\zeta_{\sigma}\zeta,\tilde{\mu}+c_0).$
\\1.\ The element $\sigma$ induces the following morphism
\[
\overline{\mathbf{Y}}^i_{2,2} 
\to \overline{\mathbf{Y}}^{i \sigma}_{2,2}:\ 
(x,y,Z_1) \mapsto 
(\bar{\zeta}_{\sigma}^{-1}x,
\bar{\zeta}_{\sigma}^{-1}y,
\bar{\zeta}_{\sigma}^{-(q+1)}Z_1).
\]
\\2.\ The element $\sigma \in I_F$ 
also induces the following morphism
\[
X^i_j \to 
X^{i \sigma}_{j \sigma}:\ 
(a,z_1) \mapsto 
(\bar{\zeta}_{\sigma}^{-(q+1)}
(a-\bar{a}_0-\bar{d}_0),z_1).
\]
\end{lemma}
\begin{proof}
Let $\sigma \in I_F$ 
and $P \in \mathbf{Y}_{2,2}(\mathbf{C}).$
First, note that we have $X_2(\sigma^{-1}(P))
=\sigma^{-1}(X_2(P)).$
Since we have $X_2=\kappa x,$ we have 
$x(\sigma^{-1}(P))
=\zeta_{\sigma}^{-1}\sigma^{-1}(x(P)).$
The same thing holds for $y$.
 We consider the followings
  by (\ref{aw6})
\begin{itemize}
 \item $x(P)^qy(P)-x(P)y(P)^q
=\tilde{\gamma}_0+\gamma^{1/q}Z_1(P),$ \\
 \item $x(\sigma^{-1}(P))^qy(\sigma^{-1}(P))
 -x(\sigma^{-1}(P))y(\sigma^{-1}(P))^q
=\zeta_{\sigma}^{-(q+1)}\tilde{\gamma}_0
+\gamma^{1/q}Z_1(\sigma^{-1}(P)).$ \\
\end{itemize}
By applying $\sigma^{-1}$ to the first equality, 
we acquire
\begin{equation}\label{fgg}
Z_1(\sigma^{-1}(P))=\zeta_{\sigma}^{-(q+1)}
(\sigma^{-1}(Z_1(P))-wd_0).
\end{equation}
Note that we have 
$\bar{d}_0=
\overline{(\tilde{\gamma}_0
-\sigma^{-1}(\tilde{\gamma}_0))/
\gamma^{1/(q-1)}}.$
Since we have 
$y(\sigma^{-1}(P))
=\zeta_{\sigma}^{-1}\sigma^{-1}(y(P)) \equiv
 \zeta_{\sigma}^{-1}y(P)\ ({\rm mod}\ 0+)$
  and 
$\sigma^{-1}(Z(P)) \equiv Z(P)\ ({\rm mod}\ 0+),$
the required 
assertion $1$ follows 
from (\ref{fgg}).
We prove the assertion $2$.
We consider the followings 
in (\ref{ss})
\[
Z_1(P)=\tilde{\gamma}_1+wa(P),\ 
Z_1(\sigma^{-1}(P))=
\zeta_{\sigma}^{-(q+1)}
\tilde{\gamma}_1+wa(\sigma^{-1}(P)).
\]
By applying $\sigma^{-1}$ 
to the first equality,
 we obtain the following equality by (\ref{fgg})
\[
a(\sigma^{-1}(P))
=\zeta_{\sigma}^{-(q+1)}
(\sigma^{-1}(a(P))-a_0-d_0).
\]
Hence, by $\sigma^{-1}(a(P)) 
\equiv a(P)\ ({\rm mod}\ 0+),$
we acquire $\sigma(a)
=\zeta_{\sigma}^{-(q+1)}
(a-{a}_0-{d}_0)\ ({\rm mod}\ 0+).$
On the other hand, 
we consider the following equalities 
induced by (\ref{ss})
\begin{equation}\label{b_1}
y(P)=y_0+w_1z_1(P),\ 
y(\sigma^{-1}(P))
=\zeta_{\sigma}^{-1}y_0
+\zeta_{\sigma}^{-q}
w_1z_1(\sigma^{-1}(P)).
\end{equation}
We easily check  
$y(\sigma^{-1}(P))
 =\zeta_{\sigma}^{-1}\sigma^{-1}(y(P))$, 
 $\sigma^{-1}(z_1(P)) \equiv z_1(P)\ 
 ({\rm mod}\ 0+)$ and $\sigma^{-1}w_1 \equiv 
 \zeta_{\sigma}^{-(q-1)}w_1\ ({\rm mod}\ 0+).$
Hence, 
by applying $\sigma^{-1}$
 to the first equality
in (\ref{b_1}), 
 we acquire  
 \[
z_1(\sigma^{-1}(P))=\zeta_{\sigma}^{q-1}
(\sigma^{-1}(w_1)/w_1)
\sigma(z_1(P))
 \equiv z_1(P)\ ({\rm mod}\ 0+).
\]
Hence, we obtain the required assertion.
\end{proof}

In the following, 
we rewrite Lemma \ref{ci2}
as in Corollary \ref{ine}. 
Let $E/F$ 
denote the unramified 
quadratic extension.
%Let ${\rm LT}_E$
%be the universal formal $\mathcal{O}_E$-module
% over $\hat{E}^{\rm ur}$ 
%of height $1.$
We choose a model ${\rm LT}_E$ 
such that 
\[
[\pi]_{{\rm LT}_E}(X)=\pi X-X^{q^2}.
\]
%Let $\pi_{i,E} \in {\rm LT}_E[\pi^i]$
%for $i \geq 1$ be a primitive element.
We consider the following composite
\[
\mathbf{a}_E:I_F^{\rm ab} \simeq I_E^{\rm ab}
\overset{\mathbf{a}_E|_{I^{\rm ab}_E}}{\longrightarrow} 
\mathcal{O}^{\times}_E 
\to (\mathcal{O}_E/\pi^2)^{\times} \simeq 
\mathbb{F}^{\times}_{q^2} \times \mathbb{F}_{q^2}
\]
\[
\sigma \mapsto (\zeta(\sigma),\lambda(\sigma))
=\biggl(\overline{\biggl
(\frac{\sigma(\pi_{1,E})}{\pi_{1,E}}
\biggr)},
\overline{\biggl(\frac{\pi_{1,E}\sigma(\pi_{2,E})
-\sigma(\pi_{1,E})\pi_{2,E}}{\pi_{1,E}
\sigma(\pi_{1,E})}\biggr)}\biggr).
\]
By $\mathbf{a}_F={\rm Nr}_{E/F} \circ
\mathbf{a}_E$,
 we have $\zeta(\sigma)^{q+1}=\zeta_0(\sigma)$
 and 
 $\lambda(\sigma)^q+\lambda(\sigma)=\lambda_0(\sigma)$
  for $\sigma \in I_E.$
\begin{corollary}\label{ine}
Let $\sigma \in I_F.$
Then, the element $\sigma$ acts 
on the set of
 the connected 
components $\pi_0(\mathbf{Y}_{2,2})$
as follows
\[
i:=(\zeta,\tilde{\mu}) \mapsto
i \sigma=(\zeta_0(\sigma)^{-1}\zeta,
\tilde{\mu}-\lambda_0(\sigma)\zeta).
\]
The element
 $\sigma \in I_{F}$ acts 
on $\mathcal{S}^i_{00}$
as follows
\[
j:=(x_0,y_0) \mapsto 
j \sigma:=(\zeta(\sigma)^{-1}x_0,
\zeta(\sigma)^{-1}y_0).
\]
Moreover, $\sigma$ acts on the components $\{X^i_j\}
_{(i,j) \in \mathcal{S}}$
as follows
\[
\sigma: X^{i}_{j} \to X^{i \sigma}_{j \sigma}:\ 
(a,z_1) \mapsto 
(\zeta(\sigma)^{-(q+1)}(a+\lambda(\sigma)\zeta),z_1).
\]
\end{corollary}
\begin{proof}
Clearly, we have 
$\bar{\zeta}_{\sigma}=\zeta(\sigma)$ in 
$\mathbb{F}^{\times}_{q^2}$.
By $\gamma_0^q-\gamma^q\gamma_0=\zeta$, $\tilde{\gamma}_0^q
=\gamma_0$
and $\tilde{\gamma}^q_1+\gamma^{1/q}\tilde{\gamma}_1
+\tilde{\gamma}_0=0,$
we easily check that 
$\bar{d}_0=-\lambda_0(\sigma)\zeta$ 
and $\bar{a}_0^q=\lambda(\sigma)\zeta$
in $\mathbb{F}_{q^2}$. Hence, by 
$\lambda(\sigma)^q+\lambda(\sigma)=\lambda_0(\sigma)$,
 we acquire 
$\bar{a}_0+\bar{d}_0=-\lambda(\sigma)\zeta.$
Therefore, the 
required assertion 
follows from Lemma \ref{ci2}.
\end{proof}

\subsection{Inertial action on 
$\overline{\mathbf{Z}}_{1,1,0}$
and $\{W^i_{0,j}\}_{(i,j) 
\in \mathcal{S}_0}$}
\label{inZ}
In this subsection, we determine 
the action of inertia
on the components 
$\overline{\mathbf{Z}}_{1,1,0}$
and $W^i_{0,j}$ for
$(i,j) 
\in \mathcal{S}_0$ in the same way 
as in the previous subsection.
We use the notations 
in subsections \ref{zo1} and
 \ref{zo2}
freely.

We briefly recall the notations in subsections 
\ref{zo1} and \ref{zo2}.
We choose an element $\kappa_1$
such that
$\kappa_1^{2q^4(q-1)}=\pi$ with 
$v(\kappa_1)
=1/2q^4(q-1).$ We set 
$\kappa:=\kappa_1^q.$
We set $\gamma:=\kappa^{q(q-1)^2}$ with 
$v(\gamma)=(q-1)/2q^2.$
We write $\gamma^{\frac{1}{q^2(q-1)}}$
for $\kappa_1^{q-1}$.
Let $\zeta \in \mu_{q-1}(\mathcal{O}_F).$
We choose elements $\tilde{\gamma}_1$ and 
$y_0$ such that
$\tilde{\gamma}_1^q=
\iota2\zeta \{1+\gamma^{1/q^2}
\bigl(\frac{\tilde{\gamma}_1}{\zeta}\bigr)\}^{1/2}$
and $y_0^{q^2-1}=
\iota/\{1+\gamma^{1/q^2}\bigl
(\frac{\tilde{\gamma}_1}{\zeta}\bigr)\}^{1/2}$
with $\iota \in \{\pm1\}.$
Set $w:=y_0^{q+1}\gamma^{\frac{1}{q^2(q-1)}}.$
Furthermore, we choose an element $w_1$
such that
$y_0^{q^2-3}(\zeta+\gamma^{1/q^2}\tilde{\gamma}_1)w_1^2=w^q.$
Then, we have $v(w)=1/2q^4$ and $v(w_1)=1/4q^3.$ 

\begin{lemma}\label{ci2''}
Let $\sigma \in I_F.$
We write $\sigma(\kappa)=\zeta_1\kappa$
with $\zeta_1 \in \mu_{2q^3(q-1)}.$
We choose elements
 $i=(\bar{\zeta},\tilde{\mu}) 
\in (\mathcal{O}_F/\pi^2)^{\times}$ and 
$\bar{y}_0 \in \mu_{2(q^2-1)}.$
We put $\bar{y}_0\sigma:
=\bar{\zeta}_1^{-1}\bar{y}_0.$
Furthermore, we set 
\[
a_0:=\overline
{\biggl(
\frac{\sigma(\tilde{\gamma}_1)
-\tilde{\gamma}_1}
{w}\biggr)},\ b_0:=\overline{\sigma(w)/w},\ 
c_0:=\overline{\sigma(w_1)/w_1}.
\]
Then, we have $a_0 \in \mathbb{F}_q$, $b_0 \in 
\{\pm1\}$ and $c_0^2=b_0.$
\\1.\ Then, $\sigma$ induces the following 
morphism
\[
\overline{\mathbf{Z}}^i_{1,1,0} \to 
\overline{\mathbf{Z}}^{i \sigma}_{1,1,0}:\ 
(Z_2,y) \mapsto 
(Z_2,\zeta_1^{-1} y).
\]
\\2.\ The element $\sigma$ 
induces the following morphism
\[
W^i_{0,\bar{y}_0} \to 
W^{i \sigma}_{0,
\bar{y}_0 \sigma}:\  
(a,s) \mapsto 
(b_0^{-1}a-a_0,c_0^{-1}s).
\]
\end{lemma}
\begin{proof}
We prove the assertion in the same way as in 
Lemma \ref{ci2}.
We omit the detail.
\end{proof}
We prepare some notations.
Let $\bar{\zeta} \in \mathbb{F}^{\times}_q$
and $\bar{y}_0 \in \mu_{2(q^2-1)}.$
We consider an element 
 $a:=(\bar{\zeta}/\bar{y}_0^{q+1})^2 \in
  \mathbb{F}^{\times}_q.$
Let $\tilde{a} \in \mu_{q-1}(\mathcal{O}_F)$
denote the unique 
lifting of $a$.
Let $\tilde{a}_1$ be an element such that 
$\tilde{a}_1^2=\tilde{a}.$
We set $t:=\tilde{a}_1
\gamma^{\frac{q^2}{q-1}}$
and $E_1:=F[t].$
Then, $E_1/F$ is 
a totally ramified quadratic extension
and
$t$ is a uniformizer of $E_1.$
Moreover, we have $t^2=\tilde{a}\pi.$
Clearly, we have $v(t)=1/2.$
We choose a model of ${\rm LT}_{E_1}$ such that
\[
[t]_{\rm LT_{E_1}}(X)
=tX-X^q.
\]
%We consider the universal formal $\mathcal{O}_{E_1}$-
%module
%${\rm LT}_{E_1}/\hat{E_1}^{\rm ur}$
%of height $1$.
%We write
%\[
%[t]_{E_1}(X)=tX-X^q.
%\]
%Let $\pi_{i,{E_1}}$ for $i \geq 1$
%be a primitive $\pi^i$-torsion
% point of ${\rm LT}_{E_1}.$
We consider the following composite
\begin{equation}\label{van1}
\mathbf{a}_{E_1}:I_{E_1}^{\rm ab} 
\overset{\mathbf{a}_{E_1}|_{I_{E_1}^{\rm ab}}}
{\longrightarrow} 
\mathcal{O}^{\times}_{E_1} \to
(\mathcal{O}_{E_1}/\pi)^{\times} \simeq 
\mathbb{F}^{\times}_q \times \mathbb{F}_q
\end{equation} 
\[\sigma \mapsto
(\zeta(\sigma),\lambda(\sigma)):=
(\overline{\sigma(\pi_{1,{E_1}})/\pi_{1,{E_1}}},
\overline{\biggl(\frac{\pi_{1,{E_1}}
\sigma(\pi_{2,{E_1}})-\pi_{2,{E_1}}
\sigma(\pi_{1,{E_1}})}
{\pi_{1,{E_1}}\sigma(\pi_{1,{E_1}})}\biggr)}).
\]
Then, by Lemma \ref{ci2''}, 
we obtain the following proposition.
\begin{corollary}\label{zi1}
Let $\sigma \in I_{E_1}.$
%Note that we have 
%$\overline{
%\bigl(\frac{t}{\gamma^{\frac{q^2}{q-1}}}\bigr)}^2=\bar{k}$
%in $\mathbb{F}^{\times}_q$.
%Under the identification
%$i \in (\mathcal{O}_F/\pi^2)^{\times}
% \simeq \mathbb{F}^{\times}_q
%\times \mathbb{F}_q,$
We choose elements
 $i=(\bar{\zeta},\tilde{\mu}) \in 
 (\mathcal{O}_F/\pi^2)^{\times}$
 and $\bar{y}_0 \in \mu_{2(q^2-1)}.$
We set
$\bar{y}_0 \sigma:
=\zeta(\sigma)^{-1}\bar{y}_0
 \in \mu_{2(q^2-1)}.$
Furthermore, we set
\[
a(\sigma):=
-\frac{\bar{y}_0^{q^2-1}
2\bar{\zeta}^2}{
\bar{y}_0^{2(q+1)}
}\lambda(\sigma) 
\in \mathbb{F}_q.
\]
Then, the
 element $\sigma$ 
 acts on the components
$\{W^i_{0,\bar{y}_0}\}_{(i,\bar{y}_0) 
\in \mathcal{S}_0}$
as follows
\[
W^i_{0,\bar{y}_0} \to 
W^{i \sigma}_{0,\bar{y}_0 \sigma}:\  
(a,s) \mapsto 
(a-a(\sigma),
\zeta(\sigma)^{-(q-1)/2}s).
\]
\end{corollary}
\begin{proof}
We have the following congruence 
by the definitions of
 $\pi_{i,E_1}$ for $i=1,2$
\[
\pi_{1,{E_1}}\sigma(\pi_{2,{E_1}}^q)-\pi_{2,{E_1}}^q
\sigma(\pi_{1,{E_1}})
 \equiv t\bigl(
 \sigma(\pi_{2,{E_1}})\pi_{1,{E_1}}
 -\sigma(\pi_{1,{E_1}})\pi_{2,{E_1}}\bigr)\ 
 ({\rm mod}\ 
 \biggl(\frac{q+1}{2(q-1)}\biggr)+).
\]
By this, 
we check the following by the definition of 
$\tilde{\gamma}_1$
\[
a_0 \equiv 
-\bar{y}_0^{q^2-1}2\bar{\zeta} \overline
{\biggl(
\frac{\sigma(\pi_{2,{E_1}}^q)
\pi_{1,{E_1}}-\pi^q_{2,{E_1}}\sigma(\pi_{1,{E_1}})}
{w^{q^4}\sigma(\pi_{1,{E_1}})\pi_{1,{E_1}}}\biggr)} \equiv 
-\overline{\biggl(\frac{y_0^{q^2-1}2\zeta t}{y_0^{q+1}
\gamma^{\frac{q^2}{q-1}}}\biggr)}\lambda(\sigma) 
\equiv a(\sigma)\ ({\rm mod}\ 0+).
\]
Hence, the required assertion follows.
\end{proof}

\section{Analysis of cuspidal part in 
the \'{e}tale cohomology of 
the Lubin-Tate space $\mathcal{X}(\pi^2)$}\label{acu1}

%In this section, we write down the action of 
%$G_2^F$
%and the inertia action explicitly
%on the stable reduction of the Lubin-Tate space
%${\mathcal{X}(\pi^2)}.$ 
%In subsection \ref{Lo}, we also write down the Lusztig surface for 
%$G_2^F$ as in \cite{Lus}.
%The cohomology $H_c^2$ of this surface
%contains all ``cuspidal'' representations of 
%$G_2^F$
%in a sense of \cite{Sta}. 

%See subsection \ref{Go}
%for a precise definition of cuspidal representation of 
%$G_2^F$.

In this section, we investigate 
the cohomologies of the curves in the stable reduction 
of $\mathcal{X}(\pi^2)$ using the explicit
 descriptions of 
 the action of 
 $\mathcal{O}^{\times}_3,$ 
 $G_2^F$ and $I_F$ 
  given in 
  the previous sections.
Recall that 
the components $\{X^{i,c}_j\}_{(i,j) 
\in \mathcal{S}},$ 
with each having an 
affine model $X^q+X=Y^{q+1}$, 
and the components 
$\{W^{i,c}_{j,k}\}_{(i,j,k) 
\in \mathcal{S}_1}$,
with each having 
an affine model $a^q-a=s^2,$ 
appear in the 
stable reduction of 
$\mathcal{X}(\pi^2)$.
See Propositions 
\ref{fou-1}, 
\ref{fou0} and \ref{fou}.
We set
\[
W:=\bigoplus_{(i,j) 
\in \mathcal{S}}
H^1(X^{i,c}_{j},\overline{\mathbb{Q}}_l),\ 
W':=\bigoplus_
{(i,j,k) 
\in \mathcal{S}_1}
H^1(W^{i,c}_{j,k},\overline{\mathbb{Q}}_l)
\]
with $p \neq l.$
The dual graph of 
the stable reduction 
$\overline{\mathcal{X}(\pi^n)}$
is known to be a tree. For example, see 
\cite[Proposition 3.4]{W3} 
for this fact. 
Because of the fact, 
 these spaces $W$ and $W'$ are direct
 summands of
 $H^1(\mathcal{X}(\pi^2)_{\mathbf{C}},
 \overline{\mathbb{Q}}_l).$ 
 Furthermore, we can easily check that 
 these subspaces $W$ and $W'$ are 
 $\mathbf{G}:=G_2^F \times 
\mathcal{O}^{\times}_3 \times I_F$-stable.
Note that 
the group $\mathbf{G}$ 
acts on the spaces 
$W$ and 
$W'$ on the left.
In this section, we analyze these representations
by using the explicit action of $\mathbf{G}$ on $W$ and $W'$.
See Proposition \ref{gey}.2, 
Corollary \ref{lap} 
and Corollary \ref{dek2} for precise statements.
As a result, we know that the 
local Jacquet-Langlands correspondence 
and the $\ell$-adic 
local Langlands correspondence 
for unramified (resp.\ ramified) 
cuspidal representations
of ${\rm GL}_2(F)$ of 
level $1$ (resp.\  of level $1/2$)
are realized in $W$. (resp.\ $W'$.)
See \ref{con} for more details.
%See \ref{con} for more details.
%In that process, we compare the \'{e}tale cohomology group 
%with \'{e}tale cohomology $H_c^2$ of the Lusztig surface mentioned above.
%Throughout this section, we assume 
%\begin{equation}\label{assu}
%p:\ {\rm odd},\ {\rm char}\ F=p>0,\ 
%{\rm or}\ {\rm char}\ F=0\ {\rm and}\ e_{F/\mathbb{Q}_p} \geq 2.
%\end{equation}
%Under this assumption, we fix an identification
%$\mathcal{O}_F/(\pi^2)
%\simeq \mathbb{F}_q[[\pi]]/(\pi^2).$
%We choose a prime number $l \neq p.$
%Let $\{\overline{\mathbf{Y}}^i_{2,2}\}_{i 
%\in \mathbb{F}^{\times}_q \times \mathbb{F}_q}$ 
%denote the connected component of the reduction
%$\overline{\mathbf{Y}}_{2,2}.$
%Let $i=(\zeta,\tilde{\mu}) 
%\in \mathbb{F}^{\times}_q 
%\times \mathbb{F}_q.$
%Recall that 
%we set 
%\[
%\mathcal{S}^i_{00}:=\{(x_0,y_0) \in 
%(\mathbb{F}^{\times})^2|\ 
%x_0^qy_0-x_0y_0^q=\zeta,\ x_0^{q^2-1}=y_0^{q^2-1}=-1\}.
%\]
%Note that we have
% $|\mathcal{S}^i_{00}|=q(q^2-1).$
%Then, in the stable 
%reduction of the Lubin-Tate space $\mathcal{X}(\pi^2),$
%there exist $q(q^2-1)$ 
%affine smooth curves with genus $q(q-1)/2;$
%$\{X^i_j:X^q+X=Y^{q+1}\}_{j \in \mathcal{S}^i_{00}}.$
%We will show that the $
%{\rm SL}_2(\mathcal{O}_F/\pi^2\mathcal{O}_F)$-representation
%$\bigoplus_{j \in \mathcal{S}_{00}}H_c^1(X^i_j,\overline{\mathbb{Q}}_l)$
%contains all ``cuspidal'' representation of 
%${\rm SL}_2(\mathcal{O}_F/\pi^2\mathcal{O}_F).$

\subsection{Analysis of 
the \'{e}tale cohomology of 
the components $\{X^{i,c}_j\}_{(i,j) 
\in \mathcal{S}}$}\label{acc1}
In this subsection, we analyze 
the following 
\'{e}tale 
cohomology group
of 
the components 
$\{X^{i,c}_j\}_{(i,j) 
\in \mathcal{S}}$ computed 
in subsection \ref{yo2}
\[
W:=\bigoplus_
{(i,j) \in \mathcal{S}}
H^1(X^{i,c}_j,\overline{\mathbb{Q}}_l)
\]
as a $\mathbf{G}$-representation.
%The space $W$ is a direct 
%summand 
%of $H^1(\mathcal{X}(\pi^2)_{\mathbf{C}},
%\overline{\mathbb{Q}}_l)$, because 
%the dual graph of the 
%stable reduction 
%$\overline{\mathcal{X}(\pi^2)}$
%is a tree.
%Furthermore, 
%$W$
%is $\mathbf{G}:=G_2^F \times 
%\mathcal{O}^{\times}_3
%\times I_F$-stable.
%$H^1(\overline{\mathcal{X}(\pi^2)},\overline{\mathbb{Q}}_l)$
%as a $G_2^F:=G_2^F$
%representation.
As a result, 
for unramified 
cuspidal representations of 
${\rm GL}_2(F)$ of level $1$,
we will check that the local 
Jacquet-Langlands correspondence and 
the local Langlands 
correspondence are realized in
 $W.$
See \ref{con} for 
a precise meaning of this.
%Let $\Gamma:=(\mathcal{O}_E/\pi^2)^{\times}
%\subset \mathcal{O}_D^{\times}/U_D^3$
%be a maximal torus. 
%We  also observe that the restriction of $W$
%to a subgroup $G_2^F \times \Gamma \times \{1\}
%\subset \mathbf{G}$ is related to
%a Lusztig representation $\rho_{DL}$
% computed in the previous subsection.
%See Proposition \ref{gey}.2 for this.

Let
$\mathcal{S}^{\mathbb{F}^{\times}_q}_{00}
=\{(x_0,y_0) \in (k^{\rm ac})^2|\ 
\xi:=x_0^qy_0-x_0y_0^q \in \mathbb{F}^{\times}_q
,\ x_0^{q^2-1}=y_0^{q^2-1}=-1\}.$
Then, we have $\mathcal{S}=\mathbb{F}_q \times 
\mathcal{S}^{\mathbb{F}^{\times}_q}_{00}$.
Let 
$X_i$ be a projective smooth curve with 
an affine equation $X^{q^2}-X=\xi(Y^{q(q+1)}-Y^{q+1}).$
The curve $X_i$ has $q$ connected components, and 
each component has an affine model
defined by $X^q+X=\xi Y^{q+1}-c$ with 
some $c \in \mathbb{F}_q.$
Then, $W$ is written as follows
\begin{equation}\label{zent}
W \simeq \bigoplus_{i \in \mathcal{S}^{\mathbb{F}^{\times}_{q}}
_{00}}H^1(X_i,\overline{\mathbb{Q}}_l)
\end{equation} 
as a $\overline{\mathbb{Q}}_l$-vector space.
%We will write down the action of $\mathbf{G}$
%on the right hand side of (\ref{zent}), which is induced by 
%the action of $\mathbf{G}$ on $W$.
Since we have $|\mathcal{S}^{\mathbb{F}^{\times}_q}_{00}|
=|{\rm GL}_2(\mathbb{F}_q)|=q(q^2-1)(q-1)$ and
${\rm dim}\ H^1(X_i,\overline{\mathbb{Q}}_l)=q^2(q-1)$
for each $i$,  
we have ${\rm dim}\ W=q^3(q-1)^2(q^2-1)$.

%The isomorphism
% (\ref{zent}) and the 
% $\mathbf{G}$-action on $W$
% induce action of
%  $\mathbf{G}$
% on the right hand side of (\ref{zent}).
Now, we write down 
the right action of $\mathbf{G}$
on the components $\{X_i\}
_{i \in 
\mathcal{S}
^{\mathbb{F}
^{\times}_{q}}_{00}}$.
This action induces 
the left action on the right 
hand side of (\ref{zent}).
Then, it is not difficult 
to check the isomorphism (\ref{zent})
is an isomorphism 
as a $\mathbf{G}$-representation.

\paragraph{$G_2^F$-action}
First, we recall the action of 
$G_2^F$ given 
in  Proposition \ref{actt}.
%on the space (\ref{zent}).
Let $g 
=\left(
\begin{array}{cc}
a_0+a_1\pi & b_0+b_1\pi \\
c_0+c_1\pi & d_0+d_1\pi
\end{array}
\right)
\in G_2^F$ with $a_j,b_j,c_j,d_j \in 
\mu_{q-1}(\mathcal{O}_F)\ (j=0,1).$
Then, $g \in G_2^F$ acts on 
$\mathcal{S}^{\mathbb{F}^{\times}_q}_{00}$ 
as follows, factoring through
$G_1^F,$ 
\begin{equation}\label{cvv0}
g\ :\ i:=(x_0,y_0) \mapsto 
ig:=(\bar{a}_0x_0+\bar{c}_0y_0,
\bar{b}_0x_0+\bar{d}_0y_0).
\end{equation}
Of course, this action of $G_1^F$ 
on $\mathcal{S}^{\mathbb{F}^{\times}_q}_{00}$
is simply transitive.
Furthermore, $g$ induces
\begin{equation}\label{cvv1}
g:X_i \longrightarrow X_{ig}\ ;\ (X,Y) \mapsto 
(\overline{{\rm det}(\bar{g})}X+f(i,g),Y).
\end{equation}

\paragraph{$\mathcal{O}^{\times}_3$-action}
%We write $\mathcal{O}^{\times}_n$
%for $\mathcal{O}^{\times}_D/U_D^n$
%for $n \geq 1.$
%We set $\xi:=x_0^qy_0-x_0y_0^q \in \mathbb{F}^{\times}_q.$
We recall the action of 
$\mathcal{O}^{\times}_3$ given in
 Proposition \ref{giu}. 
 Let $E/F$ denote 
 the unramified 
 quadratic extension.
Let $b=a_0+\varphi b_0+\pi a_1 
\in \mathcal{O}^{\times}_3$ with
 $a_0 \in \mu_{q^2-1}(\mathcal{O}_E)$
 and $a_1,b_0 \in \mu_{q^2-1}(\mathcal{O}_E)
 \cup \{0\}.$
Then, $b$ acts 
on $\mathcal{S}^{\mathbb{F}^{\times}_q}_{00}$
as follows
\begin{equation}\label{dcc2}
i=(x_0,y_0) \mapsto ib:
=(\bar{a}_0^{-1}x_0,\bar{a}_0^{-1}y_0).
\end{equation}
Moreover, $b$ induces 
a morphism
\begin{equation}\label{dcc1}
X_i \to X_{ib}\ :\ 
(X,Y) \mapsto 
\bigl(\bar{a}_0^{-(q+1)}\bigl(X-
(\bar{b}_0\bar{a}_0^{-1})\xi 
Y+(\bar{a}_1\bar{a}_0^{-1})\xi\bigr),
\bar{a}_0^{q-1}\bigl(Y-
(\bar{b}_0\bar{a}_0^{-1})^q\bigr)\bigr).
\end{equation}
Let
$t:=a_0+\pi a_1 
\in \Gamma:=(\mathcal{O}_E/\pi^2)^{\times} 
\subset 
\mathcal{O}^{\times}_3.$
Then, $t$ induces the following 
by (\ref{dcc1})
\begin{equation}\label{dcc2}
X_i \to X_{it}\ :\ 
(X,Y) \mapsto 
\bigl(\bar{a}_0^{-(q+1)}
\bigl(X+(\bar{a}_1\bar{a}_0^{-1})\xi\bigr),
\bar{a}_0^{q-1}Y\bigr).
\end{equation}

\paragraph{Inertial action}
We recall the inertia action $I_F$ given in
 Corollary \ref{ine}.
Let ${\rm LT}_E$
be the formal $\mathcal{O}_E$-module
 over $\hat{E}^{\rm ur},$  
with ${\rm LT}_E \otimes k^{\rm ac}$ 
of height $1.$
Recall that we have
 chosen a model ${\rm LT}_E$
  such that
\[
[\pi]_{{\rm LT}_E}(X)=\pi X-X^{q^2}.
\]
Let $\pi_{i,E} \in {\rm LT}_E[\pi^i]$
for $i \geq 1$ be
 primitive elements.
We define
\[
\mathbf{a}_E:I_F \to 
I_F^{\rm ab} 
\simeq I_E^{\rm ab}
\to 
(\mathcal{O}_E/\pi^2)^{\times} 
\simeq 
\mathbb{F}^{\times}_{q^2} \times 
\mathbb{F}_{q^2}
\]
by
\[
\sigma \mapsto (\zeta(\sigma),
\lambda(\sigma))
=\biggl(\overline{\biggl(\frac{\sigma(\pi_{1,E})}
{\pi_{1,E}}\biggr)},
\overline{\biggl(\frac{\pi_{1,E}\sigma(\pi_{2,E})
-\sigma(\pi_{1,E})\pi_{2,E}}
{\pi_{1,E}\sigma(\pi_{1,E})}
\biggr)}\biggr).
\]
Then, $\sigma \in I_{F}$ acts 
on $\mathcal{S}^{\mathbb{F}^{\times}_{q}}_{00}$
as follows
\[
(x_0,y_0) \mapsto 
i \sigma:=(\zeta(\sigma)^{-1}x_0,\zeta(\sigma)^{-1}y_0).
\]
Moreover, $\sigma$ induces a morphism
\begin{equation}\label{aj1}
X_{i} \to X_{i \sigma}\ :\ 
(X,Y) \mapsto (\zeta(\sigma)^{-(q+1)}
(X+\lambda(\sigma)\xi),Y).
\end{equation}
%Note that $\mathbf{a}_E$
%is induced from the restriction to the inertia 
%group $I_F$ of 
%the Artin reciprocity map $\mathbf{a}_E:W_F \to F^{\times}$
%normalized such that, if $x$ is a geometric Frobenius, 
%the image $\mathbf{a}_E(x)$ is an {\it inverse} of a prime element 
%of $F.$ 

%We set $\zeta_{g,i}:=g(x_0,y_0) \in \mathbb{F}_q$ and 
%$\mu_{\bar{g},i}:=1/(b_0\frac{x_0}{y_0}+d_0)^{q-1} \in \mu_{q+1}.$
%We can easily check that
%for $i \in \mathcal{S}_{00}$
%and
%$g,h \in {\rm SL}_2(\mathcal{O}_F/\pi^2\mathcal{O}_F),$
%the following holds
%\begin{equation}\label{for}
%\mu_{\bar{h}\bar{g},i}=\mu_{\bar{g},i} \mu_{\bar{h},\bar{g}i}.
%\end{equation}

Let $X$ be a projective smooth curve
with an affine model 
$X^q+X=Y^{q+1}$ with genus 
$q(q-1)/2.$
Then, we have ${\rm dim}\
 H^1(X,\overline{\mathbb{Q}}_l)
=q(q-1)$.
In the following, to investigate $W$, we 
prove some elementary facts on 
$H^1(X,\overline{\mathbb{Q}}_l)$ in 
Lemma \ref{oo3} and Corollary \ref{cb1}.

Let $\mathcal{I}:={\rm ker}\ {\rm Tr}_{\mathbb{F}_{q^2}/\mathbb{F}_q}.$
Then, the group $\mathcal{I} \ni a_1$
acts on $X$ by $(X,Y) \mapsto (X+a_1,Y).$
On the other hand,
$\mu_{q+1} \ni \zeta$ acts on $X$ by $(X,Y) \mapsto (X,\zeta Y).$
%The group $\mathbb{F}_q \times \mu_{q+1} \ni (\zeta,\mu)$ acts
%on $X$ as follows
%$(\zeta,\mu):X \longrightarrow X;(a,s) \mapsto (a+\zeta,\mu s).$
Therefore, we consider $H^1(X,\overline{\mathbb{Q}}_l)$
as a $\overline{\mathbb{Q}}_l[\mathcal{I} \times \mu_{q+1}]$-module.
\begin{lemma}\label{oo3}
Let the notation be as above.
Then, we have the following isomorphism
\begin{equation}\label{zent1}
H^1(X,\overline{\mathbb{Q}}_l)
\simeq \bigoplus_{\psi \neq 0 \in \mathcal{I}^{\vee}}
\bigoplus_{\chi \neq 1 \in \mu_{q+1}^{\vee}}
\psi \otimes \chi
\end{equation}
as a $\overline{\mathbb{Q}}_l[\mathcal{I} \times \mu_{q+1}]$-module.
\end{lemma}
\begin{proof}
We have the following short exact sequence
\begin{equation}\label{bb1}
0 \to 
\bigoplus_{\psi \in \mathcal{I}^{\vee}}\psi 
\to
H_c^1(X \backslash X(\mathbb{F}_q),\overline{\mathbb{Q}}_l)
\to H^1(X,\overline{\mathbb{Q}}_l) \to 0\ ({\rm exact})
\end{equation}
as 
 a $\overline{\mathbb{Q}}_l[\mathcal{I} \times \mu_{q+1}]$-module.

Let $\mathcal{L}_{\psi}(t)$
denote the smooth $\overline{\mathbb{Q}}_l$-sheaf 
associated to
the finite Galois \'{e}tale 
covering $a^q+a=t$ of 
$\mathbb{A}^1 \ni t$
and a character $\psi \in \mathcal{I}^{\vee}.$ 
Let $\mathcal{K}_{\chi}(t)$
denote the smooth $\overline{\mathbb{Q}}_l$-sheaf associated to
the Kummer covering $y^{q+1}=t$ of $\mathbb{G}_m \ni t$
and a character $\chi \in \mu^{\vee}_{q+1}.$ 
Since $X \backslash X(\mathbb{F}_q) 
\to \mathbb{G}_m;(a,s) \mapsto s^{q+1}$
is a finite Galois \'{e}tale
 covering 
 of Galois group 
 $\mathcal{I} \times \mu_{q+1},$
 the group 
 $H_c^1(X \backslash X(\mathbb{F}_q),
 \overline{\mathbb{Q}}_{l})$
 is isomorphic to
 \begin{equation}\label{bb3}
 \bigoplus_{\psi \in 
 \mathcal{I}^{\vee}}\bigoplus
 _{\chi \in \mu^{\vee}_{q+1}}
 H_c^1(\mathbb{G}_m,
 \mathcal{L}_{\psi}(t) 
 \otimes \mathcal{K}_{\chi}(t))
 \end{equation}
 as a $\overline{\mathbb{Q}}_l
 [\mathcal{I} \times \mu_{q+1}]$-module.
 Note that we have 
 ${\rm dim}\ H_c^1(\mathbb{G}_m,
 \mathcal{L}_{\psi}(t) \otimes \mathcal{K}_{\chi}(t))$
 is equal to $1$
 if $\psi \neq 1$ and $0$ otherwise
 by the Grothendieck-Ogg-Shafarevich formula. 
Furthermore, we have
\begin{equation}\label{bb2}
\bigoplus_{\psi \in \mathcal{I}^{\vee}}
\bigoplus_{\chi \in \mu^{\vee}_{q+1}}
 H_c^1(\mathbb{G}_m,\mathcal{L}_{\psi}(t) \otimes \mathcal{K}_{\chi}(t)) \simeq 
 \bigoplus_{\psi \neq 1 \in \mathcal{I}^{\vee}}
 \bigoplus_{\chi \neq 1 \in \mu^{\vee}_{q+1}}
 H_c^1(\mathbb{G}_m,\mathcal{L}_{\psi}(t) \otimes \mathcal{K}_{\chi}(t)) \oplus
 \bigoplus_{\psi \in 
 \mathcal{I}^{\vee}}\psi
 \end{equation}
 as a $\overline{\mathbb{Q}}_l[\mathcal{I} \times \mu_{q+1}]$-module.
By (\ref{bb1}), (\ref{bb3}) 
and (\ref{bb2}), 
the required assertion follows.
\end{proof}
\begin{corollary}\label{cb1}
Let $X$ be a projective smooth curve
with an affine model 
$X^{q^2}-X=Y^{q(q+1)}-Y^{q+1}.$
The group $(a,\zeta) \in 
\mathbb{F}_{q^2} \times 
\mu_{q+1}$ acts on $X$ by 
$(X,Y) \mapsto (X+a,\zeta Y).$
We consider 
$\mathbb{F}^{\vee}_q$ 
as a subgroup of 
$\mathbb{F}^{\vee}_{q^2}$ by 
${\rm Tr}_{\mathbb{F}_{q^2}/\mathbb{F}_q}^{\vee}.$
Then, 
we have the following isomorphism
\[H^1(X,\overline{\mathbb{Q}}_l)
 \simeq \bigoplus_{(\psi,\chi) \in (\mathbb{F}^{\vee}_{q^2} 
 \backslash \mathbb{F}^{\vee}_q) 
 \times (\mu_{q+1}^{\vee}
  \backslash \{1\})}\psi \otimes \chi\]
as a $\overline{\mathbb{Q}}_l[\mathbb{F}_{q^2}
 \times \mu_{q+1}]$-module.
\end{corollary}
\begin{proof}
The curve $X$ is a disjoint union of $q$ curves 
$\{X_i\}_{i \in \mathbb{F}_q}$
with an affine model
$X^q+X=Y^{q+1}.$
The set of connected components 
of $X$ is $\mathbb{F}_q.$
Then, $\mathbb{F}_{q^2}$ acts on the group
 $\mathbb{F}_q$, according to 
the trace map.
Hence, for a fixed $i_0 \in \mathbb{F}_q$, 
we have the following 
\[
H^1(X,\overline{\mathbb{Q}}_l) \simeq 
{\rm Ind}^{\mathbb{F}_{q^2}}
_{\mathcal{I}}H^1(X_{i_0},\overline{\mathbb{Q}}_l)
\]
as a $\overline{\mathbb{Q}}_l
[\mathbb{F}_{q^2} \times \mu_{q+1}]$-module.
Therefore, the required assertion follows from
Lemma \ref{oo3}.
\end{proof}

We set $\Gamma:=(\mathcal{O}_E/\pi^2)^{\times}.$
For a character $w\in \Gamma^{\vee}$,
 we say that $w$ is {\it
  strongly primitive}
 if the restriction of $w$
 to a subgroup 
 $\mathbb{F}_{q^2} \simeq 
 {\rm Ker}\ (\Gamma \to \mathbb{F}^{\times}_{q^2})$
does not factor through the trace map 
${\rm Tr}_{\mathbb{F}_{q^2}/\mathbb{F}_q}
:\mathbb{F}_{q^2} \to \mathbb{F}_q.$
Let $\Gamma^{\vee}_{\rm stp}
 \subset \Gamma^{\vee}$
denote
 the set of strongly primitive 
characters. 
In the following, for each 
$w \in 
\Gamma^{\vee}_{\rm stp},$ 
we define a representation
$\pi_w$, which is called 
{\it (strongly) cuspidal} representation
 in \cite[5.2]{AOPS} and \cite{Onn}.
%Then, in Proposition \ref{gey}, 
%we will show that all cuspidal 
%representations appear 
% in 
%the restriction 
%$W|_{G_2^F \times \{1\} \times \{1\}}$
%with each multiplicity $q$.
%Let $\mathcal{O}_E$ denote the ring of integers of $E.$
Now, we fix an element $\zeta_0 
\in \mu_{q^2-1}(\mathcal{O}_E) 
\backslash \mu_{q-1}(\mathcal{O}_F).$
We fix the following 
 embedding 
\begin{equation}\label{ed1}
{\Gamma}:=(\mathcal{O}_E/\pi^2)^{\times}
\hookrightarrow G_2^F
\end{equation}
\[a+b\zeta_0 \mapsto 
a1_{2}+b\left(
\begin{array}{cc}
\zeta_0^q+\zeta_0 & 1 \\
-\zeta_0^{q+1} & 0
\end{array}
\right)
\]
with $a, b \in \mathcal{O}_F/\pi^2.$
We identify ${\Gamma} \simeq 
\mathbb{F}_{q^2}^{\times} \times
\mathbb{F}_{q^2}$
by $a_0+a_1\pi \mapsto 
(\bar{a}_0,(\bar{a}_1/\bar{a}_0)).$
For a character 
$\psi \in \mathbb{F}_{q^2}^{\vee}
 \backslash \mathbb{F}_q^{\vee}$
and an element $\zeta 
\in \mathbb{F}_{q^2}
 \backslash \mathbb{F}_q,$ 
we define a character 
$\tilde{\psi}_{\zeta}$ 
of $N$ by the following 
\begin{equation}\label{pi3} 
\tilde{\psi}_{\zeta}\ :\ N 
\ni \left(
\begin{array}{cc}
1+\pi a_1 & \pi b_1 \\
\pi c_1 & 1+\pi d_1
\end{array}
\right)
\mapsto 
\psi\biggl(-\frac{\bar{a}_1
\zeta+\bar{c}_1
-\zeta^q(\bar{b}_1\zeta+\bar{d}_1)}
{\zeta^q-\zeta}\biggr).
\end{equation}
Note that the restriction of 
$\tilde{\psi}_{\bar{\zeta_0}} \in N^{\vee}$ 
to a subgroup $\mathbb{F}_{q^2} 
\simeq {\Gamma} \cap N
\subset N$ is equal to $\psi.$
For a strongly primitive character
 $w=(\chi,\psi)
 \in {\Gamma}^{\vee}
\simeq (\mathbb{F}^{\times}_{q^2})^{\vee}
 \times \mathbb{F}_{q^2}^{\vee}$
i.e. $\psi \notin \mathbb{F}_q^{\vee}$,
we define a character $w$ of 
${\Gamma}N=\mu_{q^2-1}(\mathcal{O}_E)N$ by
\begin{equation}\label{cha}
w(xu)=\chi(\bar{x})\tilde{\psi}_{\bar{\zeta}_0}(u)
\end{equation}
for all $x \in \mu_{q^2-1}(\mathcal{O}_E)$ and $u \in N.$
We set
\[
\pi_w:={\rm Ind}^{G_2^F}_{{\Gamma}N}(w).
\]
Then, $\pi_w$ is called
 a (strongly) 
cuspidal representation 
of $G_2^F$ 
 in \cite[Section 5.2]{AOPS}, 
 \cite{Onn}, 
 \cite{Sta} and \cite{Sta2}.
Clearly, we have 
${\rm dim}\ \pi_w=q(q-1).$

Let us set
\[
\mathbf{H}:=G_2^F \times 
\Gamma \times I_F 
\subset \mathbf{G}.
\]
To analyze $W$ as 
a $\mathbf{G}$-representation,
we will 
investigate the 
restriction 
$W|_{\mathbf{H}}$ in
 Proposition \ref{gey}.
To do so, in the following, for each 
$w \in \Gamma^{\vee}_{\rm stp}$, 
we define a 
$\mathbf{H}$-subrepresentation
 $W_w \subset W|_{\mathbf{H}}.$ 
For a character $\psi \in \mathbb{F}^{\vee}_{q^2}$
and an element $a \in \mathbb{F}_{q^2},$
we write $\psi_a \in 
\mathbb{F}^{\vee}_{q^2}$ 
for the 
character $x \mapsto \psi(ax).$
For $i \in \mathcal{S}_{00}
^{\mathbb{F}^{\times}_q},$
we write $i=(x_0,y_0).$
Let $\nu:\Gamma \to \mu_{q+1};
a_0+a_1\pi \mapsto a_0^{q-1}.$ 
For $\chi_0 \in \mu_{q+1}^{\vee},$
we denote by the same letter $\chi_0$
 for the composite
$\chi_0 \circ \nu \in \Gamma^{\vee}.$
Then, by (\ref{zent}) and Corollary \ref{cb1}, 
we obtain the following isomorphism
as a 
$\overline{\mathbb{Q}}_l$
-vector space
\begin{equation}\label{tu}
W
\simeq \bigoplus_{i \in 
\mathcal{S}
^{\mathbb{F}^{\times}_q}_{00}}
\bigoplus_{(\psi,\chi_0) 
\in (\mathbb{F}^{\vee}_{q^2} 
 \backslash \mathbb{F}^{\vee}_q) 
 \times (\mu_{q+1}^{\vee} \backslash \{1\})} 
 \overline{\mathbb{Q}}_l
  e_{i,\psi,\chi_0}.
\end{equation}
The isomorphism (\ref{tu}) 
and 
the action of $\mathbf{H}$
on $W$
 induce the 
 $\mathbf{H}$-action
 on the right hand side of
  (\ref{tu}). 
 We write down 
 the action of $\mathbf{H}$ 
 on the basis 
 $\{e_{i,\psi,\chi_0}\}$
  in (\ref{tu})
 below. 
By (\ref{cvv0}) and (\ref{cvv1}), we have the 
following $G_2^F$-action
\begin{equation}\label{da1}
G_2^F \ni g^{-1}\ :\ e_{i,\psi,\chi_0} \mapsto 
\psi_{\overline{{\rm det}(\bar{g})}^{-1}}
(f(i,g))
e_{ig,\psi_{\overline{{\rm det}(\bar{g})}^{-1}},
\chi_0}
\end{equation}
and the following $\Gamma$-action 
by (\ref{dcc2})
\begin{equation}\label{da2}
\Gamma \ni t=a_0+a_1\pi:
\ e_{i,\psi,\chi_0} \mapsto
\psi(-(\bar{a}_1/\bar{a}_0)\xi)
\chi_0^{-1}(t)e_{it^{-1},
\psi_{\bar{a}_0^{-(q+1)}},
\chi_0}
\end{equation}
where we set $\xi:=x_0^qy_0-x_0y_0^q
 \in \mathbb{F}^{\times}_q.$
Furthermore, we have the following $I_F$-action
by (\ref{aj1})
\begin{equation}\label{da3}
I_F \ni \sigma\ :\ e_{i,\psi,\chi_0} \mapsto
\psi(-\lambda(\sigma)\xi)e_{i \sigma^{-1}, 
\psi_{\zeta(\sigma)^{-(q+1)}}, 
\chi_0}.
\end{equation}

Now, we choose an 
element $y_{00}
 \in \mu_{2(q^2-1)}$
such that $y_{00}^{q^2-1}=-1.$
For 
$w=(\psi,\chi) \in 
(\mathbb{F}^{\times}_{q^2})^{\vee} 
\times
 \mathbb{F}^{\vee}_{q^2}$, 
$\zeta \in \mathbb{F}_{q^2} 
\backslash \mathbb{F}_q$
and $\chi_0 \in \mu_{q+1}^{\vee} 
\backslash \{1\}$, 
we define a vector of $W$ as follows
\[
e^w_{\zeta,\chi_0}:=
\sum_{\mu \in \mathbb{F}_{q^2}}
\chi^{-1}(\mu)e_{(\zeta \mu y_{00},\mu y_{00}),
\psi_{-((\mu y_{00})^{q+1}(\zeta^q-\zeta))^{-1}}},
\chi_0 \in W.
\]
%Note that we have a canonical embedding 
%$\Gamma=(\mathcal{O}_E/\pi^2)^{\times} 
%\hookrightarrow \mathcal{O}^{\times}_3.$
For $w \in 
\Gamma^{\vee}_{\rm stp},$
we define a 
subspace of $W$ as follows
\[
W_w:=\bigoplus_{\chi_0 
\in \mu^{\vee}_{q+1}
 \backslash \{1\}}
\bigoplus_{\zeta \in
 \mathbb{F}_{q^2} 
\backslash \mathbb{F}_q}
\overline{\mathbb{Q}}_l
e^w_{\zeta,\chi_0}.
\]
Note that we have 
${\rm dim}\ W_w=q^2(q-1).$
For $g
=\left(
\begin{array}{cc}
a_0+a_1\pi & b_0+b_1\pi \\
c_0+c_1\pi & d_0+d_1\pi
\end{array}
\right)
 \in G_2^F$, 
 we obtain 
 the following by (\ref{da1})
\begin{equation}\label{bg}
g^{-1} e^w_{\zeta,\chi_0}
=\chi(\bar{b}_0
\zeta+\bar{d}_0)
\psi_{\overline{{\rm det}(\bar{g})}^{-1}}
\biggl(-\frac{f((\zeta \mu y_{00},\mu y_{00}),g)}
{(\mu y_{00})^{q+1}(\zeta^q-\zeta)}\biggr)
e^w_{\frac{\bar{a}_0\zeta+\bar{c}_0}
{\bar{b}_0\zeta+\bar{d}_0},\chi_0}.
\end{equation}
The element $-f
((\zeta \mu y_{00},\mu y_{00}),g)/
(\mu y_{00})^{q+1}(\zeta^q-\zeta)$
does not depend 
on $\mu.$
If $g \in N,$ clearly we have 
$f((\zeta \mu y_{00},\mu y_{00}),g)
=\bar{g}(\zeta \mu y_{00},\mu y_{00})$
by $G_0(\zeta 
\mu y_{00},\mu y_{00}) \equiv 0$.
See Proposition \ref{actt} 
for the notations.
Hence, for $g \in N,$  
(\ref{bg}) has the following form 
by (\ref{pi3})
\begin{equation}\label{bg1}
g^{-1} e^w_{\zeta,\chi_0}=
\tilde{\psi}_{\zeta}(g)
e^w_{\zeta,\chi_0}.
\end{equation}
Furthermore, by (\ref{da2}) and (\ref{da3}), 
we obtain 
\begin{equation}\label{daa1}
t e^w_{\zeta,\chi_0}
=w(t)\chi^{-1}_0(t)
e^w_{\zeta,\chi_0},\ 
\sigma 
e^w_{\zeta,\chi_0}=
\chi(\zeta(\sigma))
\psi(\lambda(\sigma))
e^w_{\zeta,\chi_0}=
w \circ \mathbf{a}_E
(\sigma)e^w_{\zeta,\chi_0}
\end{equation}
for $t \in \Gamma 
\subset 
\mathcal{O}^{\times}_3$
and $\sigma \in I_F.$
Hence, $W_w$ is a 
$\mathbf{H}$-subrepresentation 
of
 $W|_{\mathbf{H}}$.

 %We denote by 
 %$\Gamma^{\vee}_{\rm stp}$
%a set of 
%strongly primitive characters. 
 Now, by decomposing
  the $\mathbf{H}$-representation
 $W|_{\mathbf{H}}$ 
 to a direct sum of 
irreducible components 
$\{W_w\}_{w \in 
\Gamma_{\rm stp}^{\vee}}$,
we have the following proposition.
\begin{proposition}\label{gey}
Let the notation be as above.
%Furthermore, let $\Gamma^{\vee}_{\rm stp}
% \subset \Gamma^{\vee}$
%be a subset which consists of strongly 
%primitive characters.
Then, we have the followings
\\1.\ The following isomorphism as a 
$\mathbf{H}$-representation holds
\begin{equation}\label{ci}
W_w \simeq \pi_w^{\vee} \otimes 
\biggl(
\bigoplus_{\chi_0 \in \mu_{q+1}^{\vee} 
\backslash \{1\}}
 \chi^{-1}_0w \biggr)
\otimes (w \circ \mathbf{a}_E).
\end{equation}
\\2.\ We have 
the following isomorphism as a 
$\mathbf{H}$-representation
\[
W \simeq 
\bigoplus_{w \in {\Gamma}^{\vee}_{\rm stp}}
\biggl(\pi_w^{\vee} \otimes 
\biggl(
\bigoplus_{\chi_0 \in \mu_{q+1}^{\vee}
 \backslash \{1\}}\chi^{-1}_0w \biggr)
\otimes (w \circ \mathbf{a}_E)\biggr).
\]
%In particular, 
%the restriction 
%$W|_{G_2^F \times \{1\} \times \{1\}}$
%is a direct sum of all cuspidal representations
%with each multiplicity $q$.
\end{proposition}
\begin{proof}
We prove the first assertion.
We set
 $\mathbf{H}_1:=\Gamma N
  \times \Gamma \times I_F
\subset \mathbf{H}$ 
and $W^{\zeta}_w:
=\overline{\mathbb{Q}}_l
e^w_{\zeta,\chi_0}
  \subset W_w.$
By (\ref{bg}) and (\ref{daa1}), 
the stabilizer in $\mathbf{H}$ 
of a subspace
 $W_w^{\zeta_0}$
 is equal to $\mathbf{H}_1.$
 For $t \in 
 \mu_{q^2-1}(\mathcal{O}_E) 
 \subset
  \Gamma \subset G_2^F$,
 we have $t e^w_{\zeta_0,\chi_0}
 =\chi^{-1}(\bar{t})e^w_{\zeta_0,\chi_0}$. 
 Therefore, 
 by (\ref{bg1}), 
 the subgroup 
 $\Gamma N \subset G_2^F$ 
 acts on $W_w^{\zeta_0}$
via the character $w$ in (\ref{cha}).
 Hence, by (\ref{daa1}), 
 we acquire 
 \[
 W_w^{\zeta_0} 
 \simeq w^{-1} \otimes \chi^{-1}_0w 
 \otimes 
 (w \circ \mathbf{a}_E)
 \]
 as a $\mathbf{H}_1$-representation.
 Since $G_2^F \hookrightarrow 
 \mathbf{H}$ 
 permutes the subspaces
 $\{W^{\zeta}_w\}_{\zeta \in \mathbb{F}_{q^2} 
 \backslash \mathbb{F}_q}$
  transitively by (\ref{bg}), 
  we obtain the following isomorphism
  \[
  W_w \simeq \bigoplus_{\chi_0 
  \in \mu^{\vee}_{q+1} \backslash 
  \{1\}}{\rm Ind}^{\mathbf{H}}
  _{\mathbf{H}_1}(w^{-1} 
  \otimes 
  \chi_0^{-1}w 
  \otimes w \circ 
  \mathbf{a}_E) \simeq 
  \pi_w^{\vee} 
  \otimes \biggl(
  \bigoplus_{\chi_0 
  \in \mu^{\vee}_{q+1} 
  \backslash \{1\}}
  \chi_0^{-1} w \biggr)
  \otimes (w \circ \mathbf{a}_E)
  \]
  as a $\mathbf{H}$-representation.
  Hence, 
  the first assertion follows.
  
  The second assertion follows from
  the first one and the following isomorphism
  \[
  W|_{\mathbf{H}} \simeq \bigoplus_{w \in 
  \Gamma^{\vee}_{\rm stp}}
  W_w
  \] as a $\mathbf{H}$-representation.
Hence, the required assertions follow.
%The assertion $2$ follows from $1$.
%We prove the assertion $1.$
\end{proof}
Let 
$w=(\chi,\psi) \in 
\Gamma^{\vee}_{\rm stp}.$ 
We define a $\mathcal{O}^{\times}_3$-representation
\[
\rho_w:={\rm Hom}_{G_2^F \times I_F}
(\pi_w^{\vee} \otimes (w \circ \mathbf{a}_E),W).
\]
Then, we have 
\begin{equation}\label{i_1}
W \simeq \bigoplus_{w \in \Gamma^{\vee}_{\rm stp}}
\pi_w^{\vee} \otimes \rho_w \otimes (w \circ
 \mathbf{a}_E) 
\end{equation}
as a $\mathbf{G}$-representation.

By combining 
 Proposition \ref{gey} with some fact in 
the representation theory of a finite group in 
\cite[Lemma 16.2]{BH}, we understand $W$ as 
a $\mathbf{G}$-representation 
by the following corollary.
\begin{corollary}\label{lap}
Let the notation be as above.
Let $U:=U^2_D/U_D^3 \subset \Gamma 
\subset \mathcal{O}^{\times}_3$. Note that 
$U \simeq \mathbb{F}_{q^2}.$
The additive character 
$\psi \in \mathbb{F}^{\vee}_{q^2}
 \backslash \mathbb{F}^{\vee}_{q}$ 
 is considered as a character 
 of $U.$
Then, we have the followings:
\\1.\ The $\mathcal{O}^{\times}_3$-representation
$\rho_w$ is irreducible.
Moreover, we have the following isomorphism 
\begin{equation}\label{to}
\rho_w|_{\Gamma} \simeq 
\bigoplus_{\chi_0 \in \mu^{\vee}_{q+1} \backslash \{1\}}
\chi_0^{-1} w
\end{equation}
as a $\Gamma$-representation.
\\2.\ We have the followings
\begin{equation}\label{loy}
{\rm dim}\ \rho_w=q,\ 
\rho_w|_{U} \simeq \psi^{\oplus q},\ 
{\rm Tr}\ \rho_w(\zeta)
 =-\chi(\bar{\zeta})
\end{equation}
for $\zeta \in \mu_{q^2-1}(\mathcal{O}_E) 
 \backslash \mu_{q-1}(\mathcal{O}_F).$
 \end{corollary}
\begin{proof}
The isomorphism (\ref{to})
follows from Proposition \ref{gey},
and (\ref{loy}) follows from (\ref{to})
immediately. Hence, it suffices to prove 
that $\rho_w$ is irreducible.
Let $H:=U_D^1/U_D^3 \subset \mathcal{O}^{\times}_3$
be a $p$-Sylow subgroup of order $q^4$. 
Then, we have $U \subset H.$
By applying \cite[Lemma 16.2]{BH} to the situation
$G=H/{\rm ker}\ \psi, N=U/{\rm ker}\ \psi$, 
we know that there exists a unique irreducible $H$-representation 
$\rho$ of degree $q$
 such that $\rho|_{U}$ is a multiple of $\psi.$
 Hence, by ${\rm dim}\ \rho_w=q$, the $\rho_w$ must be isomorphic
 to $\rho$, and hence irreducible.
Thereby, we have proved the required assertion.
\end{proof}

%%%%%%%%%%%%%%%%%%%%%%%%%%%%%%%%%%%%%%%%%%%%%%
%%%%%%%%%%%%%%%%%%%%%%%%%%%%%%%%%%%%%%%%%%%%%%

\subsection{Analysis of $\bigoplus_{(i,j,y_0) \in 
\mathcal{S}_1}
H^1(W^{i,c}_{j,y_0},\overline{\mathbb{Q}}_l)$}
\label{acu2}
Recall that, in subsections 
\ref{zo2} and \ref{zo3}, 
we have shown that 
the following curves
\[\{W^{i,c}_{j,y_0}\}_{(i,j,y_0)
 \in \mathcal{S}_1}\]
 with each having an affine model $a^q-a=s^2,$
 appear in the stable reduction of
  $\mathcal{X}(\pi^2)$.
In this subsection, we analyze 
the following 
cohomology group
\begin{equation}
W':=\bigoplus_{(i,j,y_0) \in 
\mathcal{S}_1}
H^1(W^{i,c}_{j,y_0},\overline{\mathbb{Q}}_l).
\end{equation}
Then, we understand $W'$ 
as a $\mathbf{G}$-representation
 very explicitly in
Corollary \ref{dek2}.
The $\mathbf{G}$-representation
$W'$ is related to 
ramified representation
of ${\rm GL}_2(F)$ of 
normalized level $1/2$.
See \ref{con} for a precise statement.

Let $Y$ be the smooth compactification
of an affine curve $Y_0:a^q-a=s^2.$
In the following, to analyze $W'$, 
we will 
show some elementary facts on
 $H^1(Y,\overline{\mathbb{Q}}_l)$
 in Lemma \ref{el1}.
We have ${\rm dim}\ H^1(Y,\overline{\mathbb{Q}}_l)=q-1$,
because the genus of $Y$ is $(q-1)/2.$
The complement $Y \backslash Y_0$
consists of one point.
Furthermore, we have
$H^1(Y,\overline{\mathbb{Q}}_l)
\simeq H_c^1(Y_0,\overline{\mathbb{Q}}_l).$
The curve $Y_0$ is a finite Galois \'{e}tale 
covering of $\mathbb{A}^1$ of Galois group 
$\mathbb{F}_q$ 
by $(a,s) \mapsto s.$
Then, we consider $H^1(Y,\overline{\mathbb{Q}}_l)$
as a $\overline{\mathbb{Q}}_l[\mathbb{F}_q]$-module.
Let $\alpha$ be an automorphism
of $Y_0$ such that $(a,s) \mapsto (a+1,s).$
On the other hand, a group  
$\mu_{2(q-1)} \ni b$ acts on $Y_0$ as follows
$\beta_b:(a,s) \mapsto (b^2a,bs).$
Note that the automorphism group of 
$Y_0$ is generated by $\alpha$ and $\beta_b$
with $b \in \mu_{2(q-1)}.$
See also \cite[Lemma 6.12]{CM2}
for the automorphism of 
the Artin-Schreier curve $a^q-a=s^2.$
For $\psi \in 
\mathbb{F}^{\vee}_q$,
let $\mathcal{L}_{\psi}(s^2)$
be the smooth $\overline{\mathbb{Q}}_l$-sheaf
on $\mathbb{A}^1$ defined by the covering 
$Y_0$ and $\psi.$

We introduce the following elementary lemma.
\begin{lemma}\label{el1}
Let the notation be as above.
\\1.\ Then, we have 
\[
H^1(Y,\overline{\mathbb{Q}}_l) \simeq \bigoplus_{\psi 
\in \mathbb{F}^{\vee}_q \backslash \{0\}}\psi=:V
\]
as a $\overline{\mathbb{Q}}_l[\mathbb{F}_q]$-module.
We write $\{e_{\psi}\}_{\psi 
\in \mathbb{F}^{\vee}_q \backslash \{0\}}$ 
for the basis of $V$ above.
\\2.\ Let $b \in \mu_{2(q-1)}$. 
For a character $\psi \in \mathbb{F}^{\vee}_q$ and 
$x \in \mathbb{F}^{\times}_q$, we write
 $\psi_x \in \mathbb{F}^{\vee}_q$
for a character $y \mapsto \psi(xy).$ 
Then, the automorphism $\beta_b$
 of $Y_0$ 
induces the following action on $V$
\[
\beta_b\ :\ e_{\psi} \mapsto c_{\psi,b}e_{\psi_{b^{-2}}}\]
with some constant 
$c_{\psi,b} \in \overline{\mathbb{Q}}^{\times}_l.$
Furthermore, we have 
$c_{\psi,-1}=1.$
\end{lemma}
\begin{proof}
We have 
$H_c^1(Y_0,\overline{\mathbb{Q}}_l)
 \simeq \bigoplus_{\psi \in \mathbb{F}^{\vee}_q
\backslash 
\{0\}}
H_c^1(\mathbb{A}^1,\mathcal{L}_{\psi}(s^2))$
as a $\overline{\mathbb{Q}}_l[\mathbb{F}_q]$-module.
By the Grothendieck-Ogg-Shafarevich formula, 
we have ${\rm dim}\ 
H_c^1(\mathbb{A}^1,\mathcal{L}_{\psi}(s^2))=1$
and $H_c^1(\mathbb{A}^1,\mathcal{L}_{\psi}(s^2)) \simeq \psi$
as a $\overline{\mathbb{Q}}_l[\mathbb{F}_q]$-module.
Hence, the first assertion follows.

The second assertion follows from 
$\beta_b \alpha \beta^{-1}_b
=\alpha^{b^2}$ for any $b \in \mu_{2(q-1)}$.
The assertion $c_{\psi,-1}=1$ follows from the 
Lefschetz trace formula.
Hence, we have proved the required assertions.
\end{proof}

By ${\rm dim}\
 H^1(W^{i,c}_{j,y_0},\overline{\mathbb{Q}}_l)=q-1,$
 we have 
${\rm dim}\ W'=2q(q-1)(q^2-1)^2$.
We set
\[
\mathcal{T}:=\mathcal{S}_1 \times 
(\mathbb{F}^{\vee}_q 
\backslash \{0\})=
\mathbb{F}^{\times}_q \times 
\mathbb{F}_q \times 
\mathbb{P}^1(\mathbb{F}_q) 
\times \mu_{2(q^2-1)} \times 
(\mathbb{F}_q^{\vee} \backslash \{0\})
\]
Then, by Lemma \ref{el1}.1 and 
the identification
$(\mathcal{O}_F/\pi^2)
^{\times} \simeq 
\mathbb{F}_q^{\times} 
\times 
\mathbb{F}_q$, 
we have the following isomorphism
\begin{equation}\label{thq}
W' \simeq \bigoplus
_{((\zeta,\tilde{\mu}),j,y_0,\psi) \in 
\mathcal{T}}\overline{\mathbb{Q}}_l
e_{(\zeta,\tilde{\mu}),j,y_0,\psi}
\end{equation}
as a 
$\overline{\mathbb{Q}}_l$-vector space.
In the following, by 
the $\mathbf{G}$-action on $W'$, 
we consider
 the right hand
  side in 
  (\ref{thq})
 as a 
 $\mathbf{G}$-representation.

In the following, we 
define $\mathbf{G}$-subrepresentations 
$W^w_a \subset W'$ for $a \in \mathbb{F}_q^{\times}$
and $w \in ((\mathcal{O}_F/\pi^2)^{\times})^{\vee}$, 
and investigate a shape of $W^w_a$ as a 
$\mathbf{G}$-representation
in 
Proposition \ref{dek}.
As a result, we understand $W'$
very explicitly in Corollary \ref{dek2}.

Let $w=(w_1,w_2) 
\in ((\mathcal{O}_F/\pi^2)^{\times})^{\vee}
 \simeq 
(\mathbb{F}_q^{\times})^{\vee} 
\times \mathbb{F}_q^{\vee}$, 
 $\psi \in \mathbb{F}_q^{\vee} 
 \backslash \{0\}$, 
 $\zeta \in \mathbb{F}_q^{\times}$, 
 $j \in \mathbb{P}^1(\mathbb{F}_q)$
  and $y_0 \in \mu_{2(q^2-1)}.$
We define a vector of $W'$
under the isomorphism 
(\ref{thq})
\[
e^w_{\zeta,y_0,j,\psi}:=
\sum_{(\mu,\tilde{\mu}) \in 
\mathbb{F}_q^{\times} \times \mathbb{F}_q}
w_1^{-1}(\mu)w_2^{-1}(\tilde{\mu})
e_{(\mu^2\zeta ,\tilde{\mu}),j,\mu y_0,\psi}.
\]
Then, clearly we have 
\begin{equation}\label{vb}
e^w_{{\mu_1}^2\zeta,\mu_1 y_0,j,\psi}
=w_1(\mu_1)e^w_{\zeta,y_0,j,\psi}
\end{equation}
for any $\mu_1 \in \mathbb{F}_q^{\times}$.

We consider a set $\mathbb{F}^{\times}_q
\times \mu_{2(q^2-1)}$
and the following equivalence relation
on the set:
\[
(\zeta,y_0) \sim (\zeta',y_0') \Leftrightarrow 
(\zeta',y_0')=(\mu^2 \zeta,\mu y_0) 
\]
for some $\mu \in
 \mathbb{F}^{\times}_q.$
Let $\mathcal{U}:=(\mathbb{F}^{\times}_q 
\times \mu_{2(q^2-1)})/\sim.$
Then, we have $|\mathcal{U}|=2(q^2-1).$ 
We write $[(\zeta,y_0)] \in \mathcal{U}.$
For each 
$a \in \mathbb{F}^{\times}_q,$
we set
\[
\mathcal{K}_a:=\{
[(\zeta,y_0)] \in \mathcal{U}\ |\ 
a=\zeta^2y_0^{-2(q+1)}
\} \subset \mathcal{U}.
\]
Then,
 we have $|\mathcal{K}_a|=2(q+1).$
For $w \in 
((\mathcal{O}_F/\pi^2)
^{\times})^{\vee}$, 
$[(\zeta,y_0)] \in \mathcal{U}$,   
$j \in \mathbb{P}^1(\mathbb{F}_q)$
 and $\psi \in \mathbb{F}_q^{\vee} 
\backslash \{0\}$, 
We define $W^w_{[(\zeta,y_0)],j,\psi}:=
\overline{\mathbb{Q}}_l
e^w_{\zeta,y_0,j,\psi} \subset W'.$
This one-dimensional 
$\overline{\mathbb{Q}}_l$-vector 
subspace depends on
$(w,[(\zeta,y_0)],j,\psi)$
by (\ref{vb}).

We fix a character $\psi \in \mathbb{F}_q^{\vee} 
\backslash \{0\}$.
We set $\mathcal{T}_1:=\mathbb{F}^{\times}_q \times 
((\mathcal{O}_{F}/\pi^2)^{\times})^{\vee}.$
Obviously, we have $|\mathcal{T}_1|=q(q-1)^2.$
Let $(a,w) \in \mathcal{T}_1.$
We write
 $w=(w_1,w_2) 
\in ((\mathcal{O}_F/\pi^2)^{\times})^{\vee} \simeq 
(\mathbb{F}_q^{\times})^{\vee} 
\times \mathbb{F}_q^{\vee}$.
Then, we define a $\mathbf{G}$-subrepresentation of $W'$
as follows
\begin{equation}\label{w_1}
W^w_{a}:=
\bigoplus_{[(\zeta,y_0)] \in \mathcal{K}_a}
\bigoplus_{j \in \mathbb{P}^1(\mathbb{F}_q)}
\bigoplus_{\zeta_1 \in \mathbb{F}^{\times}_q}
W^w_{[(\zeta_1\zeta,y_0)],j,\psi_{\zeta_1^{-1}}} \subset W'.
\end{equation}
Then, we have ${\rm dim}\ W_a^w=2(q+1)(q^2-1).$
Then, by (\ref{thq}), 
we easily check 
the following isomorphism
\begin{equation}\label{vb2}
W' \simeq 
\bigoplus_{(a,w) \in \mathcal{T}_1}W^{w}_{a}
\end{equation}
as a $\overline{\mathbb{Q}}_l$-vector space.
By the action of $\mathbf{G}$ on $W'$, 
we consider 
the right hand side of (\ref{vb2}) as 
a $\mathbf{G}$-representation.
Then, the subspace 
$W^w_a$ is $\mathbf{G}$-stable
for each $(a,w) \in \mathcal{T}_1$.
In the following, we will explicitly write 
down the action of $\mathbf{G}$ on $W^w_a.$
To do so, we prepare some notations.

We fix elements 
$(a,(w_1,w_2)) \in \mathcal{T}_1$
and $[(\zeta,y_0)] \in \mathcal{K}_a.$ 
%In the following, we investigate $W_a^w$ as a 
%$\mathbf{G}$-representation for each 
%$(a,w) \in \mathcal{T}_1.$
Let $\tilde{\zeta}, 
\tilde{a} \in \mu_{q-1}(\mathcal{O}_F)$
be the unique liftings of $ 
\zeta,a \in \mathbb{F}^{\times}_q$ 
respectively.
 Let $E/F$ be the unramified 
 quadratic extension as before.
%Let $\tilde{y_0}^2 \in 
%\mu_{q^2-1}(\mathcal{O}_{E})$ be the unique lifting 
%of $y_0^2 \in \mathbb{F}^{\times}_{q^2}.$ 
Let 
\begin{equation}\label{aa1}
\alpha:=\left(
\begin{array}{cc}
0 &  1\\
\pi \tilde{a} & 0
\end{array}
\right) \in {\rm GL}_2(F).
\end{equation} 
Of course, we have $\alpha^2=\pi \tilde{a}.$
We set ${E_1}:=F(\alpha) \subset {\rm GL}_2(F).$
Then, ${E_1}$ is a quadratic 
ramified extension of $F.$
We choose an embedding 
\[
{E_1} \hookrightarrow D\ :\ 
a_0+b_0\alpha \mapsto a_0+b_0 \alpha' 
\]
for $a_0,b_0 \in F.$
Note that we have 
$\alpha'/\varphi \equiv {\zeta}/
y_0^2\ ({\rm mod}\ \varphi).$
Let 
\[
{\rm M}_2(F) \supset 
\mathfrak{U}:=
\left(
\begin{array}{cc}
\mathcal{O}_F & \mathcal{O}_F \\
\mathfrak{p}_F & \mathcal{O}_F
\end{array}
\right) \supset 
\mathfrak{B}_{\mathfrak{U}}:=
\left(
\begin{array}{cc}
\mathfrak{p}_F & \mathcal{O}_F \\
\mathfrak{p}_F & \mathfrak{p}_F
\end{array}
\right),\ U_{\mathfrak{U}}^n:
=1+\mathfrak{B}^n_{\mathfrak{U}}
 \subset {\rm GL}_2(\mathcal{O}_F)
\]
for $n \geq 1.$
Note that $\mathfrak{U} \subset 
{\rm M}_2(F)$ is a chain order
and $\mathfrak{B}_{\mathfrak{U}}$
is the Jacobson radical 
of the order.
See \cite[Section 12]{BH} 
for more details.
Note that $U_{\mathfrak{U}}^n$
is a compact open subgroup of 
${\rm GL}_2(F)$.
The image of $\mathfrak{U}$ and 
${U}^n_{\mathfrak{U}}$
under 
${\rm GL}_2(\mathcal{O}_F) 
\to G_2^F$
is denoted by $\overline{\mathfrak{U}}$
and $\overline{U}^n_{\mathfrak{U}}$
 respectively.
Similarly, we denote by 
$\overline{\mathcal{O}}^{\times}_{E_1}$
the image of $\mathcal{O}^{\times}_{E_1} 
\subset {\rm GL}_2(\mathcal{O}_F)$
under ${\rm GL}_2(\mathcal{O}_F) 
\to G_2^F$.
The subgroup 
$\overline{\mathcal{O}}^{\times}_{E_1}
\overline{U}^1_{\mathfrak{U}}
\subset 
\overline{\mathfrak{U}}^{\times}$
 is a normal subgroup 
and its index is equal to $q-1.$
Thereby, 
we have 
$[G_2^F:
\overline{\mathcal{O}}^{\times}_{E_1}
\overline{U}^1_{\mathfrak{U}}]=q^2-1.$
Similarly, let $\overline{\mathcal{O}}
^{\times}_{E_1} 
\subset 
\mathcal{O}^{\times}_3$ and 
$\overline{U}^n_D \subset 
\mathcal{O}^{\times}_3$
be the images of
 $\mathcal{O}^{\times}_{E_1} 
\hookrightarrow 
\mathcal{O}^{\times}_D$ and 
$U_D^n \subset 
\mathcal{O}^{\times}_D$ 
under the canonical map
$\mathcal{O}^{\times}_D \to 
\mathcal{O}^{\times}_3$
respectively.
Then, we have $
[\mathcal{O}^{\times}_3:
\overline{\mathcal{O}}^{\times}_{E_1}
\overline{U}_D^1]=q+1.$

Now, we describe
 the $\mathbf{G}$-action on $W^w_a.$
\paragraph{$G_2^F$-action}
First, we consider the action of 
$G_2^F.$
Let $g=
\left(
\begin{array}{cc}
a_0+a_1\pi & b_0+b_1\pi \\
c_1\pi & d_0+d_1\pi
\end{array}
\right)
 \in \overline{\mathfrak{U}}^{\times}$
 with $a_i,b_i,c_1,d_i \in 
 \mu_{q-1}(\mathcal{O}_F)
 \cup \{0\}\ (i=0,1).$ 
Let $\tilde{w}_2:G_2^F \to 
\overline{\mathbb{Q}}^{\times}_l$ be a character 
defined by 
the composite of 
${\rm det}:G_2^F 
\to (\mathcal{O}_F/\pi^2)^{\times}$
and $(\mathcal{O}_F/\pi^2)^{\times} \simeq
 \mathbb{F}^{\times}_q \times \mathbb{F}_q 
 \overset{\rm pr_2}{\to} \mathbb{F}_q \overset{w_2}{\to} 
 \overline{\mathbb{Q}}^{\times}_l$.
% We write ${\rm det}\ (g)=a_0d_0(1+\Delta \pi) \in
%  (\mathcal{O}_F/\pi^2)^{\times}$.
% Then, we have 
%  $\bar{\Delta}=\bar{a}_1\bar
%  {a}_0^{-1}+\bar{d}_1\bar{d}_0^{-1}
%-\bar{b}_0\bar{c}_1(\bar{a}_0\bar{d}_0)^{-1}$
 %in $\mathbb{F}_q$
 Let $b \in \mu_{2(q^2-1)}$
  be an element such that 
 $b^2=\bar{a}_0/\bar{d}_0.$
 Let $[(\zeta',y'_0)] \in \mathcal{K}_a.$
 Then, $g^{-1}$ acts on $W_a^w$ as follows by 
 Proposition \ref{zg1} and 
 Lemma \ref{el1}.2
 \begin{equation}\label{zzg1}
g^{-1}: e^w_{\zeta'\zeta_1,y'_0,j,\psi_{\zeta_1^{-1}}} 
\mapsto
 c_{\psi_{\zeta_1^{-1}},b}w_1(\bar{d}_0)
 \tilde{w}_2(g)
 \psi_{\zeta_1^{-1}}(\bar{c}_1\bar{a}_0^{-1}
 +\bar{b}_0\bar{d}_0^{-1}\zeta_1^2a)
 e^w_{ \bar{a}_0\bar{d}_0^{-1}\zeta' \zeta_1,
 y'_0,j,\psi_{(\bar{a}_0\bar{d}_0^{-1}\zeta_1)^{-1}}}.
 \end{equation}
By (\ref{zzg1}), the quotient 
$\overline{\frak{U}}^{\times}
/\overline{\mathcal{O}}^{\times}_{E_1}
\overline{U}^1_{\mathfrak{U}}$ 
acts on the index set
$\zeta_1 \in \mathbb{F}_q^{\times}$ 
of $W^w_a$ 
simply transitively.
In particular, let $g=
\left(
\begin{array}{cc}
a_0+a_1\pi & b_0+b_1\pi \\
c_1\pi & a_0+d_1\pi
\end{array}
\right)
 \in \overline{\mathcal{O}}^{\times}_{E_1}
\overline{U}^1_{\mathfrak{U}} \subset 
\overline{\mathfrak{U}}^{\times}.$ 
Then, $g^{-1}$ 
acts on $W_a^w$
as follows by (\ref{zzg1}) and Lemma \ref{el1}.2
\begin{equation}\label{xz1}
g^{-1}: e^w_{\zeta'\zeta_1,y'_0,j,\psi_{\zeta_1^{-1}}} 
\mapsto
 w_1(\bar{a}_0)\tilde{w}_2(g)
 \psi_{\zeta_1^{-1}}
 (\bar{a}_0^{-1}(\bar{c}_1+\bar{b}_0a\zeta_1^2))
 e^w_{\zeta' \zeta_1,
 y'_0,j,\psi_{\zeta_1^{-1}}}.
\end{equation}
Thereby, the restriction
 $W_a^w|_{\overline{\mathcal{O}}^{\times}_{E_1}
\overline{U}^1_{\mathfrak{U}} 
\times \{1\} \times \{1\}}$
is a direct sum of characters.
 Let us define characters  
 \begin{equation}\label{c_1}
 \Lambda_{w_1,a}\ :\ \overline{\mathcal{O}}^{\times}_{E_1}
\overline{U}^1_{\mathfrak{U}} \to 
\overline{\mathbb{Q}}^{\times}_l:
g \mapsto w_1(\bar{a}_0)
 \psi(\bar{a}_0^{-1}(\bar{c}_1+\bar{b}_0a)).
 \end{equation} 
 and 
\begin{equation}\label{ch1}
 \tilde{w}_2\Lambda_{w,a}\ :\ 
 \overline{\mathcal{O}}^{\times}_{E_1}
\overline{U}^1_{\mathfrak{U}} \to 
\overline{\mathbb{Q}}^{\times}_l:
g \mapsto \Lambda_{w_1,a}(g)\tilde{w}_2(g).
 \end{equation}
%Let $w:\overline{\mathcal{O}}^{\times}_E \to 
%\overline{\mathbb{Q}}^{\times}_l$
%be a character defined by
%\\By (\ref{xz1}), $g=
%\left(
%\begin{array}{cc}
%1+a_1\pi & b_0+b_1\pi \\
%c_1\pi & 1+d_1\pi
%\end{array}
%\right)
% \in 
%\overline{U}^1_{\mathfrak{U}}$ acts on $W_a^w$
%as follows
%\[
%g: e^w_{\zeta\zeta_1,y'_0,j,\psi_{\zeta_1^{-1}}} 
%\mapsto
% \tilde{w}_2(g)\psi(\bar{c}_1+\bar{b}_0a)
% e^w_{\zeta \zeta_1,
% y'_0,j,\psi_{\zeta_1^{-1}}}.
%\]

\paragraph{$\mathcal{O}^{\times}_3$-action}
Secondly, we consider the action of 
$\mathcal{O}^{\times}_3$
on $W_a^w.$
Let $b:=a_0+\varphi b_0+\pi a_1 \in 
\mathcal{O}^{\times}_3$ with $a_0 \in 
\mu_{q^2-1}(\mathcal{O}_{E})$ 
and $a_1,b_0 \in \mu_{q^2-1}(\mathcal{O}_{E}) \cup 
\{0\}.$
%with $a_0 \in \overline{\mathcal{O}}_F.$
Let $\tilde{w}_2^D:\mathcal{O}^{\times}_3 \to 
\overline{\mathbb{Q}}^{\times}_l$
be the composite
 of ${\rm Nrd}_{D/F}:\mathcal{O}^{\times}_3 
 \to (\mathcal{O}_F/\pi^2)^{\times}$
and $(\mathcal{O}_F/\pi^2)^{\times} 
\overset{\rm pr_2}{\to} 
\mathbb{F}_q \overset{w_2}{\to} 
\overline{\mathbb{Q}}^{\times}_l$.
Then, $b$ acts on $W_a^w$ as follows
by Proposition \ref{op1}
\[
b:e^w_{\zeta'\zeta_1,y'_0,j,\psi_{{\zeta_1}^{-1}}} 
\mapsto
\tilde{w}_2^D(b)
\psi_{\zeta_1^{-1}}({\rm Tr}_{\mathbb{F}_{q^2}/\mathbb{F}_q}
(\bar{b}_0\bar{a}^{-1}_0{y'_0}^{-2q})\zeta'\zeta_1)
e^w_{\zeta'\zeta_1\bar{a}_0^{q+1},
y'_0\bar{a}_0,j,\psi_{{\zeta_1}^{-1}}}.
\]
In particular, 
if $b=a_0+\varphi b_0+\pi a_1 \in 
\overline{\mathcal{O}}^{\times}_{E_1}
\overline{U}^1_{D}$
with $a_0 \in \mu_{q-1}(\mathcal{O}_F),$
$b$ acts on $W^w_a$ as follows by (\ref{vb}) 
\begin{equation}\label{dx1}
b:e^w_{\zeta'\zeta_1,y'_0,j,\psi_{{\zeta_1}^{-1}}} \mapsto
w_1(\bar{a}_0)\tilde{w}_2^D(b)
\psi_{\zeta_1^{-1}}({\rm Tr}
_{\mathbb{F}_{q^2}/
\mathbb{F}_q}(\bar{b}_0\bar{a}^{-1}_0
{y'_0}^{-2q})\zeta'\zeta_1)
e^w_{\zeta'\zeta_1,y'_0,j,\psi_{{\zeta_1}^{-1}}}.
\end{equation}
Hence, $W_a^w|_{\{1\} \times \overline{\mathcal{O}}
^{\times}_{E_1}\overline{U}^1_{D} \times \{1\}}$
is a direct sum of characters.
Let us define characters
\begin{equation}\label{c_2}
\Lambda_{w_1,a}^D:\overline{\mathcal{O}}
^{\times}_{E_1}\overline{U}^1_{D} \to 
\overline{\mathbb{Q}}^{\times}_l:
b \mapsto w_1(\bar{a}_0)
\psi({\rm Tr}
_{\mathbb{F}_{q^2}/
\mathbb{F}_q}(\bar{b}_0
\bar{a}^{-1}_0y_0^{-2q})\zeta)
\end{equation}
and 
\begin{equation}\label{ch2}
\tilde{w}_2^D
\Lambda_{w_1,a}^D\ :\ 
\overline{\mathcal{O}}
^{\times}_{E_1}\overline{U}^1_{D} \to 
\overline{\mathbb{Q}}^{\times}_l:
b \mapsto \tilde{w}^D_2(b)
\Lambda_{w_1,a}^D(b).
\end{equation}
%Furthermore, $b=1+\varphi b_0+\pi 
%a_1 \in \overline{U}^1_{D}$
%acts on $W_a^w$ as follows
%\[
%b:e^w_{\zeta,y'_0,j,\psi} \mapsto
%w_2(\bar{\Delta'})
%\psi(-{\rm Tr}_{\mathbb{F}_{q^2}/
%\mathbb{F}_q}(\bar{b}_0{y'_0}^{-2q})\zeta)
%e^w_{\zeta,y'_0,j,\psi}.
%\]

\paragraph{Inertial action}
Finally, we consider the action of inertia on $W_a^w.$
Recall the following map in (\ref{van1})
\[
\mathbf{a}_{E_1}:I_{E_1}^{\rm ab} \to 
(\mathcal{O}_{E_1}/\pi)^{\times} \simeq 
\mathbb{F}^{\times}_q \times \mathbb{F}_q\ ;\ 
\sigma \mapsto (\zeta(\sigma),\lambda(\sigma)).
\]
Let $\tilde{w}'_2:
(\mathcal{O}_{E_1}/\pi)^{\times} 
\to 
\overline{\mathbb{Q}}^{\times}_l$
be the composite 
$(\mathcal{O}_{E_1}/\pi)^{\times} 
\overset{\rm Nr_{{E_1}/F}}{\to} 
(\mathcal{O}_F/\pi^2)^{\times} 
\overset{\rm pr_2}{\to} \mathbb{F}_q 
\overset{w_2}{\to} \overline{\mathbb{Q}}^{\times}_l.$
We write $\tilde{w}'_2 \circ \mathbf{a}_{E_1}$
for the composite of $I^{\rm ab}_{E_1}
 \overset{\mathbf{a}_{E_1}}{\to}
  \mathcal{O}^{\times}_{E_1} \to 
  (\mathcal{O}_{E_1}/\pi)^{\times}$ and $\tilde{w}'_2.$
Let
\[
\mathbf{a}_F:I_F^{\rm ab} \to 
\mathcal{O}^{\times}_F \to 
(\mathcal{O}_F/\pi^2)^{\times}
\simeq \mathbb{F}^{\times}_q \times \mathbb{F}_q
\]
\[
\sigma \mapsto (\zeta_0(\sigma),\lambda_0(\sigma))
\]
be the reciprocity map 
in (\ref{van0}).
Let $\sigma \in I_F$ 
and $\kappa$ an element 
such that $\kappa^{2q^3(q-1)}=\pi.$ 
We write $\sigma(\kappa)=\zeta_1(\sigma)\kappa$
with $\zeta_1(\sigma) \in \mu_{2q^3(q-1)}.$
Let the notation be as 
in Lemma \ref{ci2''}.
Then, $\sigma$ acts on $W^w_a$ as follows 
by Lemma \ref{ci2''}
\[
\sigma:e^w_{\zeta_1\zeta',y'_0,j,\psi_{{\zeta_1}^{-1}}} \mapsto
w_2(\lambda_0(\sigma))
\psi_{\zeta_1^{-1}}(a_0b_0)
c_{\psi_{\zeta_1^{-1}},c_0}
e^w_{\zeta'\zeta_1
\bar{\zeta}_1(\sigma)^2,
\bar{\zeta}_1(\sigma)y'_0,j,
\psi_{\bar{\zeta}_1(\sigma)^{q-1}\zeta_1^{-1}}}.
\]
Note that we have $\bar{\zeta}_1(\sigma)^2
=\zeta_0(\sigma)$.
If $\sigma \in I_{E_1}$, then, we have 
$\zeta(\sigma)=\bar{\zeta}_1(\sigma) 
\in \mathbb{F}^{\times}_q$ and $c_0 \in \{\pm 1\}.$
Hence, in particular, 
$\sigma \in I_{E_1}$ acts on $W_a^w$ as follows
by Corollary \ref{zi1},
 Lemma \ref{el1} and (\ref{vb})
\begin{equation}\label{ix1}
\sigma:e^w_{\zeta_1\zeta',
y'_0,j,\psi_{\zeta_1^{-1}}} \mapsto
w_1(\zeta(\sigma))
\tilde{w}'_2 \circ \mathbf{a}_{E_1}(\sigma)
\psi_{\zeta_1^{-1}}(-y_0^{q^2-1}2a\lambda(\sigma)\zeta_1^2)
e^w_{\zeta_1\zeta',y'_0,j,\psi_{\zeta_1^{-1}}}.
\end{equation}
Note that we have $y_0^{q^2-1} \in \{\pm1\}.$
Let us define characters 
\begin{equation}\label{c_3}
\Lambda'_{w_1,a}\ :\ 
(\mathcal{O}_{E_1}/\pi)^{\times} 
\to \overline{\mathbb{Q}}^{\times}_l:
\zeta_0+\lambda_0 \alpha 
\mapsto w_1(\bar{\zeta}_0)
\psi(2a\frac{\bar{\lambda}_0}{\bar{\zeta}_0})
\end{equation}
with $\zeta_0 \in \mu_{q-1}(\mathcal{O}_{E_1}),
\lambda_0 \in \mu_{q-1}(\mathcal{O}_{E_1}) \cup 
\{0\}$
and 
\begin{equation}\label{ch3}
(\tilde{w}'_2 
\Lambda'_{w_1,a}) \circ \mathbf{a}_{E_1}:
I_{E_1}^{\rm ab} \to \overline{\mathbb{Q}}^{\times}_l:
\tilde{w}'_2 \circ \mathbf{a}_{E_1}(\sigma)
\Lambda_{w_1,a} \circ \mathbf{a}_{E_1}(\sigma).
\end{equation}

\begin{lemma}\label{keyr}
Let $\Lambda_{w_1,a}$, 
$\Lambda^D_{w_1,a}$ 
and $\Lambda'_{w_1,a}$ 
be the 
characters defined 
 in (\ref{c_1}), 
 (\ref{c_2}) and (\ref{c_3})
  respectively.
Then, we have the followings;
\\1.\ The restriction $\Lambda_1$ 
of $\Lambda_{w_1,a}$
 to a subgroup $\overline{\mathcal{O}}
 ^{\times}_{E_1}$
and the restriction 
$\Lambda_2$ of  
$\Lambda_{w_1,a}^D$
to a subgroup 
$\overline{\mathcal{O}}^{\times}_{E_1}$
factor through 
$(\mathcal{O}_{E_1}/\pi)^{\times}$.
Moreover, we have 
$\Lambda_1=\Lambda_2=\Lambda'_{w_1,a}$.
%as character of $(\mathcal{O}_{E_1}/\pi)^{\times}.$
\\2.\ The restriction of the character $\Lambda_{w_1,a}$ 
to a subgroup $\overline{U}^1_{\mathfrak{U}}$
is given by $g \mapsto \psi({\rm Tr}
(\frac{\alpha}{\pi}(g-1)))$,
where ${\rm Tr}$ is the composite $G_2^F 
\overset{\rm Trace}{\longrightarrow}
\mathcal{O}_F/\pi^2 \overset{\rm can.}{\to} \mathbb{F}_q.$
Similarly, the restriction of 
the character $\Lambda_{w_1,a}^D$
to a subgroup $\overline{U}^1_{D}$
is given by $b \mapsto 
\psi(\Tilde{{\rm Trd}}_{D/F}
(\frac{\alpha'}{\pi}(b-1)))$,
where $\Tilde{\rm Trd}_{D/F}$
is the composite 
$\mathcal{O}_3^{\times} \overset{\rm 
Trace}{\to} \mathcal{O}_F/\pi^2 
\overset{\rm can.}{\to} \mathbb{F}_q.$
The restriction of $\Lambda'_{w_1,a}$
to a subgroup $x \in \overline{U}^1_{E_1}$
is given by $\psi 
\circ {\rm Tr}_{{E_1}/F}(\frac{\alpha}{\pi}(x-1)).$
\end{lemma}
\begin{proof}
The required assertions 
are checked by direct computations. 
\end{proof}
\begin{remark}
For the meaning of the above lemma,
see \cite[19.2 and 19.3]{BH} and \cite[56.5]{BH}.
\end{remark}

\begin{proposition}\label{dek}
Let the notation be as in 
(\ref{w_1}), 
(\ref{ch1}), (\ref{ch2}) and (\ref{ch3}).
We set
\[
\pi_{w,a}:={\rm Ind}^{G^F_2}_{\overline{
\mathcal{O}}_{E_1}^{\times} \overline{U}^1_{\mathfrak{U}}}
(\tilde{w}'_2 \Lambda_{w_1,a}),\ 
\rho_{w,a}:=
{\rm Ind}^{\mathcal{O}^{\times}_3}
_{\overline{\mathcal{O}}_{E_1}^{\times}
\overline{U}^1_{D}}
(\tilde{w}_2^D\Lambda_{w_1,a}^D),\ 
\pi'_{w,a}:={\rm Ind}_{{E_1}/F}((\tilde{w}'_2 \Lambda'_{w_1,a})
 \circ \mathbf{a}_{E_1}).
\]
\\1.\ Then, 
$\pi_{w,a}$ and $\pi'_{w,a}$ are irreducible.
On the other hand, $\rho_{w,a}$ is not 
irreducible.
We have 
${\rm dim}\ \pi_{w,a}=q^2-1$, ${\rm dim}\ \rho_{w,a}=q+1$
and ${\rm dim}\ \pi'_{w,a}=2.$
\\2.\ The following isomorphism
as a $\mathbf{G}$-representation holds
\[
W^w_a \simeq \pi_{w,a}^{\vee} \otimes \rho_{w,a}
\otimes \pi'_{w,a}.
\]
\end{proposition}
\begin{proof}
We set $\xi:=(\tilde{w}'_2 \Lambda'_{w_1,a})
\circ \mathbf{a}_{E_1}.$
Let $\tau \neq 1 \in {\rm Gal}({E_1}/F).$
Then, we easily check $\xi^{\tau} \neq \xi.$
Hence, $\pi'_{w,a}$ is irreducible.
By Mackey's irreducibility 
criterion in 
\cite[Proposition 7.23]{Se},
we check that $\pi_{w,a}$ 
is irreducible.

We show that $\rho_{w,a}$
is not irreducible.
We consider the character (\ref{ch2}).
For simplicity, we set
$\tilde{w}:=\tilde{w}_2^D\Lambda_{w_1,a}^D$.
The group  
$\overline{\mathcal{O}}^{\times}_{E}$
is contained in the stabilizer
 of $\overline{\mathcal{O}}_{E_1}^{\times}
\overline{U}^1_{D}$ in 
$\mathcal{O}_3^{\times}$.
We choose an element 
$s \in \overline{\mathcal{O}}^{\times}_E
\backslash \overline{\mathcal{O}}_{E_1}^{\times}
\overline{U}^1_{D}.$ 
Then, we have $\tilde{w}^s(x)
:=\tilde{w}(s^{-1}xs)=\tilde{w}(x)$ for $x \in
\overline{\mathcal{O}}_{E_1}^{\times}
\overline{U}^1_{D}$.
Hence, again by Mackey's irreducibility criterion,
$\rho_{w,a}$ is not irreducible.

We prove the second assertion. 
We consider the subspace 
$\tilde{W}:=
W^w_{[(\zeta,y_0)],j,\psi} \subset W^w_a$.
The stabilizer of $\tilde{W}$ in $\mathbf{G}$
is equal to 
$\mathbf{H}_1:=\overline{\mathcal{O}}^{\times}_{E_1}
\overline{U}^1_{\mathfrak{U}} \times 
\overline{\mathcal{O}}^{\times}_{E_1}\overline{U}^1_D \times 
I_{E_1}$.
Furthermore, we have 
an isomorphism 
by (\ref{xz1}), (\ref{dx1})
 and (\ref{ix1})
\[
\tilde{W} \simeq (\tilde{w}_2 \Lambda_{w_1,a})
\otimes 
(\tilde{w}_2^D \Lambda_{w_1,a}^D)
\otimes ((\tilde{w}'_2 \Lambda'_{w_1,a})^{\tau^{-y_0^{q^2-1}}}
 \circ \mathbf{a}_{E_1})
\]
as a $\mathbf{H}_1$-representation.
On the other hand, by (\ref{ind}), 
we easily check that $W^w_a$ is a direct sum   
\[
\bigoplus_{[(\zeta,y_0)] \in \mathcal{K}_a}
\bigoplus_{j \in \mathbb{P}^1(\mathbb{F}_q)}
\bigoplus_{\zeta_1 \in \mathbb{F}^{\times}_q}
W^w_{[(\zeta_1\zeta,y_0)],j,\psi_{\zeta_1^{-1}}}
\]
of subspaces 
permuted transitively by the action of $\mathbf{G}.$
Hence, the required assertion follows.
\end{proof}
\begin{corollary}\label{dek2}
Let the notation be as in Proposition \ref{dek}.
Then, we have the following isomorphism
\[
W' \simeq \bigoplus_{(w,a) \in \mathcal{T}_1}
 \pi_{w,a}^{\vee} \otimes \rho_{w,a}
\otimes \pi'_{w,a}
\]
as a $\mathbf{G}$-representation.
\end{corollary}
\begin{proof}
This follows from Proposition \ref{dek} 
and (\ref{vb2}) immediately.
\end{proof}

%%%%%%%%%%%%%%%%%%%%%%%%%%%%%%%%%%%%%%%%%%%%%%
%%%%%%%%%%%%%%%%%%%%%%%%%%%%%%%%%%%%%%%%%%%%%%

\subsection{Conclusion}\label{con}
As mentioned in the 
introduction, we 
investigate a relationship
between our $\mathbf{G}$-representations 
$W$ and $W'$, and the 
local Jacquet-Langlands 
correspondence and the 
local Langlands correspondence.
In the following, 
we often quote several facts
from the book \cite{BH}.
First,
we briefly recall
 the $\ell$-adic local 
Langlands correspondence
and the local Jacquet-Langlands 
correspondence 
for ${\rm GL}_2.$
For example, see 
\cite[Sections 34 and 35]{BH}.

Let $\mathbf{\mathcal{G}}_2(F,\overline{\mathbb{Q}}_l)$
denote the set of equivalence classes of $2$-dimensional, 
semisimple, Deligne representations of the Weil group $W_F,$
over $\overline{\mathbb{Q}}_l.$
See \cite[p.200 and p.221]{BH} for more details.
Let $\mathbf{\mathcal{A}}_2(F,\overline{\mathbb{Q}}_l)$
denote the set of equivalence classes of irreducible 
smooth representations of ${\rm GL}_2(F)$
over $\overline{\mathbb{Q}}_l$.
See \cite[p.212]{BH} for more details.
 
We recall the $\ell$-adic local 
Langlands correspondence in 
\cite[p.222 and p.223]{BH}.
There is a unique bijection
\[
{\rm LL}_{\ell}\ :\ 
\mathbf{\mathcal{G}}_2(F,\overline{\mathbb{Q}}_l)
\longrightarrow 
\mathbf{\mathcal{A}}_2(F,\overline{\mathbb{Q}}_l)
\]
which commutes with automorphisms 
of $\overline{\mathbb{Q}}_l.$
Furthermore, for any isomorphism 
$\iota:\overline{\mathbb{Q}}_l \simeq \mathbb{C}$, 
the correspondence ${\rm LL}_{\ell}$
satisfies
\[
L(\chi{\rm LL}_{\ell}(\sigma)
^{\iota},s)=L(\chi \sigma^{\iota},s-\frac{1}{2})
\]
\[
\epsilon(\chi{\rm LL}_{\ell}
(\sigma)^{\iota},s,\psi)
=\epsilon(\chi \sigma^{\iota},s-\frac{1}{2},\psi)
\]
for all $\sigma \in \mathbf{\mathcal{G}}_2
(F,\overline{\mathbb{Q}}_l)$, all $\chi \in 
(F^{\times})^{\vee}$ and all $\psi \in F^{\vee}.$
See \cite[Section 6]{BH} for L-function and local constants.

Let $\mathbf{\mathcal{A}}_2^{\diamondsuit}
(F,\overline{\mathbb{Q}}_l)$
denote the set of equivalence classes of irreducible 
smooth representations of 
${\rm GL}_2(F)$ which are
 essentially square-integrable, 
 over $\overline{\mathbb{Q}}_l$.
See \cite[17.4 and 56.1]{BH}. 
We write $\mathbf{\mathcal{A}}_1
(D,\overline{\mathbb{Q}}_l)$
for the set of equivalence classes 
of irreducible smooth representations 
of $D^{\times}$ over $\overline{\mathbb{Q}}_l$.
The local Jacquet-Langlands correspondence
is a bijection
\[
{\rm JL}\ :\ 
\mathbf{\mathcal{A}}_2^{\diamondsuit}
(F,\overline{\mathbb{Q}}_l) \longrightarrow 
\mathbf{\mathcal{A}}_1(D,\overline{\mathbb{Q}}_l)
\]
satisfying the appropriate trace identity.
See \cite[p.334 and 56.9]{BH} 
and \cite{JL}
for more details.

\paragraph{unramified case\ :}
We consider the $\mathbf{G}$-representation
$W$ in subsection 
\ref{acc1}.
As in (\ref{i_1}), 
we have the following isomorphism
\[
W \simeq \bigoplus_{w \in \Gamma_{\rm stp}^{\vee}}
\pi_w^{\vee} \otimes \rho_w \otimes \chi \circ \mathbf{a}_{E}
\]
as a $\mathbf{G}$-representation.

Let $\chi:\mathbb{Z} \to \overline{\mathbb{Q}}_l^{\times}$
be a character. 
In the following, 
we identify $F^{\times} 
\simeq \mathbb{Z} \times 
\mathcal{O}^{\times}_F$ by 
$x=\pi^{v(x)}u 
\mapsto (v(x),u).$
We also 
identify $E^{\times} \simeq \mathbb{Z} 
\times \mathcal{O}^{\times}_E$ 
in the same manner as 
$F^{\times}$.
Then, we define
 an irreducible and 
  cuspidal representation 
 $\pi_{w,\chi}$
of ${\rm GL}_2(F)$ of level $1$ 
in the following.
The level means the normalized level 
in \cite[p.91]{BH}.
First, we also write $\pi_w$ for the inflation 
to ${\rm GL}_2(\mathcal{O}_F)$ of
 the ${G}_2^F$-representation
$\pi_w.$ We extend $\pi_w$ to 
a representation of 
$F^{\times}
{\rm GL}_2(\mathcal{O}_F)$
by using $\chi,$ which we denote 
by $\pi_{w} \otimes 
\chi.$
Then, we define
$\pi_{w,\chi}:=
{\rm{c-Ind}}_{F^{\times}
{\rm GL}_2(\mathcal{O}_F)}
^{{\rm GL}_2(F)}(\pi_w 
\otimes \chi)$.
Then, $\pi_{w,\chi}$ 
is an irreducible 
and cuspidal 
representation 
of ${\rm GL}_2(F)$
as in \cite[5.4]{AOPS}.
We are able to 
construct $\pi_{w,\chi}$ 
in another way.
Recall that we fixed 
the embedding 
$\Gamma \hookrightarrow G_2^F$ 
in (\ref{ed1}).
Let $U^n_{\mathfrak{M}}$ be a 
compact open normal 
subgroup $1_2+\pi^n
{\rm M}_2(\mathcal{O}_F) 
\subset {\rm GL}_2(\mathcal{O}_F).$
We inflate
 a character $w$ (\ref{cha}) of $\Gamma N$ to 
 $\mathcal{O}^{\times}_EU^1_{\mathfrak{M}} 
 \subset {\rm GL}_2(F)$, for which 
we write $\tilde{w}$.
Then, we extend $\tilde{w}$ 
to a character $\tilde{w}_{\chi}$ of 
$J^1_{\mathfrak{M}}:
=E^{\times}U^1_{\mathfrak{M}}$ by using $\chi.$
We consider 
$\pi^0_{w,\chi}:=
{\rm c}-{\rm Ind}_{J^1_{\mathfrak{M}}}
^{{\rm GL}_2(F)}\tilde{w}_{\chi}.$
We easily check that 
$\pi^0_{w,\chi}$ is isomorphic to
$\pi_{w,\chi}.$
The fact that the representation 
$\pi^0_{w,\chi} \simeq \pi_{w,\chi}$ 
is an irreducible
 and cuspidal representation 
is also verified 
by \cite[Theorem 15.3]{BH}.
It is not difficult to 
check that 
the representation $\pi_{w,\chi}$
contains an 
unramified simple stratum. 
This follows from 
Definition \ref{cus} 
almost immediately.
Hence, 
by \cite[Lemma 20.3]{BH}, 
the representation 
$\pi_{w,\chi}$ 
is an {\it unramified}
 irreducible cuspidal 
 representation 
in a sense of \cite[20.1]{BH}.
(i.e. there exists an unramified character $\phi \neq 1$
of $F^{\times}$ such that 
$\pi_{w,\chi} 
\phi \simeq \pi_{w,\chi}.$)

Secondly, we define a smooth representation
$\rho_{w,\chi}$ of 
$D^{\times}$ in the following.
We have the 
reduction map $\mathcal{O}^{\times}_D
\to \mathcal{O}^{\times}_3.$
If this map is restricted to 
a subgroup $\mathcal{O}^{\times}_EU_D^1$, 
it induces a surjection
$\mathcal{O}^{\times}_EU_D^1
 \to \mathcal{O}^{\times}_3$
First, we write $\rho_w$ 
for the inflation
to $\mathcal{O}_E^{\times}U^1_D$
of $\rho_w.$
Then, we extend $\rho_w$ to 
 a representation 
of $J_D^1:=E^{\times}U_D^1$
by using $\chi$, 
which we denote 
by $\rho_w \otimes \chi.$
Then, we define a $D^{\times}$-representation 
$\rho_{w,\chi}$
by 
\[
\rho_{w,\chi}:=
{\rm c-Ind}_{J_D^1}^{D^{\times}}(\rho_w \otimes \chi).
\]

Thirdly, we define a $2$-dimensional
 representation
  $\pi'_{w,\chi}$ of $W_F$
in the following.
Let $w_{\chi}$ be 
a character of $E^{\times}$ 
defined by 
$\chi$ and $w$.
Moreover, let $\Delta$ be 
the 
unramified character of $E^{\times}$
of order $2$ 
as in \cite[34.4]{BH}.
We write $v_F:W_F \to \mathbb{Z}$
for the canonical map taking 
a geometric Frobenius element to $1$.
We set $||x||:=q^{-v_F(x)}$ for $x \in W_F.$
We define $\pi'_{w,\chi}$
 by 
\[
\pi'_{w,\chi}:=||\cdot||^{-\frac{1}{2}}
{\rm Ind}_{E/F}((\Delta w_{\chi}) 
\circ \mathbf{a}_E).
\]
See \cite[p.219 and p.222]{BH} for this normalization.
%Let ${\rm JL}$ and ${\rm LL}$
%be the Jacquet-Langlands correspondence and 
%the local Langlands correspondence for ${\rm GL}_2$.
%See for example \cite[p.219 and p.334]{BH} 
%for the correspondences.

\paragraph{ramified case\ :}
We consider the $\mathbf{G}$-representation
 $W'$ in subsection \ref{acu2}.
Now, we recall the 
following isomorphism
\[
W' \simeq \bigoplus_{(w,a) \in \mathcal{T}_1}
 \pi_{w,a}^{\vee} \otimes \rho_{w,a}
\otimes \pi'_{w,a}
\]
as a $\mathbf{G}$-representation 
in Corollary \ref{dek2}.

As in the unramified case, we 
define representations of 
${\rm GL}_2(F),$ $D^{\times}$
 and $W_F$ one by one.
 We choose a character $\mathbb{Z} \to 
 \overline{\mathbb{Q}}^{\times}_l.$
 First, we define a 
 ramified cuspidal representation of 
 ${\rm GL}_2(F)$ of level $1/2$.
We inflate the character $\tilde{w}_2:G_2^F \to 
\overline{\mathbb{Q}}^{\times}_l$ 
to ${\rm GL}_2(\mathcal{O}_F)$, and extend it 
to a character of ${\rm GL}_2(F)$
by using $\chi,$ 
 which we denote by $w_{2,\chi}.$ 
We inflate the character $\Lambda_{w_1,a}:
\overline{\mathcal{O}}^{\times}_{E_1}
\overline{U}^1_{\mathfrak{U}} \to 
\overline{\mathbb{Q}}^{\times}_l$
to $\mathcal{O}^{\times}_{E_1}U^1_{\mathfrak{U}} \subset
 {\rm GL}_2(\mathcal{O}_F),$ which we denote 
 by the same letter $\Lambda_{w_1,a}.$
 Then, we extend it to a character 
 of $J^1_{\mathfrak{U}}
 :=E_1^{\times}U^1_{\mathfrak{U}}$
 by using $\chi$, 
 which we denote by $\Lambda_{w_1,a,\chi}.$
We set as follows
\[
\pi_{w_1,a,\chi}:={\rm c-Ind}
^{{\rm GL}_2(F)}
_{J^1_{\mathfrak{U}}}\Lambda_{w_1,a,\chi}.
\]
Recall the definition of 
the element 
$\alpha 
\in \mathfrak{U}$ in (\ref{aa1}).
Since $\frac{\alpha}{\pi} \in E_1$
is minimal in a sense of 
\cite[Definition 13.4]{BH}, 
a triple $(\mathfrak{U},1,\frac{\alpha}{\pi})$
is a ramified simple stratum by 
\cite[Proposition 13.5]{BH}.
See \cite[p.96]{BH} for a definition of 
ramified simple stratum.
Since $\pi_{w_1,a,\chi}$ contains 
$(\mathfrak{U},1,\frac{\alpha}{\pi}),$
 $\pi_{w_1,a,\chi}$ is an 
irreducible and cuspidal representation 
of ${\rm GL}_2(F)$ of level $1/2$.
Furthermore, $\pi_{w_1,a,\chi}$ is minimal i.e.
 $l(\pi_{w_1,a,\chi}) \leq l(\pi_{w_1,a,\chi}\phi')$
 for any character $\phi'$ of $F^{\times}$,
 where $l(\cdot)$ means the normalized level.
We have $l({w}_{2,\chi}
\pi_{w_1,a,\chi})=1/2.$

Secondly, we define a smooth 
$D^{\times}$-representation.
Let $\chi'$ denote the character of $\mathbb{Z}$
defined by $n \mapsto (-1)^n \chi(n).$
We inflate the character 
$\tilde{w}^D_2:\mathcal{O}_3^{\times} \to 
\overline{\mathbb{Q}}^{\times}_l$
to $\mathcal{O}^{\times}_D$, and extend it to
a character of $D^{\times}$ by using $\chi$,
 which we denote by $w^D_{2,\chi}.$
We inflate a character 
$\Lambda_{w_1,a}^D:
\overline{\mathcal{O}}^{\times}_{E_1}
\overline{U}^1_{D} \to 
\overline{\mathbb{Q}}^{\times}_l$
to ${\mathcal{O}}^{\times}_{E_1}
{U}^1_{D}$, which we denote by 
the same letter 
$\Lambda_{w_1,a}^D.$
Then, we extend it to a character of
 a group $J'^1_D:=E_1^{\times}U_D^1 
 \subset D^{\times}$
 by using $\chi'$, which we denote by 
 $\Lambda_{w_1,a,\chi}^D.$ See \cite[56.5]{BH}.
Now, we define 
\[
\rho_{w_1,a,\chi}:=c-{\rm Ind}^{D^{\times}}
_{J'^1_{D}}\Lambda_{w_1,a,\chi}^D.
\]

Thirdly, we define a $2$-dimensional 
$W_F$-representation.
We inflate the character 
$w_2:(\mathcal{O}_F/\pi^2)^{\times} \to 
\overline{\mathbb{Q}}^{\times}_l$ 
to $\mathcal{O}^{\times}_F$,
and extend it to a character $F^{\times}$
by using $\chi$, which we denote by 
$w'_{2,\chi}$.
We inflate a character 
$\Lambda'_{w_1,a}:(\mathcal{O}^{\times}_{E_1}/\pi)^{\times} 
\to 
\overline{\mathbb{Q}}^{\times}_l$
to $\mathcal{O}^{\times}_{E_1},$ and extend it to
a character $E_1^{\times}$ by using $\chi$, which we denote by
$\Lambda'_{w_1,a,\chi}.$
By the local class field theory, 
we obtain the character 
$\Lambda'_{w_1,a,\chi} \circ \mathbf{a}_{E_1}:
W^{\rm ab}_{E_1} \simeq E_1^{\times}
 \to \overline{\mathbb{Q}}^{\times}_l.$
For a pair $(E_1/F,\Lambda'_{w_1,a,\chi}),$
there exists a character $\Delta$ of
  $E^{\times}_1$ of level zero, which is defined 
  in \cite[34.4]{BH}.
We define 
\[
\pi'_{w_1,a,\chi}:=
||\cdot||^{-\frac{1}{2}}{\rm Ind}_{E_1/F}
((\Delta\Lambda'_{w_1,a,\chi}) \circ \mathbf{a}_{E_1}).
\]
Then, we have the following main theorem in this paper.
\begin{theorem}
Let the notation be as above.
\\1.\ Then, we have the following,
 for the unramified case, 
\[
\rho_{w,\chi}={\rm JL}(\pi_{w,\chi}),\ 
\pi_{w,\chi}=
{\rm LL}_{\ell}(\pi'_{w,\chi}).
\]
\\2.\ We consider the ramified case.
We set
\[
\pi_{w,a,\chi}:=w_{2,\chi}\pi_{w_1,a,\chi},\ 
\rho_{w,a,\chi}:=w^{D}_{2,\chi}\rho_{w_1,a,\chi},\ 
\pi'_{w,a,\chi}:=w'_{2,\chi}\pi'_{w_1,a,\chi}.
\]
Then, we have the following
\[
\rho_{w,a,\chi}
={\rm JL}(\pi_{w,a,\chi}),\ 
\pi_{w,a}={\rm LL}_{\ell}(\pi'_{w,a,\chi}).
\]
\end{theorem}
\begin{proof}
We prove the assertion $1.$
The first (resp.\ second ) equality follows 
from the description of ${\rm JL}$
(resp.\ ${\rm LL}_{\ell}$)
given in \cite[56.6]{BH} 
and Corollary \ref{lap} 
(resp.\ \cite[p.219 and p.223]{BH}).
The required 
assertion $2$ is proved 
by Lemma \ref{keyr} and 
\cite[34,35]{BH} for ${\rm LL}_{\ell}$, and 
\cite[56.5]{BH} for ${\rm JL}.$
\end{proof}

\noindent
Tetsushi Ito\\ 
Department of Mathematics, Faculty of Science,
Kyoto University, Kyoto 606-8502, Japan\\ 
tetsushi@math.kyoto-u.ac.jp\\ 

\noindent
Yoichi Mieda\\ 
Faculty of Mathematics,
Kyushu University, 744 Motooka, Nishi-ku, 
Fukuoka city, 
Fukuoka 819-0395, Japan\\ 
mieda@math.kyushu-u.ac.jp\\

\noindent
Takahiro Tsushima\\ 
Faculty of Mathematics,
Kyushu University, 744 Motooka, Nishi-ku, 
Fukuoka city, 
Fukuoka 819-0395, Japan\\
tsushima@math.kyushu-u.ac.jp\\

\end{document}